\documentclass[a4paper,oneside,11pt]{article}

\usepackage{amsmath,amsfonts,amscd,amssymb}
\usepackage{longtable,geometry}
\usepackage[english]{babel}
\usepackage[utf8]{inputenc}
\usepackage[active]{srcltx}
\usepackage[T1]{fontenc}
\usepackage{graphicx}
\usepackage{enumitem}
\usepackage{bbm}
\usepackage{enumerate}   
\usepackage{mathtools}
\usepackage{hyperref}\hypersetup{colorlinks=false, pdfborder={0 0 0},}

\usepackage{MnSymbol}
\usepackage{stmaryrd}
\usepackage{nicefrac}

\usepackage{rotating}

\usepackage{xcolor}
\usepackage{framed}

\colorlet{shadecolor}{blue!15}

\newtheorem{theorem}{Theorem}[section]

\newtheorem{corollary}[theorem]{Corollary}
\newtheorem{lemma}[theorem]{Lemma}
\newtheorem{proposition}[theorem]{Proposition}
\newtheorem{definition}[theorem]{Definition}

\newtheorem{remark}[theorem]{Remark}

\newenvironment{proof}[1][\relax]%s
  {\paragraph{Proof\ifx#1\relax\else~of #1\fi}}%
  {~\hfill$\square$\par\bigskip}

\newcommand{\calA}{\mathcal{A}}
\newcommand{\calB}{\mathcal{B}}
\newcommand{\calC}{\mathcal{C}}
\newcommand{\calD}{\mathcal{D}}
\newcommand{\calE}{\mathcal{E}}

\newcommand{\calG}{\mathcal{G}}
\newcommand{\calH}{\mathcal{H}}
\newcommand{\calI}{\mathcal{I}}

\newcommand{\calL}{\mathcal{L}}

\newcommand{\calO}{\mathcal{O}}

\newcommand{\calS}{\mathcal{S}}

\newcommand{\calU}{\mathcal{U}}
\newcommand{\calV}{\mathcal{V}}
\newcommand{\calW}{\mathcal{W}}
\newcommand{\calX}{\mathcal{X}}
\newcommand{\calY}{\mathcal{Y}}

\newcommand{\bbE}{\mathbb{E}}

\newcommand{\bbG}{\mathbb{G}}

\newcommand{\bbI}{\mathbb{I}}

\newcommand{\bbN}{\mathbb{N}}

\newcommand{\bbP}{\mathbb{P}}
\newcommand{\bbQ}{\mathbb{Q}}
\newcommand{\bbR}{\mathbb{R}}

\newcommand{\bbT}{\mathbb{T}}

\newcommand{\bbV}{\mathbb{V}}

\newcommand{\bbZ}{\mathbb{Z}}
\newcommand{\sfP}{{\sf P}}
\newcommand{\sfE}{{\sf E}}

\definecolor{speccol}{rgb}{0.7,0.1,0.5}

\definecolor{speccol2}{rgb}{0.237,0.145,0.033}

\newcommand{\sign}{\mathrm{sign}}
\newcommand{\Var}{\mathrm{Var}}
\newcommand{\Ber}{\mathsf{Ber}}

\newcommand{\Zodd}{\bbZ_{\mathrm{odd}}}

\newcommand{\heightfcns}{\mathcal{H}}
\newcommand{\Strip}{{\rm Strip}}

\makeatletter
\newcommand{\mylabel}[3]{#3\def\@currentlabel{#2}\label{#1}}
\makeatother

\title{Logarithmic delocalization of random Lipschitz functions on honeycomb and other lattices}
\author{
Alex M. Karrila\thanks{
\AA bo Akademi Matematik, Henriksgatan 2, 20500 \AA bo, Finland. E-mail: \texttt{alex.karrila@abo.fi; alex.karrila@gmail.com}
}
}
\date{}

\begin{document}

\maketitle

\begin{abstract}
We study random one-Lipschitz integer functions $f$ on the vertices of a finite connected graph, sampled according to the weight $W(f) = \prod_{\langle v, w \rangle \in E} \mathbf{c}^{ \mathbb{I} \{ f(v) = f(w) \} }$ where $\mathbf{c} \geq 1$, and restricted by a boundary condition. For planar graphs, this is arguably the simplest ``2D random walk model'', and proving the convergence of such models to the Gaussian free field (GFF) is a major open question.
Our main result is that for subgraphs of the honeycomb lattice (and some other cubic planar lattices), with flat boundary conditions and $1 \leq \mathbf{ c } \leq 2$, such functions exhibit logarithmic variations. This is in line with the GFF prediction and improves a non-quantitative delocalization result by P.~Lammers. The proof goes via level-set percolation arguments, including a renormalization inequality and a dichotomy theorem for level-set loops. In another direction, we show that random Lipschitz functions have bounded variance whenever the wired FK-Ising model with $p=1-1/\mathbf{c}$ percolates on the same lattice (corresponding to $\mathbf{c} > 2 + \sqrt{3}$ on the honeycomb lattice). Via a simple coupling, this also implies, perhaps surprisingly, that random homomorphisms are localized on the rhombille lattice.
\end{abstract}

\tableofcontents

\section{Introduction}

\paragraph{Background and main results}

\begin{figure}
\begin{center}
\includegraphics[width=0.5\textwidth]{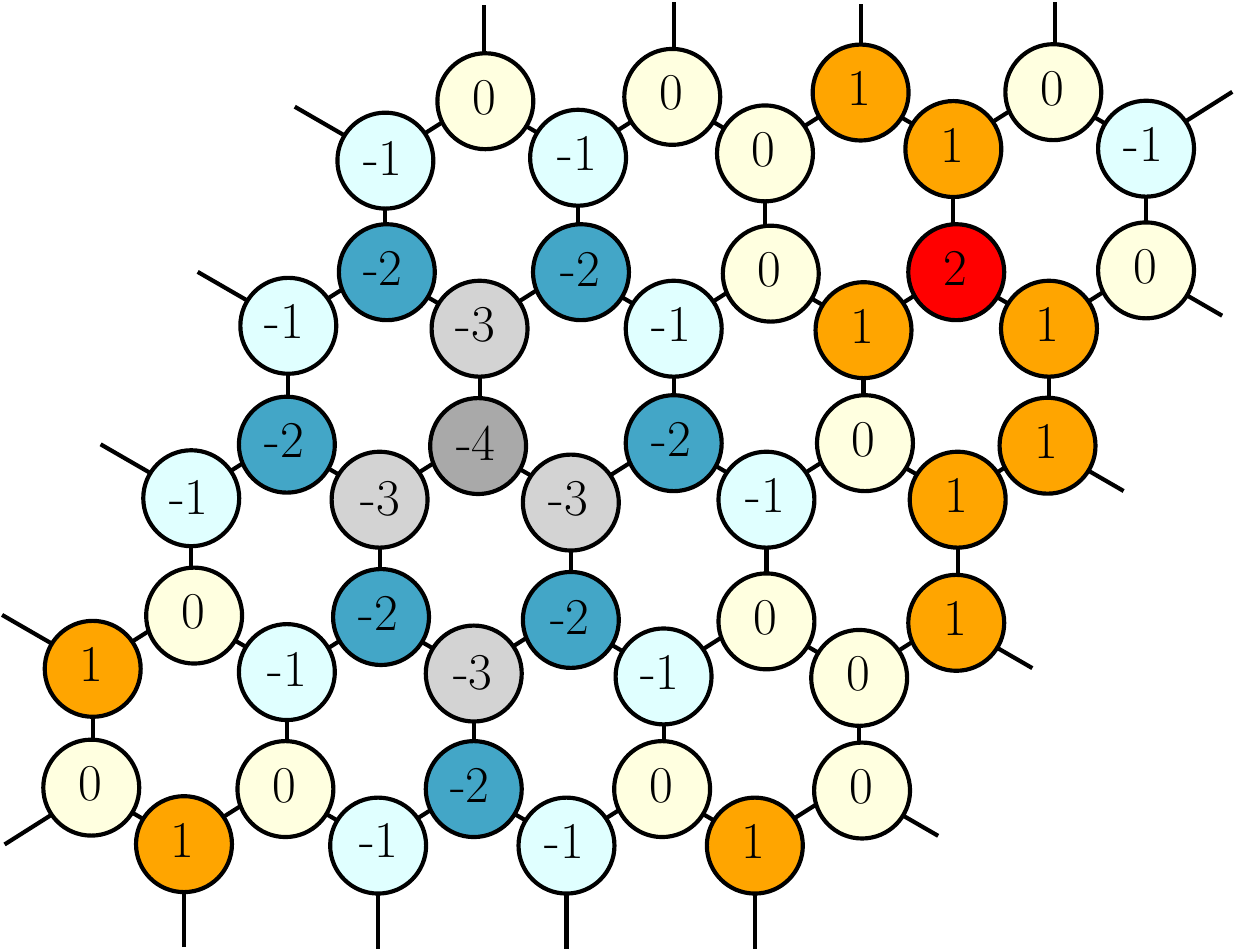}
\end{center}
\caption{
\label{fig:simple Lip fcn}
An example of a $1$-Lipschitz integer-valued function on the vertices of the honeycomb lattice, pinned to zero outside of the depicted vertices.
}
\end{figure}

Random fields on lattices and their scaling limits play a central role in statistical mechanics.
On the one-dimensional lattice $\bbZ$, such scaling limits are rather universally established to be Gaussian processes such as the Brownian motion or the Brownian bridge, by Donsker's theorem and various generalizations.
In two-dimensions, the situation is starkly different: first off, many parametric models then (conjecturally) exhibit a phase transition between a \textit{localized} (or \textit{smooth}) and a \textit{delocalized} (or \textit{rough}) phase.
The delocalized models are then often expected to converge to an appropriate Gaussian random field, but proving rigorously such convergence remains out of reach but for very few special models.
The localized phase, in turn, can be interpreted as an analogue of long-range order in the Ising model, and should exhibit a completely different (and less rich) limit behaviour. It is thus of interest to identify the phase of a given two-dimensional model and, if delocalized, prove indications of its relation to a Gaussian (free) field.

This paper studies random Lipschitz functions, arguably the simplest random field model. A function $f: \bbV \to \bbZ$ on the vertices $\bbV$ of a connected graph $\bbG = (\bbV, \bbE)$ is \textit{(one-)Lipschitz} if
\begin{align*}
(f(v) - f(w)) \in \{ -1, 0, 1\} \qquad \text{for all adjacent vertices } v, w,
\end{align*}
and to such functions we associate a weight
\begin{align*}
W(f) = \prod_{e= \langle v, w \rangle \in \bbE} \mathbf{c}^{ \mathbb{I} \{ f(v) = f(w) \} },
\end{align*}
where $\mathbf{c} > 0$ is the model parameter. For a finite vertex set $D\subsetneq \bbV$, let $\sfP_D$ then denote the probability measure supported on Lipschitz functions $f: \bbV \to \bbZ$ with $f_{|D^c} \equiv 0$ and probabilities proportional to $W(f)$ (Figure~\ref{fig:simple Lip fcn}).\footnote{Note that on the linear graph $\{0, 1, \ldots, T \}$ with the boundary condition $f(0)=0$, this is an elementary random walk model --- even simpler than the simple random walk which exhibits a parity effect.} Our main result is the following phase characterization.

\begin{theorem}[Corollary~\ref{cor:localized regimes} and Theorem~\ref{thm:main thm 2nd}; special cases]
\label{thm:main thm intro statement}
Let $\bbG$ be the infinite two-dimensional hexagonal lattice and $d_\bbG$ its graph distance. For $1 \leq \mathbf{c} \leq 2$, the random Lipschitz functions on $\bbG$ are \emph{logarithmically delocalized}, in the sense that there exist $C=C(\mathbf{c}), c=c(\mathbf{c}) > 0$ such that for any finite $D \subset \bbV$ and $x \in D$ with $d_\bbG(x, D^c) \geq 2$,
\begin{align}
\label{eq:log deloc intro}
c  \log d_\bbG (x, D^c) \leq \Var_D (f(x)) \leq C \log d_\bbG (x, D^c),
\end{align}
where $\Var_D$ denotes variance under $\sfP_D$.
For $\mathbf{c} > 2 + \sqrt{3}$, the same model \emph{localized} in the sense that there exists $C=C(\mathbf{c}) > 0$ such that for any finite $D \subset \bbV$ and $x \in D$
\begin{align}
\label{eq:loc intro}
\Var_D (f(x))  \leq C.
\end{align}
\end{theorem}

We obtain similar variance bounds on some other lattices, most interestingly the analogue of Equation~\eqref{eq:log deloc intro} on the square-octagon lattice for $1 \leq \mathbf{c} \leq 2$ and of the analogue of~\eqref{eq:loc intro}, as well as an exponential decay (Proposition~\ref{prop:exp decay in loc phase}), on the square lattice for $\mathbf{c} > 1 + \sqrt{2}$.

With the zero boundary condition above, the random Lipschitz functions should converge to the Gaussian free field (GFF). This conjecture is supported by the here proven logarithmic variance, as well as the scale-invariance arguments that are central in the proof (e.g., Proposition~\ref{thm:renorm}). A full GFF limit identification would, however, following priorly known techniques require exact solvability that is not (expected to be) present in the model. The connection to the GFF is illustrated more by the following corollary, stating informally speaking that the gradients of random Lipschitz functions converge to a log-correlated infinite-volume limit field. (For several definitions, see Section~\ref{subsec:grad Gibbs meas}.)

\begin{corollary}[Corollary~\ref{cor:log var for inf vol lim} and Proposition~\ref{prop:infinite vol limit}; special case]
\label{thm:conseq of main thm intro statement}
Let $\bbG$ be the hexagonal lattice and $1 \leq \mathbf{c} \leq 2$. The discrete gradients of a function under $\sfP_D$ converge weakly (in the topology of local convergence), as $D$ grows to $\bbG$, to a random Lipschitz gradient Gibbs measure $\sfP$. The measure $\sfP$ is invariant under the symmetries and translations of the hexagonal lattice and there exist $C, c > 0$ such that for any $x, y \in \bbV$,
\begin{align*}
\sfE[f(y)-f(x)] = 0 \quad
\text{and} \quad
c  \log d_\bbG (x, y) \leq \sfE[(f(y)-f(x))^2] \leq C \log d_\bbG (x, y),
\end{align*}
where $(f(y)-f(x))$ denotes the sum of the gradients on a(ny) path from $x$ to $y$.
\end{corollary}

Another consequence of our main results are phase identifications in other random models coupled to random Lipschitz functions; see Appendix~\ref{app:Other models}. Let us highlight the following result, which appears to be a new observation, is at least to the author somewhat surprising, and whose statement and proof (Proposition~\ref{prop:star-triangle coupling} and Corollary~\ref{cor:localized regimes} or~\cite{GM-loc}) are both strikingly simple. (See~\cite{CPT18} for several equivalent definitions of \textit{localization} in this model.)

\begin{theorem}
Uniform random homomorphisms on the rhombille lattice are localized.
% in the sense of bounded variance.
\end{theorem}

\paragraph{Related literature}

The closest prior work and the starting point of the present paper is~\cite{Piet-deloc}, which proves (as a special case) delocalization in the sense of non-existence of certain natural Gibbs measures for a wide range of lattice models, including random Lipschitz functions with $1 \leq \mathbf{c} \leq 2$, on any bi-periodic planar lattice of maximum degree $3$. This fairly soon implies divergence of variances (Corollary~\ref{cor:non-quant deloc}), while our main result on the logarithmic, scale-invariant and GFF-like behaviour takes quite a bit more. The improvement from divergent to logarithmic variance can  be contrasted, informally speaking, to proving absence of long-range order vs. finding the decay rate of correlations in a spin model.

As regards analogous but not directly related results, random Lipschitz functions have mostly been studied on the triangular lattice, logarithmic delocalization being proven in the uniform ($\mathbf{c}=1$) case~\cite{GM}, for $\mathbf{c} = \sqrt{2}$ (\cite{DGPS}, special case), and extended during the writing of this work to $1 \leq \mathbf{c} \leq \sqrt{2}$, in~\cite{GL-Lip}. Localization of the same model has been proven for $\mathbf{c} > \sqrt{3} - \epsilon$ (\cite{GM-loc}, special case).\footnote{Corollary~\ref{cor:localized regimes} of the present paper provides an alternative proof of localization for $\mathbf{c} > \sqrt{3} $ on the triangular lattice.} These results rely crucially on the coupling of triangular-lattice random Lipschitz functions to the loop $O(n) $ model while, in contrast, the present work studies the random Lipschitz functions directly and over a (limited) range of different lattices. Random Lipschitz functions on other than planar lattices have been recently studied at least in~\cite{PSY-Lip_on_tree, PSY-Lip_on_expander, Peled-highD-Lip}.

Among other planar random field models where logarithmic delocalization is proven, we name the dimer height functions in the liquid phase~\cite{KOS}, the random homomorphism model on the honeycomb lattice~(\cite{DGPS} and Appendix~\ref{app:Other models}) and the square lattice~\cite{Hom}, various other parameter regimes of the six-vertex model~\cite{KOS, FerSpo06, GP19, DCKMO, GL-Lip}, and a family of models interpolating between the solid-on-solid and integer Gaussian field models~\cite{Piet-dich-th}. 
Delocalization results with a divergent, but not necessarily logarithmic, variance for various models can be found at least in~\cite{CPT18, Lis, PO21}.%
\footnote{
Additionally, combining~\cite[Theorem~2.7]{Piet-deloc} and~\cite[Proof of Theorem~1.1]{CPT18} implies the result for the random homomorphism model on a variety of lattices, and~\cite{CGHP20} proves an analogous result on the loop $O(n)$ model.
}
Localization results, in turn, can be found at least in~\cite{DGHMT, RS19, GP19} for various random field models; see also Appendix~\ref{app:Other models}. We conclude this literature overview by remarking that a full proof of a Gaussian limit only appears to be known for uniform dimers~\cite{Ken01, Ken08} and interacting dimers~\cite{GMT17} (and, trivially, the discrete Gaussian free field).

\paragraph{Some core ideas and novelties}

On a very large perspective, the proof of our main result follows the renormalization strategy introduced in~\cite{DST} in the context of FK-percolation and extended (along with many other percolation techniques) for a random field model in~\cite{Hom}; also many of the above mentioned references employ similar arguments. The core of the proof is the so-called renormalization inequality (Proposition~\ref{thm:renorm}), comparing level-line loop probabilities in different length scales. This induces a dichotomy in the behaviour of these probabilities: either they remain uniformly positive over different length scales, or decay rapidly in the scale (Theorem~\ref{thm:loop dichotomy}). These two behaviour modes are then connected to logarithmic and bounded variance, respectively, and~\cite{Piet-deloc} excludes the latter, thus concluding the proof. 

The first novelty needed for this strategy is the implementation of basic random field tools for the (non-$O(n)$) random Lipschitz model. In particular, in contrast to many other earlier treated models, random Lipschitz functions $f$ \textit{do not} satisfy absolute-value FKG; however, $2f+1$ do, and the loss of information in this affine map (the often-studied minimal absolute-value $|2f(x) + 1|=1$ only reveals $f(x) \in \{ 0, -1 \}$) is compensated for by coupling to the absolute values a sign Ising and sign FK-Ising model (Sections~\ref{subsubsec: Ising tools}--\ref{subsec:sign FK model}, Proposition~\ref{prop:pos assoc of percolated height fcns}).
Percolation models coupled to other random fields date at least back to Edwards and Sokal~\cite{Edw-Sok}, and their use has flourished recently~\cite{GM, Lis, DCKMO, Piet-deloc, Piet-dich-th, GL-Lip}, and ours can be interpreted via the general cluster swapping arguments in~\cite{She05}. The direct connection to the FK model hopefully makes our percolated model especially intuitive and illustrative of these ideas.

As the renormalization argument originates in percolation theory, core steps towards its proof are a Russo--Seymour--Welsh (RSW) type theorem (Theorem~\ref{thm:RSW}) and a ``boundary pushing theorem'' (Theorem~\ref{prop:push}) in terms of level sets of random Lipschitz functions. Perhaps of interest for other models is the idea that, essentially due to the weakness of our absolute-value FKG, the proof of the RSW theorem \textit{uses} the boundary pushing theorem: one first crucially uses a shift-invariant strip geometry in the RSW, and then pushes the boundary to compare to a finite domain.

The renormalization inequality (Proposition~\ref{thm:renorm}) in stated and proven in a two-alternative form which is logically equivalent to the usual inequality (Remark~\ref{rem:equiv renorm ineqs}). This allows one to only study level sets of only one given height of the proof (of the non-trivial alternative), making the proof very similar to percolation models and simplifying it compared to, e.g., the analogous proof in~\cite{Hom}.

The proof of Corollary~\ref{thm:conseq of main thm intro statement}/Proposition~\ref{prop:infinite vol limit} contains a mixing lemma (Corollary~\ref{cor:mixing}) which may be of independent interest.

\paragraph{Acknowledgements} 
An unpublished proof of the analogue of Proposition~\ref{prop:infinite vol limit} in the context of the six-vertex model was produced by the author and Piet Lammers in 2020. The author thanks Hugo Duminil-Copin and Piet Lammers for discussions and correspondence during the present project. The Academy of Finland (grant \#339515) is gratefully acknowledged for financial support.

\part{General combinatorial aspects}

In this part of the work, we study random Lipschitz functions on a general (not necessarily planar) graphs.

\section{The random models and their couplings}

We define the random models of interest; some conventions in the rest of the article will differ from those in the Introduction.

\subsection{The random Lipschitz model}

Let $D=(V, E)$ be a finite connected graph. A function $h: V \to \bbZ$ is \textit{$L$-Lipschitz} if 
\begin{align*}
|h(v) - h(w)| \leq L \qquad \text{for all adjacent } v, w \in V.
\end{align*}
It will turn more convenient to study $1$-Lipschitz functions $f$ via $h:= 2f + 1$. We denote $2\bbZ + 1 =: \Zodd$ and call odd $2$-Lipschitz functions $h: V \to \Zodd$ \textit{height functions} (on $D$). The set of all height functions on $D$ is denoted by $\heightfcns_D$. 

Let $\Delta \subset V$, $\Delta \neq \emptyset$ be a specified set of vertices. We call a function $\xi: \Delta \to \Zodd$ \textit{a(n admissible) boundary condition (on $\Delta$)} if it can be extended to a height function on $D$. A necessary and sufficient criterion for this is that
\begin{align*}
| \xi(v) - \xi (w) | \leq 2 d_D (v, w) \qquad \text{for all } v, w \in \Delta,
\end{align*}
where $d_D$ is the graph distance. The necessity is clear, while the sufficiency is proven by observing that
\begin{align}
\label{eq:max and min height fcns}
\underline{h}(x) = \max_{v \in \Delta} \left( \xi(v) - 2d_D(v, x) \right)
\qquad \text{and} \qquad
\overline{h}(x) = \min_{v \in \Delta} \left( \xi(v) + 2d_D(v, x) \right)
\end{align}
are the minimal and maximal height functions extending $\xi$, respectively. For $\Delta' \supset \Delta$, $\underline{h}$ and $\overline{h}$ also define the \textit{minimal} and \textit{maximal (admissible) boundary conditions extending} $\xi$ to $\Delta'$. 

We will also crucially need \textit{(admissible) set-valued boundary conditions}, i.e., maps $\xi$ from $\Delta$ to finite subsets of $\Zodd$, such that there exists $\chi: \Delta \to \Zodd$ with $\chi(x) \in \xi(x)$ for all $x \in \Delta$ which is an admissible boundary condition. Note that the pointwise minimum (resp. maximum) of such $\chi$ is also admissible, defining the minimal single-valued $\chi \in \xi$, whose corresponding $\underline{h}$ in turn yields the minimal admissible extension of $\xi$ to $\Delta' \supset \Delta$.

% \paragraph{The measure}

%We now study random odd $2$-Lipschitz functions on $G$ with a boundary condition $\xi$, denoted $\bbP^\xi$.

Allowing a bit different setting than in the introduction, we impose weights $\mathbf{c}_e \geq 1$ that may depend on the edge $e \in E$, and define the \textit{weight} of a height function $h$ as
\begin{align*}
W(h) = \prod_{e= \langle v, w \rangle \in E} \mathbf{c}_e^{ \mathbb{I} \{ h(v) = h(w) \} },
\end{align*}
i.e., the weight favours ``flat'' functions $h$. 
Then, with an admissible  (possibly set-valued) boundary condition $\xi$, the \textit{random Lipschitz} or \textit{random height function probability measure} is
\begin{align*}
\bbP^\xi_D [\{ h \} ] = \mathbb{I} \{ h(x) \in \xi(x) \text{ for all } x \in \Delta \} W(h)/Z^{\xi}_D,
\end{align*}
where $Z^\xi_D$ is the appropriate normalizing factor. We will omit the subscripts $D$  from $\bbP^\xi_D$ if there is no ambiguity.
% We will later denote $\bbP^\xi [\{ h \} ] = \bbP^\xi [ h ]$.

\subsection{The sign Ising model}
\label{subsubsec: Ising tools}

%The reason that we rather study an odd $2$-Lipschitz function $h=2f +1$ than a one-Lipschitz function $f$ is its fruitful connection to the Ising model.
We define the \textit{(ferromagnetic) Ising model} on a finite graph $\calD =(\calV, E)$ with weights $\mathbf{c}_e \geq 1$ by setting, for any $\sigma\in\{\pm1\}^{ \calV}$,
\begin{align*}
W_{\mathrm{Ising}}^{\calD} (\sigma) &:= \prod_{e = \langle u, v \rangle \in E} \mathbf{c}_e^{\bbI \{ \sigma(u) = \sigma(v) \} }, \\
\sfP_{\mathrm{Ising}}^{\calD} [ \sigma ] &:= \tfrac{1}{Z} W_{\mathrm{Ising}}^{\calD} (\sigma).
\end{align*}
For $S \subset \calV$, we define the Ising model with \textit{$+$ boundary conditions} on $S$ via
\begin{align*}
\sfP_{\mathrm{Ising}}^{\calD,+} [ \sigma ] &:= \tfrac{1}{Z^+} \mathbbm{1}\{ \sigma_{|S} = +1\} W_{\mathrm{Ising}}^{\calD} (\sigma).
\end{align*}
%We remind that throughout this paper, $\mathbf{c}_e \geq 1$ for all $e \in E$, i.e. the Ising model is ferromagnetic.

Let now $D=(V, E)$ be a finite connected graph and $H \in \heightfcns_D$ a \textit{positive} height function on $D$ (equivalently, $H=|h|$ for some $h \in \heightfcns_D$). The \textit{fixed-sign edges} $E_{\mathsf{fix}} =E_{\mathsf{fix}}(H)$ of $H$ are those $e=\langle v, w \rangle \in E$ for which $\max \{ H(v), H(w) \} \geq 3$; indeed $\sign (h)$ is then fixed along such edges whenever $|h|=H$. Define the graph $\calD = \calD(H)= (\calV, \calE \cong E) = G/\sim$, where $\sim$ is the equivalence relation on $V$ of being $E_{\mathsf{fix}}$-connected, $\calV = V / \sim$, and we directly identify $\calE$ with $E$. Encode the sign of $h$ with $|h|=H$ in the function a function $\sigma_{h} : \calV \to \{\pm1\}$. The mapping $h \mapsto \sigma_h$ is then a bijection between height functions $h $ with $|h|=H$ and sign functions $\sigma : \calV \to \{\pm1\}$. The reader may now easily verify the following.

\begin{lemma}
\label{cor:signs are Ising}
Let $D=(V, E)$ be a finite connected graph, $H \in \heightfcns_D$ a \emph{positive} height function on $D$ and $\calD = \calD(H)$. There exists $K=K(D,H) > 0$ such that
\begin{align*}
W (h) = K \times W_{ \mathrm{Ising}}^{\calD} (\sigma_{h})
\qquad \text{for all $h \in \heightfcns_D$ with $|h|=H$.}
\end{align*}
In particular, if $h$ is drawn randomly among height functions with $|h|=H$, with probabilities proportional to $W(h)$, then $\sigma_h \sim \sfP_{\mathrm{Ising}}^{\calD } $. If it is additionally known that $\sign(h)=+1$ on $S \subset \calV$, then $\sigma_h \sim \sfP_{\mathrm{Ising}}^{\calD,+} $.
\end{lemma}

\subsection{The sign FK-Ising model}
\label{subsec:sign FK model}

Recall that the \textit{FK-Ising model} on a finite graph $\calD = (\calV, E)$ with edge parameters $p_e \in [0,1]$ is a random subset $\varpi \subset E$ of edges, drawn with the probabilities
\begin{align*}
\sfP^{\calD}_{\mathrm{FK}} [\varpi] = \frac{1}{Z} 2^{c(\varpi ) } \prod_{e \in E }  \left( (p_e)^{\mathbbm{1}\{ e \in \varpi \}} (1-p_e)^{\mathbbm{1}\{ e \not \in \varpi \}} \right),
\end{align*}
where $c(\varpi)$ is the number of connected components in the graph $(\calV, \varpi) $ and $Z$ is the appropriate normalization factor. For a given $S \subset \calV$, we define the FK-Ising model $\sfP^{\calD,+}_{\mathrm{FK}}$ with \textit{wired boundary conditions} on $S$ as the above model on the graph where the vertices on $S$ are identified. We assume familiarity the basic properties of the FK-Ising model such as the Spatial Markov property, monotonicity in $p_e$, positive association (FKG), and comparison of boundary conditions; see, e.g., \cite{Dum17}. 

For notational simplicity, we will liberally identify random edge sets with their indicator functions $E \to \{ 0,1 \}$. Let us also define the \textit{Bernoulli $p_e$ edge percolation} as the random $B \subset E$, with $B_e$ being independent Bernoulli $p_e$ variables (i.e., each edge $e$ is present with probability $p_e$ and absent with probability $1-p_e$, independently of the others). For the rest of this paper, we set
\begin{align*}
p_e = 1-1/\mathbf{c}_e \in [0,1],
\end{align*}
motivated by the standard coupling of the Ising and FK-Ising models in our conventions:

\begin{lemma}
\label{lem:FK coupling}
Let  $\calD = (\calV, E)$ be a finite graph, and $\mathbf{c}_e \geq 1$ and $p_e = 1 - 1/\mathbf{c}_e$ for every $e\in E$. There is a coupling of $\sigma \sim \sfP_{\mathrm{Ising}}^\calD$ (resp. $\sigma \sim \sfP^{\calD,+}_{\mathrm{Ising}}$) and  $\varpi \sim \sfP^{\calD}_{\mathrm{FK}}$ (resp. $\varpi \sim \sfP^{\calD,+}_{\mathrm{FK}}$) satisfying, and determined by, the following conditional laws: \\
1) given $\sigma$, the conditional law of $\varpi$ is 
\begin{align*}
\varpi_{\langle v, w \rangle} = 
\begin{cases}
& 0, \qquad \sigma_v \neq \sigma_w \\
& B_{\langle v, w \rangle}, \qquad \sigma_v = \sigma_w,
\end{cases}
\end{align*}
where $B$ is the Bernoulli $p_e$ edge percolation; and \\
2) given $\varpi $, the conditional law of $\sigma$ is constant on each connected component of $\varpi$, and sampled by independent fair coin flips on the different components (resp. except set to $+1$ on those intersecting $S$).
\end{lemma}

Due to our choice of $\calD$ in the previous subsection, let us also explicate the contraction property of the FK-Ising model. See Appendix~\ref{app:roskis} for the (easy) proof. % The following basic property is perhaps less common in literature but very useful for us.

\begin{lemma}
\label{lem:FK cluster conditioning}
Let $D=(V, E)$ be a finite graph, $E_{\mathsf{fix}} \subset E$ and $\calD = D/\sim = (\calV, E)$, where $\sim$ is the equivalence relation on $V$ of being $E_{\mathsf{fix}}$-connected and $\calV = V / \sim$. There is a coupling of $ \omega \sim \sfP^{D}_{\mathrm{FK}} [ \;\cdot \; | \; E_{\mathsf{fix}} \subset \omega]$ (resp. $ \omega \sim \sfP^{D,+}_{\mathrm{FK}} [ \;\cdot \; | \; E_{\mathsf{fix}} \subset \omega]$ with given $S \subset V$) and $\varpi \sim \sfP^{\calD}_{\mathrm{FK}} $ (resp. $\varpi \sim \sfP^{\calD,+}_{\mathrm{FK}} $ with the same $S$ interpreted on $\calV$),
satisfying, and determined by, the following conditional laws:
1) given $\omega$, the conditional law of $\varpi$ is
\begin{align*}
\varpi_e =
\begin{cases}
& \omega_e, \qquad e \not \in E_{\mathsf{fix}}\\
& B_e, \qquad e \in E_{\mathsf{fix}},
\end{cases}
\end{align*}
where $B$ is the Bernoulli $p_e$ edge percolation; and
2) $\omega = \varpi \cup E_{\mathsf{fix}}$.
\end{lemma}

Note that by the FKG, for instance, $\sfP^{D}_{\mathrm{FK}} [ \;\cdot \; | \; E_{\mathsf{fix}} \subset \omega]$ dominates stochastically $\sfP^{D}_{\mathrm{FK}} [ \;\cdot \; ]$, which will be very useful. We now define the main random object in this paper.

\begin{definition}[The triplet $(h, \omega, B)$]
Given $D=(V, E)$ a finite connected graph, $\xi$ an admissible (possibly set-valued) boundary condition on $\Delta \neq \emptyset$, and $h \sim \bbP^\xi$, we always couple to $h$ an independent Bernoulli $p_e$ percolation $B$ and another edge percolation $\omega$ defined by
\begin{align*}
\omega_{\langle v, w \rangle}
=
\begin{cases}
&1, \qquad \langle v, w \rangle \in E_{\mathsf{fix}}(|h|), \text{ i.e., } \max \{ |h(v)|, |h(w)| \} \geq 3\\
&B_e, \qquad  h(v)= h(w) = \pm 1 \\
&0, \qquad  \{ h(v), h(w) \} = \{ \pm 1 \}.
\end{cases}
\end{align*}
With a slight abuse of notation, this coupling is also denoted as $(h, \omega, B) \sim \bbP^\xi$.
\end{definition}%

\begin{corollary}[Coupling of $|h|$ and $\omega$]
\label{cor:FK cond law given abs height}
Let $H$ be a positive height function on $D=(V, E)$. Let $h$ be drawn randomly among height functions with $|h|=H$, with probabilities proportional to $W(h)$ (resp. with the additional condition that $\sign(h)=+1$ on $S \subset V$).\footnote{This is clearly a set-valued boundary condition, so to it is associated a triplet $(h, \omega, B)$.} Then, \\
1) the marginal law of $\omega $ is $ \sfP^{D}_{\mathrm{FK}} [\cdot \; | \; E_{\mathsf{fix}}(H) \subset \omega]$ (resp. $ \sfP^{D,+}_{\mathrm{FK}} [\cdot \; | \; E_{\mathsf{fix}}(H) \subset \omega]$); and \\
2) given ($H$ and) $\omega$, the conditional law of $\sign(h)$ (which determines $h$) is constant on each connected component of $\omega$, and given by independent fair coin flips for the different components (resp. except set to $+1$ on those intersecting $S$).
\end{corollary}

\begin{proof}
1) Let $\calD = \calD(H)$ as defined in Section~\ref{subsubsec: Ising tools}. By Lemma~\ref{cor:signs are Ising}, $\sign(h) \sim \sfP_{\mathrm{Ising}}^\calD $ (resp. $\sign(h) \sim \sfP^{\calD, +}_{\mathrm{Ising}} $). Thus, the conditional law of Lemma~\ref{lem:FK coupling} implies that
\begin{align*}
\varpi_{\langle v, w \rangle} = 
\begin{cases}
& 0, \qquad \sign(h(v)) \neq \sign(h(w)) \Leftrightarrow \{ h(v), h(w) \} = \{ \pm 1 \} \\
& B_{\langle v, w \rangle}, \qquad \text{otherwise}
\end{cases}
\end{align*}
is coupled to $\sign(h) $ in the coupling of that Lemma, and in particular $\varpi \sim \sfP^{\calD}_{\mathrm{FK}}$ (resp. $\varpi \sim \sfP^{\calD, +}_{\mathrm{FK}}$).
Comparing to the definition of $\omega$, we observe that $\omega = \varpi \cup  E_{\mathsf{fix}}(H)$, whose law is by Lemma~\ref{lem:FK cluster conditioning}  $ \sfP^{D}_{\mathrm{FK}} [\cdot \; | \; E_{\mathsf{fix}}(H)  \subset \omega]$ (resp.  $ \sfP^{D, +}_{\mathrm{FK}} [\cdot \; | \; E_{\mathsf{fix}}(H)  \subset \omega]$).

2) We continue in the coupling above. Given % $\omega$, the conditional law of 
$\varpi$,
%is given by Lemma~\ref{lem:FK cluster conditioning} and then
the conditional law of $\sigma = \sign (h)$ is by Lemma~\ref{lem:FK coupling} given by independent fair coin flips on the $\calG$-clusters of $\varpi$. But since $\omega = \varpi \cup  E_{\mathsf{fix}}(H)$, these $\calG$-clusters are exactly the $G$-clusters of $\omega$.
\end{proof}

We conclude this section with noticing that $\omega$ increases if either $H$ or the graph $D$ grows.

\begin{lemma}
\label{lem:FK coupling 2}
Let $D = (V, E)\subset D'=(V', E')$ be finite connected graphs, $H \in \heightfcns_{D}$ and $H' \in \heightfcns_{D'}$ positive height functions  $ H \preceq H'_{|D}$,  let $V \supset S \subset S' \subset V'$ (possibly empty), and $\omega \subset E$ (resp. $\omega' \subset E'$) be as given by the above lemma based on $D$, $H$ and $S$ (resp. on $D'$, $H'$ and $S'$). There is a coupling of $\omega $ and $\omega'$ such that $\omega \subset \omega'_{|D}$.
\end{lemma}

\begin{proof}
Let us explicate the case $S=S'=0$; the full generality is just notationally tougher. Let $\tilde{\omega}$ be an FK configuration on $D'$, conditional on the absence of all edges in $D' \setminus D$ and the presence of the edges of $E_{\mathsf{fix}}(H) \subset E \subset E'$. From the definition of the FK-Ising model, one obtains that $\tilde{\omega}_{|D} $ and $\omega$ are equally distributed. Note that $E_{\mathsf{fix}}(H) \subset E'_{\mathsf{fix}}(H')$. It is then standard that $\tilde{\omega}$ can be coupled to $\omega' \sim \sfP^{D}_{\mathrm{FK}} [\cdot \; | \; E'_{\mathsf{fix}}(H')  \subset \omega']$ so that $\tilde{\omega} \subset \omega'$.
\end{proof}

\section{Localization: from FK to random Lipschitz}

In this section, we are interested in random Lipschitz functions as the graph $D$ grows towards an infinite lattice.
Let thus $\bbG = (\bbV, \bbE)$ be a countably infinite, locally finite, connected graph with edge weights $e \mapsto \mathbf{c}_e \geq 1$. Equip the space of integer-valued functions on $\bbV$ (resp. on $\bbE$) with the (metric) topology of convergence over finite subsets. For a finite connected subgraph $D= (V,E)$ of $\bbG$, we denote by $ \partial D$  for the vertices of $D$ that are adjacent to $\bbE \setminus E$ in $\bbG$. Throughout this paper, constant (possibly set-valued)
boundary conditions on $\partial D$ are directly annotated as a superscript, e.g., $\bbP^{1}_D$ or $\bbP^{\pm 1}_D$.
%We say that $D$ is \textit{generated by} $V$ if $E = \{ e \in \bbE: e$ traverses between two vertices of $ V \}$. 

Recall that the FK-Ising measures $\sfP^{D,+}_{\mathrm{FK}} $ wired on $S=\partial D$ (where still $p_e = 1-1/\mathbf{c}_e$) have an infinite-volume limit (e.g., weakly as functions on $\bbE$) as $D \uparrow \bbG$, which is independent of the sequence of domains $D$; this limit is denoted by $\sfP^{\bbG,+}_{\mathrm{FK}} $. 

\begin{theorem}
Fix the graph $\bbG$, weights $\mathbf{c}_e$, and $x \in \bbV$.
If the full-plane FK-Ising model $\sfP^{\bbG,+}_{\mathrm{FK}} $ almost surely has an infinite cluster, then there exists $C>0$ such that for all subgraphs $D = (V,E) \subset \bbG$ such that $x \in V$, we have
\begin{align*}
\Var^{1}_{D} (h(x)) \leq C,
\end{align*}
where $\Var^{1}_{D}$ denotes variance under $\bbP^{1}_D$.
\end{theorem}

\begin{proof}
%It suffices to show that for any growing sequence of domains $D_n \uparrow \bbG$, there exists $C$ such that $\Var^{1}_{D_n} (h(x)) \leq C$ for all $n$.
%
The proof is based on a coupling, which we first explain for a fixed $D$. 
Let $\nu' \sim \sfP^{\bbG,+}_{\mathrm{FK}}$ and let $H$ be a random height function following the law of \textit{absolute}-height under $\bbP^{1}_{D}$, independent of $\nu'$. Given $H$ and $\nu'$, let the conditional law of $(\pi, \omega) \in \{ 0,1 \}^E \times \{ 0,1 \}^E$ be two FK models on $D$ in the standard increasing FK coupling of the following boundary conditions: $\pi$ with the external boundary condition $\nu'_{| D^c}$ and $ \omega \sim \sfP^{D,+}_{\mathrm{FK}} [\cdot \; | \; E_{\mathsf{fix}}(H) \subset \omega]$, so $\pi \subset \omega$. Then, let $h \in \heightfcns_D$ be determined by its the conditional law given $H$ and $\omega$, given by adding signs to $H$ by independent fair coin flips on each component of $\omega$, except a deterministic sign $+1$ given to the components intersecting $\partial D$. Finally, let $\nu \in \{ 0, 1 \}^\bbE$ be given by $\nu_{| D^c} = \nu'_{| D^c}$, $\nu_{| D} = \pi$.
Hence, by Corollary~\ref{cor:FK cond law given abs height}, $(h, \omega) \sim \bbP^{1}_{D}$, by the well-known Gibbs property of $\sfP^{\bbG,+}_{\mathrm{FK}}$, $\nu \sim \sfP^{\bbG,+}_{\mathrm{FK}}$, and by construction, $\nu_{|D} \leq \omega$.

Let $\omega_\partial$ denote the clusters of $\omega$ adjacent to $\partial D$, and let $\nu_\infty$ denote the infinite clusters of $\nu$ (which exist a.s. by assumption). Then, in the coupling above, we have
\begin{align*}
h(x) \geq 1 - 2 d_D(x, \omega_\partial) \geq 1 - 2 d_D(x, \nu_\infty \cup \partial D).
\end{align*}
Observe that the shortest $\bbG$-path from $x$ to $\nu_\infty \cup \partial D$ must either only use edges of $E \subset \bbE$, or enter $\bbE \setminus E$ via $\partial D$; hence, we have
% is also a $D$-path,
%Observe also that since $D$ was assumed to be generated by $V$, the shortest $\bbG$-path from $x$ to $\nu_\infty \cup \partial D$ is also a $D$-path, so we have
\begin{align*}
d_D(x, \nu_\infty \cup \partial D ) \leq d_\bbG (x, \nu_\infty \cup \partial D ) \leq d_\bbG (x, \nu_\infty ),
\end{align*}
where the second step is trivial.
Hence for any $m \in \bbN$
\begin{align}
\label{eq:FK localization argument}
\bbP_D^1 [ h(x) \leq 1-2m ] 
%\leq \sfP^{\bbG,+}_{\mathrm{FK}} [d_D(x, \nu_\infty \cup \partial D ) \geq m] 
\leq 
\sfP^{\bbG,+}_{\mathrm{FK}} [d_\bbG (x, \nu_\infty ) \geq m],
\end{align}
and by the symmetry of $h \sim \bbP_D^1$ around $1$, the same bound holds for $ \bbP_D^1 [ h(x) \geq 1+2m ]$.

We are now ready conclude. It suffices to show that for any sequence of domains $D_n$, there exists $C$ such that $\Var^{1}_{D_n} (h(x)) \leq C$ for all $n$. Suppose for a contradiction that there existed a (sub)sequence so that $\Var^{1}_{D_n} (h(x))$ was unbounded. The above probability bounds are uniform in $D$ and tend to zero as $m \uparrow \infty$; hence the collection of laws of $h(x)$ under different $\bbP_{D_n}^1$ is tight. Thus, by picking a further subsequence, we may assume that $h(x)$ under $\bbP_{D_n}^1$ converge weakly to some law $\mu$, i.e., the point masses $\bbP^1_{D_n} [h(x)=k]$ converge to $\mu[\{ k\}]$ for all $k \in \bbN$. It is known~\cite[Lemma~8.2.4]{She05} that the point mass function of $h(x)$ on $\Zodd$, under any $\bbP^1_D$, is log-concave. Also, it is clearly symmetric around $1$ and hence also attains its maximum at $1$. Consequently, this is the case for the point masses of $\mu$, too. Let $M \in \bbN$ then be any point where $\mu[\{ 2M+3\}]/ \mu[\{ 2M + 1 \}] \leq 1- \epsilon $ for some $\epsilon >0$. Then, for all $n$ large enough, $n>n_0$, $\bbP^1_{D_n} [h(x)=2M + 3]/\bbP^1_{D_n} [h(x)=2M+1] \leq 1- \epsilon/2 $, and by the log-concavity of $\bbP^1_{D_n}$, $\bbP^1_{D_n} [h(x)=2m + 3]/\bbP^1_{D_n} [h(x)=2m+1] \leq 1- \epsilon/2 $  for all $m \geq M $, and all $n> n_0$. This exponential decay bounds $\Var^{1}_{D_n} (h(x))$ uniformly over $n>n_0$, and in particular, the sequence $\Var^{1}_{D_n} (h(x))$ cannot diverge, a contradiction.
\end{proof}

\begin{corollary}
\label{cor:localized regimes}
Let $\mathbf{c}_e = \mathbf{c}$ be constant over the edges of $\bbG$ where
\begin{itemize}[noitemsep, topsep=0pt]
\item[i)] $\bbG$ is the planar square lattice $\bbZ^2$ and $\mathbf{c} > 1 + \sqrt{2}$; or
\item[ii)] $\bbG$ is the planar hexagonal (honeycomb) lattice and $\mathbf{c} > 2 + \sqrt{3}$; or
\item[iii)] $\bbG$ is the planar triangular lattice and $\mathbf{c} > \sqrt{3}$.
\end{itemize}
There exists $C>0$ such that for all finite connected subgraphs $D = (V,E) \subset \bbG$ and all $x \in V$,
\begin{align*}
\Var^{1}_{D} (h(x)) \leq C.
\end{align*}
\end{corollary}

\begin{proof}
The  FK-Ising model with constant edge parameters $p_e = p$ is known to be critical at~\cite[Theorems~1 and~4]{BefDum12}
\begin{itemize}[noitemsep, topsep=0pt]
\item[i)] $p=p_c = \tfrac{\sqrt{2}}{ 1 + \sqrt{2} }$ on the square lattice;
\item[ii)] $p = p_c = \sqrt{3}-1$ 
% $p_c = \frac{1 + \sqrt{3}}{ 2 + \sqrt{3} }$
on the hexagonal lattice; and
\item[iii)] $p = p_c =  1 - \tfrac{ 1 }{ \sqrt{3} }$ on the triangular lattice,
\end{itemize}
and for $p>p_c$, the related full-plane measures are known to almost surely exhibit an infinite cluster, by~\cite[Corollary~1.3]{DLT} and ergodicity. The result then follows from the previous corollary, where, due to the self-similarity of the lattice, $C$ can be taken independent of $x$.
\end{proof}

\section{Basic properties of the models}
\label{sec:basics}

In this section, we are again working on a fixed finite connected graph $D=(V,E)$.

\subsection{Height functions}
\label{subsec:SMP and FKG}

This subsection reviews well-known results. We start with the Spatial Markov property (SMP), whose proof we leave for the reader. We say that a subgraph $D=(V, E) \subset D' = (V', E')$ is \textit{generated by} $V$ if $E = \{ e \in E': e$ traverses between two vertices of $ V \}$.

\begin{lemma}
\label{lem:SMP}
Let $D' = (V', E')$ be a finite connected graph and $D=(V,E)$ its connected subgraph generated by $V$; let $\xi : \Delta \to \Zodd$ be a boundary condition for $\heightfcns_{D'} $ with $\Delta \cap V = \partial D$, and let $h \sim \bbP^{\xi}_{D'}$. For $A \subset V'$ with $A \cap V = \partial D $ and a $\bbP^{\xi}_{D'}$-possible height configuration $ h_{|A} = \alpha$ on $A$, the conditional law of $h_{|D}$ given $ h_{|A} = \alpha$ is $\bbP^{\alpha_{| \partial D}}_D $.
\end{lemma}

We now move on to the FKG inequality. Let us define the \emph{partial order relation} $\preceq$ on $\heightfcns_D$ by setting $h \preceq h'$ for $h, h' \in \heightfcns_D$ if and only if $h(x) \leq h'(x)$ for all $x \in V$. 
A function $F: \heightfcns_D \to \bbR$ is  \emph{increasing} if
$h \preceq h' $ implies that $ F(h) \leq F(h')$, and an event $A$ is {\em increasing} if its indicator function $\mathbbm{1}_A$ is an increasing function. 
%\emph{Decreasing} functions (resp. events) are negatives (resp. complements) of increasing ones. 

A similar partial order is naturally defined for single-valued boundary conditions on a fixed boundary set $\Delta$. More generally, we say that an admissible boundary condition $\xi$ on $\Delta$ is \textit{interval-valued}, if it is set-valued of the form $\xi(x)=[a(x), b(x)]$, where $a, b: \Delta \to \Zodd$ satisfy $a(x) \leq b(x)$ for all $x \in \Delta$. The partial order above is extended to interval-valued boundary conditions $\xi, \xi'$ on $\Delta$, by defining that $\xi \preceq \xi'$ if $a(x) \leq a'(x)$ and $b(x) \leq b'(x)$ for all $x \in \Delta$.

The results below are stated in terms of expectations of increasing functions, but we will mostly apply them to probabilities of increasing events or their complements.
%or %decreasing events. 

\begin{proposition}\label{prop: CBC and FKG}\label{prop:FKG}
  Let $D$ be a finite connected graph, $\xi \preceq \xi'$ two interval (or single) valued boundary conditions on $\Delta$, and $F,G: \heightfcns_D \to \bbR$ increasing functions; then, we have
    \begin{align}
    	\bbE^{\xi} [F(h)G(h)] &\geq \bbE^{\xi} [F(h)] \bbE^{\xi} [G(h)], \tag{FKG}\label{eq:FKG-h}\\
    	\bbE^{\xi'} [F (h)] &\geq \bbE^{\xi} [F (h)].\tag{CBC}\label{eq:CBC-h}    
    \end{align}
\end{proposition}

\subsection{Absolute heights}
\label{subsec: SMP and FKG for abs val}

%This subsection addresses the positive association properties of the absolute-height.
The proofs for this section are given in Appendix~\ref{app:basics} for the fluency of reading.

We say that an admissible set-valued boundary condition $\xi$ on $\Delta$ is an \textit{absolute-value boundary condition} if it is on some $S \subset \Delta$ (possibly empty) interval-valued $[a(x), b(x)]$ where $-1 \leq a(x) \leq b(x)$ and $a(x)+ b(x) \geq 2$, and on $\Delta \setminus S$ of the form $\pm [a(x), b(x)]$ where $1 \leq a(x) \leq b(x)$.
For two absolute-value boundary conditions $\xi, \xi'$ on $\Delta$, we denote $\xi \preceq_{abs} \xi'$ if $S \subset S'$ and 
\begin{align*}
\begin{cases}
a(x) \leq a'(x) \quad \text{and} \quad b(x) \leq b'(x) \qquad \text{for all } x \in S \text{ and } x \in \Delta \setminus S' \\
b(x) \leq |a'(x)| \qquad \text{for all } x \in S' \setminus S.
\end{cases}
\end{align*}
(The reader may verify that $\preceq_{abs}$ indeed is a partial order on the set of absolute-value boundary conditions.) Note that an interval-valued boundary condition $\xi \succeq \{ \pm 1 \}$ is also an absolute-value boundary condition, and for two such conditions, $\xi \preceq \xi'$ and $\xi \preceq_{abs} \xi'$ are equivalent.

\begin{proposition}\label{prop:|h|-monotonicity}
	Let $D$ be a finite connected graph, $\xi, \xi'$ two absolute-value boundary conditions on $\Delta$ with $\xi \preceq_{abs} \xi' $, and $F,G: \heightfcns_D \to \bbR$ increasing functions; then we have
	\begin{align} 
			 \bbE^{\xi}_D \big[F( \vert h \vert) G( \vert h \vert)\big] 
		&\geq \bbE^{\xi }_D \big[F( \vert h \vert)\big]  \bbE^{\xi}_D \big[ G( \vert h \vert)\big],
		 \tag{FKG-|h|} \label{eq:FKG-|h|}\\
		\bbE^{\xi'}_D \big[F( \vert h \vert) \big]
		&\geq 
		\bbE^{\xi}_D \big[F( \vert h \vert) \big].
		\tag{CBC-|h|} \label{eq:CBC-|h|}
	\end{align}
\end{proposition}

The proof of the above statement is a verification (although as such nontrivial) of the well-known Holley and FKG criteria; see appendix. We remark that~\eqref{eq:CBC-|h|} would \textit{not} hold true if we directly considered integer-valued one-Lipschitz functions instead of odd two-Lipschitz functions. An easy counterexample is obtained with $\mathbf{c}_e \equiv 1$ on the linear three-vertex graph $(1, 2, 3)$, by comparing boundary conditions $\xi, \xi'$ on $\Delta = \{ 1, 3\}$ with $\xi(1) = \xi' (1) = 0$ and $\xi(3)=0$, $\xi'(3) = \{ \pm 1\}$. Likewise, it is easily verified that absolute-value boundary conditions do \textit{not} satisfy Spatial Markov property. The following, \textit{inequality} form SMP is however very useful.

\begin{proposition}
\label{prop:SMP ineq for abs val}
Let $D' = (V', E')$ be a finite connected graph, $D=(V,E)$ its connected subgraph generated by $V$. Let $\xi$ be an absolute-value boundary condition on $\Delta \supset \partial D$, with $|\xi|$ single-valued on $\partial D$. Then, for any $F: \heightfcns_D \to \bbR$ increasing,
	\begin{align}
	\label{eq:bdary cut}
	\tag{SMP-|h|}
	\bbE^{\xi}_{D'} \big[F( \vert h_{|D} \vert) \big] \geq \bbE^{\xi_{|D}}_{D} \big[F( \vert h \vert) \big].
	\end{align}
\end{proposition}

\subsection{Heights with percolations}

The (easy) proofs for this subsection are postponed to Appendix~\ref{app:roskis}. We start with a SMP.

\begin{lemma}
\label{lem:SMP with percolation}
Let $D' = (V', E')$ be a finite connected graph and $D=(V,E)$ its connected subgraph; let $E(V)=E$ (resp. $ E(V^c) $, resp. $E(V, V^c)$) denote the edges that connect two vertices in $V$ (resp. two in $V^c$, resp. one in $V$ to one in $V^c$) and let $h \sim \bbP^{\xi}_{D'}$. Given any ($\bbP^{\xi}_{D'}$-possible) configuration of $(h_{|V^c}, \omega_{| E (V^c)}, B_{| E (V^c)})$, \emph{and the event} $\omega_{|E(V, V^c)} = 0$, the conditional law of $(h_{|V}, \omega_{|E}, B_{|E})$ is $\bbP^{\alpha }_D$, where $\alpha$ is the boundary condition obtained by extending $\xi_{|V \setminus \partial D}$ with the set-valued boundary condition $\{ \pm 1 \}$ to the endpoint vertices $\partial D$ of $E(V, V^c)$ in $V$.
\end{lemma}

\begin{remark}
\label{rem:SMP with perco}
This lemma directly implies the same conditional law for $(h_{|V}, \omega_{|E}, B_{|E})$ given less information, for instance only given that $\omega_{|E(V, V^c)} = 0$ or given $(|h_{|V^c}|, \omega_{| E (V^c)})$, and the event $\omega_{|E(V, V^c)} = 0$.
\end{remark}

We then move onward to positive association.
For pairs of height functions $h, h' \in \heightfcns_D$ and percolation configurations $\nu, \nu' \subset E$, we denote $(h, \nu) \preceq (h', \nu')$ if  $h \preceq h'$ and  $\nu \subset \nu'$. The concepts of increasing  functions and events of $(h, \nu)$ then readily generalize. The function-percolation pairs, with the percolation $\nu$ being either $\omega$ or $B$, then satisfy the following.

\begin{proposition}
\label{prop:pos assoc of percolated height fcns}
The statements of (1) Proposition~\ref{prop: CBC and FKG}; (2) Propositions~\ref{prop:|h|-monotonicity} and~\ref{prop:SMP ineq for abs val} remain true when $F$ and $G$ are taken to be functions $\heightfcns_D \times \{ 0,1 \}^E \to \bbR$ that are (1) increasing functions in $(h,B)$ or $(h,-B)$; (2) increasing functions in $(|h|, B)$, $(|h|, -B)$, or $(|h|, \omega)$.
\end{proposition}

We conclude this subsection with a stronger form of Proposition~\ref{prop:SMP ineq for abs val} for $\{ \pm 1 \}$ boundary conditions:

\begin{corollary}
\label{cor:+-1 SMP ineq for abs val}
Let $D' = (V', E')$ be a finite connected graph, $D=(V,E) \subset D'$ \emph{any} connected subgraph, and $\xi'$ be \emph{any} absolute-value boundary condition for $\heightfcns_{D'}$ on $\Delta \subset (D^c \cup \partial D)$. Then, for any $F$ depending only on, and increasing in $(|h|, B)$, $(|h|, -B)$, or $(|h|, \omega)$ \emph{on $D$}
	\begin{align*}
	\bbE^{\xi'}_{D'} \big[F \big] \geq \bbE^{\pm 1}_{D} \big[F \big].
	\end{align*}
\end{corollary}

\begin{proof}
Suppose first that $F$ is a function of  $(|h|, \omega)$ on $D$.
For $(|h|, \omega)$ on $D'$, the event $\{ \omega_{|E(V, V^c)} = 0\} $ is decreasing and by Proposition~\ref{prop:pos assoc of percolated height fcns},
\begin{align*}
\bbE^{  \xi'  }_{D'} [F] \geq \bbE^{  \chi  }_{D'} [F] \geq \bbE^{ \chi }_{D'} [F \; | \;  \omega_{|E(V, V^c)} = 0],
\end{align*}
where $\chi$ is constant $\pm 1$ on $\Delta$.
By Lemma~\ref{lem:SMP with percolation} and Remark~\ref{rem:SMP with perco}, the conditional measure on the right coincides on $D$ with $\bbE^{  \pm 1 }_{D}$; this concludes the first case. The case when $F$ is a function of $(|h|, B)$ or $(|h|, -B)$ on $D$ boils down to functions of $|h|$ only by conditioning on $B$ (which is independent of $h$).
\end{proof}

\section{Two equivalent phase dichotomies}
\label{sec:equivalent phase characterizations}

\subsection{Variance dichotomy}

\begin{theorem}
\label{thm:variance dichotomy}
Let $\bbG = (\bbV, \bbE)$ be a countably infinite, locally finite, connected graph with edge weights $e \mapsto \mathbf{c}_e \geq 1$; fix $x \in \bbV$ and let $D=(V,E) \ni x$ be a finite connected subgraphs of $\bbG$.
Then, $\bbE_{D}^{  \pm 1  } [h(x)^2]$ is increasing in $D$ and in particular
either
\begin{align}
\tag{Var-loc}
\label{eq:Var-loc}
\bbE_{D}^{\pm 1  } [h(x)^2] \uparrow r(x) \in \bbR \qquad \text{when } D \uparrow \bbG
\end{align}
or
\begin{align}
\tag{Var-deloc}
\label{eq:Var-deloc}
\bbE_{D}^{ \pm 1 } [h(x)^2] \uparrow + \infty \qquad \text{when } D \uparrow \bbG.
\end{align} 
Furthermore, the same alternative takes place for all $x \in \bbV$.
\end{theorem}

\begin{proof}
The fact that $\bbE_{D}^{ \pm 1 } [h(x)^2]$ is monotonous in $D$ follows directly from Corollary~\ref{cor:+-1 SMP ineq for abs val}. If~\eqref{eq:Var-loc} holds for some $x \in \bbV$, then it also holds for any $y \in \bbV$ since $| h(x)-h(y) | \leq 2 d_D (x, y)$ and $d_D (x, y) = d_\bbG (x, y)$ for all $D$ large enough.
\end{proof}

The log-concavity on $\Zodd$ of the law of $h(x)$ under $\bbP^{\pm 1}_D$~\cite[Lemma~8.2.4]{She05} implies several interesting equivalent formulations of~\eqref{eq:Var-loc}, see~\cite[Section~7]{CPT18}. Let us explicate the exponential decay in relation to, e.g.,~\cite{GM-loc}; note also that~\eqref{eq:CBC-h} implies similar exponential decay under any bounded boundary condition instead of $\{ \pm 1 \}$.

\begin{proposition}[Exponential decay]
\label{prop:exp decay in loc phase}
In the setup of the previous theorem, suppose that~\eqref{eq:Var-loc} occurs. There exist $C, c >0$ such that for all $D$,
\begin{align*}
\bbP_D^{\pm 1} [h(x) \geq 2m + 1] \leq C e^{-cm}.
\end{align*}
\end{proposition}

\subsection{Gibbs measure dichotomy}

We remind the reader that we equip the space of integer-valued functions on $\bbV$ with the topology of convergence over finite subsets. Height functions on $D \subset \bbG$ are studied in this space by extending the to the entire $\bbG$ by an arbitrary manner, say for definiteness by the constant function $1$ outside of $D$.

A probability measure $\mu$ supported on odd two-Lipschitz functions on the vertices of $\bbG$ is a \textit{(random Lipschitz function) Gibbs measure} if for $h \sim \mu$ any finite connected subgraph $D=(V, E)$ of $\bbG$ generated by $V$ it holds that the conditional law of $h_{|D}$, given the occurrence of any ($\mu$-possible) configuration $\{ h_{|\partial D} = \xi \}$, is $\bbP_{D}^\xi$.

\begin{theorem}
\label{thm:Gibbs dichotomy}
Let $\bbG = (\bbV, \bbE)$ be a countably infinite, locally finite, connected graph with edge weights $e \mapsto \mathbf{c}_e \geq 1$, let $D_n=(V_n,E_n)$ be finite connected subgraphs of $\bbG$, and $h_n \sim \bbP^{ \pm 1 }_{D_n}$, Then,
either
\begin{itemize}[noitemsep, font=\upshape]
\item[\mylabel{eq:Gibbs-loc}{Gibbs-loc}{}] \emph{(Gibbs-loc):} $h_n$ converge weakly to a Gibbs measure, for any sequence $D_n \uparrow \bbG$; or
\item[\mylabel{eq:Gibbs-deloc}{Gibbs-deloc}{}] \emph{(Gibbs-deloc):} there exists no sequence $D_n \uparrow \bbG$ and $x \in \bbV$ such that $h_n (x)$ would constitute a tight family real random variables.
\end{itemize}
Furthermore~\eqref{eq:Gibbs-loc}~$\Leftrightarrow$~\eqref{eq:Var-loc} and~\eqref{eq:Gibbs-deloc}~$\Leftrightarrow$~\eqref{eq:Var-deloc}.
\end{theorem}

Note that in the case~\eqref{eq:Gibbs-loc}, the limiting Gibbs measure is independent of the sequence $D_n$, which is observed by intertwining two sequences. Consequently, e.g., for is the square lattice $\bbG = \bbZ^2$ and $\mathbf{c}_e \equiv \mathbf{c}$, the limit is also invariant under translations and symmetries of $\bbZ^2$, which is seen by performing these operations to $D_n$ and $h_n$. This observation is crucial in the present paper, due to the following strong result.

\begin{theorem}[{\cite[Theorem~2.5]{Piet-deloc}}, special case]
\label{thm:Piet's non-quantitative deloc thm}
Suppose that $\bbG$ is a planar graph of maximum degree $3$ and invariant under the translations of $\bbZ^2$, with weights $1 \leq \mathbf{c}_e \leq 2$ also invariant under these translations. Then, there exist no $\bbZ^2$-translation invariant random Lipschitz function Gibbs measures on $\bbG$.
\end{theorem}

Note that by mapping linearly, e.g., the honeycomb lattice can be embedded in the plane so that it is $\bbZ^2$ translation invariant and the above theorem holds. The above observation on the limiting Gibbs measure in Theorem~\ref{thm:Gibbs dichotomy} the directly gives the following.

\begin{corollary}
\label{cor:non-quant deloc}
In the setup of the above theorem,~\eqref{eq:Gibbs-deloc} and~\eqref{eq:Var-deloc} occur.
\end{corollary}

The main objective of the present paper is thus to quantify this non-quantitative divergence of variance. We now turn to the proof of Theorem~\ref{thm:Gibbs dichotomy}. Similar statements, although apparently not covering the case of random Lipschitz functions, have appeared at least in~\cite{CPT18, PO21}; we give a proof mimicking~\cite{CPT18}.

\begin{proof}[Theorem~\ref{thm:Gibbs dichotomy}]
It is rather immediate that~\eqref{eq:Gibbs-deloc}~$\Rightarrow$~\eqref{eq:Var-deloc}.

It is also fairly easy to show that~\eqref{eq:Gibbs-loc}~$\Rightarrow$~\eqref{eq:Var-loc}. Namely, assuming~\eqref{eq:Gibbs-loc}, in particular the marginal law of $h_n(x)$ converges weakly. For each $\bbE_{D_n}^{ \{ \pm 1 \} } = \bbE_n$, $h_n(x)$ is a log-concave random variable on $\Zodd$ by~\cite{She05}, and thus so is $h(x)$ under the weak limit $\sfP$. In particular, $\sfE [h(x)^2] < \infty$. Since the laws of $|h_n(x)|$ under $\bbE_n$ are stochastically increasing in $n$, it follows that
\begin{align*}
\bbE_n [\min \{ h_n(x)^2, M \}]
\leq \sfE [\min \{ h(x)^2, M \}]
\end{align*}
for all $n$ and all $M \in \bbR$, and letting $M \uparrow \infty$, we obtain $\bbE_n [h_n(x)^2] \leq \sfE [h(x)^2] < \infty$ for all $n$. By the dichotomy in the previous theorem,~\eqref{eq:Var-loc} must thus occur.

Due to Theorem~\ref{thm:variance dichotomy}, it then suffices to show that either~\eqref{eq:Gibbs-loc} or~\eqref{eq:Gibbs-deloc} must occur. 
Suppose that~\eqref{eq:Gibbs-deloc} does \textit{not} occur, i.e., $h_n(x)$ are tight over some sequence $D_n \uparrow \bbG$. Note that $h_n(x)$ are tight if and only if $|h_n(x)|$ are tight. The latter are stochastically increasing in $n$ by Corollary~\ref{cor:+-1 SMP ineq for abs val}, and hence tightness over some sequence $D_n \uparrow \bbG$ implies tightness over any sequence. Since $|h_n(y)- h_n(x) | \leq 2 d_{D_n} (x, y)$, $h_n(y)$ are tight also for any $y$ and any sequence $D_n \uparrow \bbG$. It follows that $|h_n|$ are tight in our topology of $\bbZ^\bbV$.

We now construct a coupling of $(\bbE_n)_{n \in \bbN}$, under which $h_n$ converge almost surely.
First, for any two $D \subset D'$, by Lemmas~\ref{lem:Holley} and~\ref{lem:FK coupling 2}, we may couple $|h|$ and $\omega$ so that $|h| \preceq |h'|$ and $\omega \subset \omega'$. We thus couple $(|h_n|, \omega_n)$ under $(\bbE_n)_{n \in \bbN}$ by chaining couplings of this type, so that $(|h_1|, \omega_1) \preceq (|h_2|, \omega_2) \preceq \ldots$, so $(|h_n|, \omega_n)$ are under this coupling both increasing and tight. It follows that they increase to a limit variable, i.e., they converge almost surely.

Finally, recall from Corollary~\ref{cor:FK cond law given abs height} that given the pair $(|h_n|, \omega_n)$, the sign of $h_n$ can be sampled by i.i.d. fair coin flips on the clusters of $\omega_n$. In the above coupling, sampling the signs onto $|h_n|$ with different $n$ from the same coin flips (e.g., ordering vertices, giving each vertex a coin flip, and sampling the sign on an $\omega_n$-cluster from the lowest-index vertex in it), we obtain a coupling of $h_n$ where $h_n$ converge almost surely, and hence also weakly. It is standard that the limiting measure inherits the Gibbs property.
\end{proof}

\part{Logarithmic variance on cubic planar lattices}

So far, our arguments have been combinatorial and applicable for the random Lipschitz model on very general graphs. In this second, core part of the article, we use plane geometric arguments to improve the result of Theorem~\ref{thm:Piet's non-quantitative deloc thm} by proving a logarithmic rate of divergence for the variance (Theorem~\ref{thm:main thm 2nd}).

\section{Overview}

\subsection{Setup and main result}
\label{subsec:main thm setup}

\begin{figure}
\begin{center}
\includegraphics[width=0.35\textwidth]{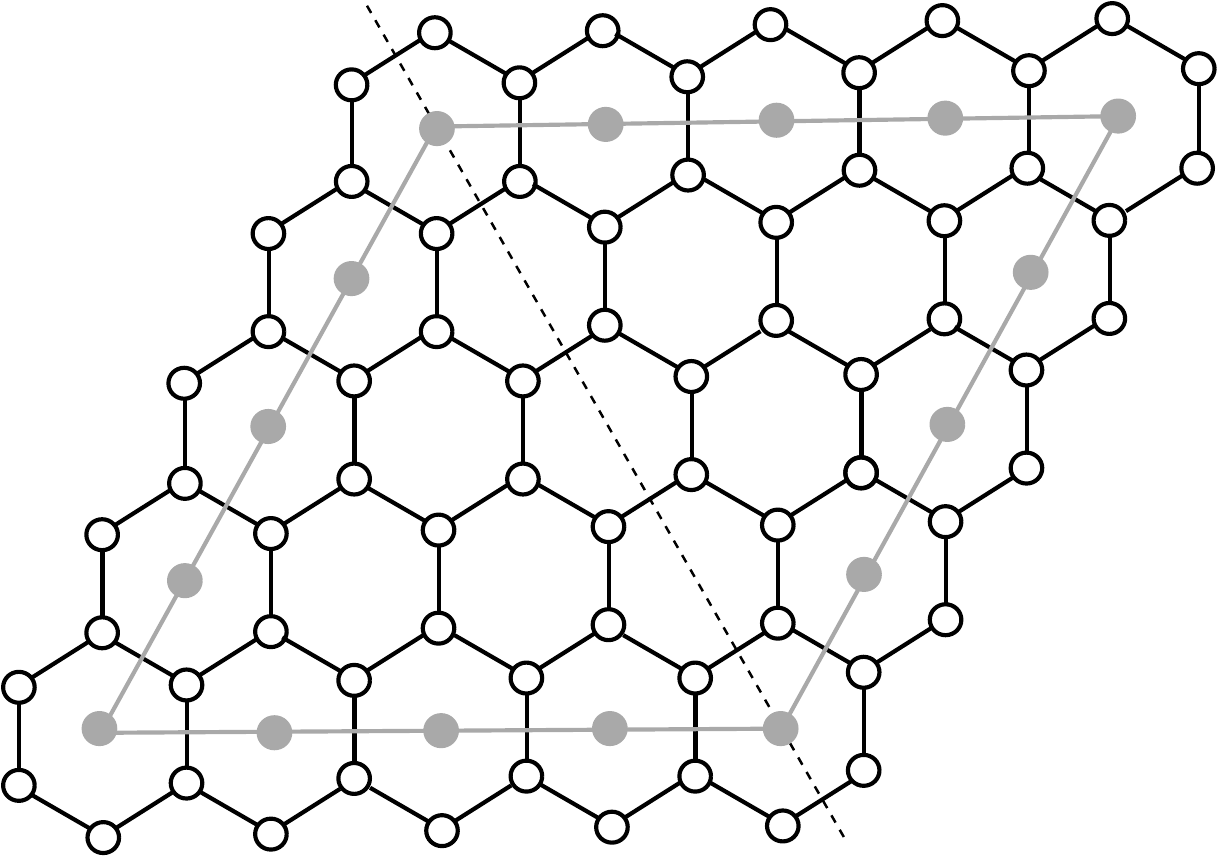} \qquad
\includegraphics[width=0.26\textwidth]{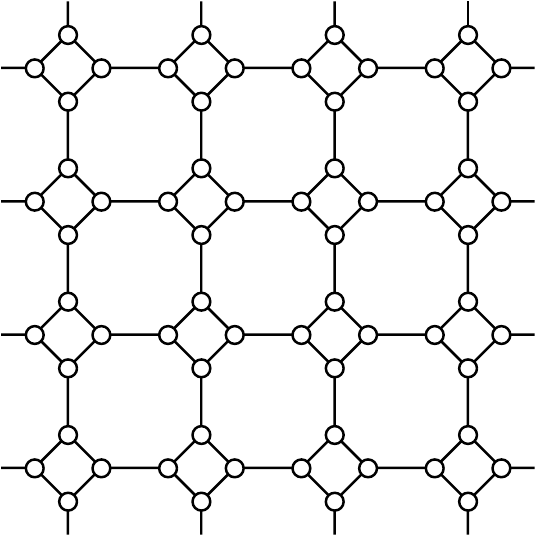} \\
\includegraphics[width=0.35\textwidth]{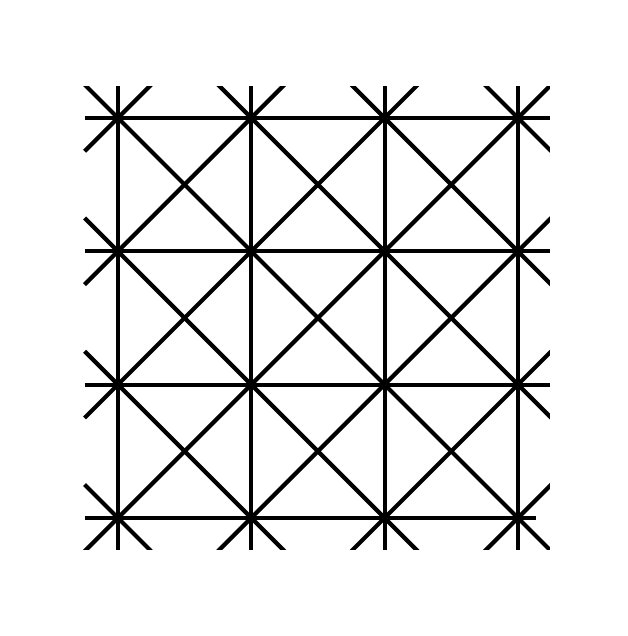}%
\includegraphics[width=0.35\textwidth]{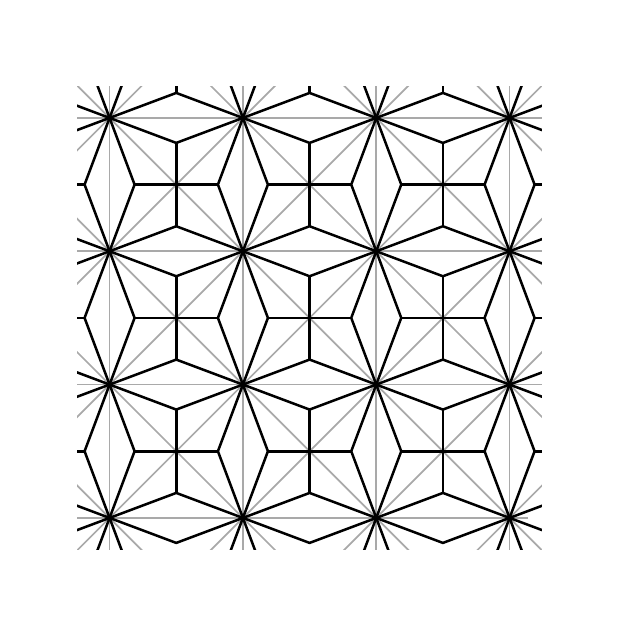}%
\includegraphics[width=0.35\textwidth]{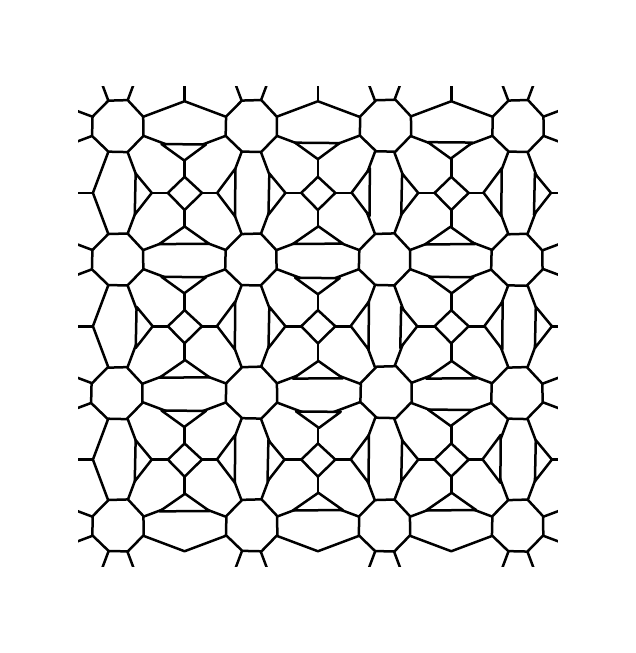}
\end{center}
\caption{
\label{fig:cubic lattices}
Illustrations for the main theorem~\ref{thm:main thm 2nd} and its outline. Top row: the hexagonal and square-octagon lattices are probably the two most important lattices for which the theorem is applicable. Top left: the lozenge $L_2$ on the hexagonal lattice is determined by a 4-by-4 lozenge walk on the dual (in gray); the same notation directly interpreted as a planar domain, dual subgraph (on or inside the loop) as well as a primal subgraph (strictly inside). The reflection symmetry axis of the lozenge is dashed. Bottom row: a general ``recipe'' to produce lattices where Theorem~\ref{thm:main thm 2nd} applies. Start from a lattice with the required symmetries and with maximum degree and maximum dual degree both at most $11$ (left). Draw its corner-to-face adjacency graph (middle). Then replace each vertex of degree $d$ of the latter with a little $d$-gon (right). The resulting graph (right) is cubic with maximal dual degree $11$, and has the same symmetries as the original one (left).
}
\end{figure}

For the rest of this article, we fix the lattice $\bbG$ and the weights $\mathbf{c}_e$ (inexplicit constants may depend on these) and impose following assumptions (see Figure~\ref{fig:cubic lattices} for examples): $\bbG$ is a connected, countably infinite, locally finite graph. It is planar and bi-periodic (and comes with such an embedding), one of the translational symmetries being along the horizontal axis. The maximum degree for $\bbG$ is $3$ and for its dual $\bbG^*$ it is at most $11$. The graph $\bbG$ is invariant under reflection with respect to both horizontal and vertical axes. There is a discretization of either a square of side-length $n$ translational invariance steps (e.g., in the case of the square-octagon lattice), or a  rhombus of side-length $n$ with angles $60^\circ$ and $120^\circ$ (e.g., in honeycomb lattice), such that both $\bbG$ and the discretized square (resp. rhombus) are invariant under reflection with respect to the diagonal line of the square (resp. the diagonal line between the $120^\circ$ angles of the rhombus).
% For definiteness, and since the honeycomb lattice is the prime example of these lattices, we present arguments with such symmetric reference domain with the rhombi.
Finally, we assume throughout this section that $1 \leq \mathbf{c}_e \leq 2$ for all $e \in \bbE$ and that $ \mathbf{c}_e$ satisfy the same symmetries as the lattice $\bbG$. The main result of this section can now be stated.

\begin{theorem}
\label{thm:main thm 2nd}
%Let the lattice $\bbG$ and the random Lipschitz weights be as detailed above.
In the above setup, there exist $ C, c > 0$ such that for any finite $D \subset \bbV$ and $x \in D$ with $d_\bbG(x, D^c) \geq 2$,
\begin{align}
\tag{log-Var-deloc}
\label{eq:log var deloc}
c \log d_D (x, \partial D) \leq \bbE^{ \pm 1 }_D [ h(x)^2 ] \leq C \log d_D (x, \partial D),
\end{align}
where the boundary condition $\pm 1$ is imposed on $\partial D$.
\end{theorem}

%The rest of this section constitutes the proof of this theorem. [MORE ON ORGANIZATION HERE!]

\subsection{Outline for Theorem~\ref{thm:main thm 2nd}}
\label{subsec:outline of main thm}

The core of the proof of Theorem~\ref{thm:main thm 2nd} is, in jargon, to deduce a dichotomy of level-line loop probabilities via a renormalization inequality. We now explain this jargon and show how it implies Theorem~\ref{thm:main thm 2nd}.  

Let $L_n$ denote the discrete square/rhombus with side length $2n$ translational invariance steps.
Let $\calO_{\omega = 1 } (L_n)$ denote the event that there is a loop path surrounding $L_n$ with $\omega = 1$ on the edges, and let
\begin{align}
\label{eq:def of an}
a_n = \bbP_{L_{Rn}}^{ \pm 1 } [\calO_{\omega = 1 } (L_n)],
\end{align}
where $R$ is a constant that we fix to $R=3$ to facilitate the geometric arguments in the proof and the lozenges $L_{n}$ and $L_{Rn}$ are cocentric. The loop probability dichotomy is now stated as follows.
%Let the notation $\calO_{\omega = 1 } (L_n)$ (resp. $\calO^*_{\textsf{big} } (L_n)$) denote the event that there is a loop path surrounding $L_n$ with $\omega = 1$ on the edges (resp. a dual $\bbG^*$ loop path surrounding $L_n$ with $h \geq 5$ on the endpoints of the crossed primal edges and additionally $B_e = 0$ if the crossed primal edges connects two vertices of height $5$). Let
%\begin{align*}
%a_n = \bbP_{L_{Rn}}^{ \pm 1 } [\calO_{\omega = 1 } (L_n)] \qquad \text{and} \qquad 
%b_n = \bbP_{L_{R^3n}}^{ \pm 1 } [\calO_{\omega = 1 } (L_n)],
%\end{align*}
%where $R$ is a constant that we fix to $R=23$ (CHECK) to facilitate the geometric arguments in what follows.

\begin{theorem}[Loop probability dichotomy]
\label{thm:loop dichotomy}
%Fix $1 \leq \mathbf{c} \leq 2$.
There exist $C, c, \alpha > 0$ such that either
\begin{align}
\label{eq:loop-deloc}
\tag{loop-deloc}
a_n & \geq c \qquad \text{for all $n$},
\\
\label{eq:loop-loc}
\tag{loop-loc}
\text{or} \qquad
a_n &\leq C \exp(-c n^\alpha) \qquad \text{for all $n$}.
\end{align}
\end{theorem}

This is indeed equivalent with the earlier dichotomies.

\begin{theorem}[Three equivalent dichotomies]
\label{thm:3 equiv dichotomies}
%Fix $1 \leq \mathbf{c} \leq 2$.
%\begin{itemize}
%\item[i)] If~\eqref{eq:loop-loc} occurs, then also~\eqref{eq:Gibbs-loc} and~\eqref{eq:Var-loc} occur.
%\item[ii)] If~\eqref{eq:loop-deloc} occurs, then~\eqref{eq:log var deloc} (and hence also ~\eqref{eq:Var-deloc} and~\eqref{eq:Gibbs-deloc}) occurs.
%\end{itemize}
i) If~\eqref{eq:loop-loc} occurs, then also~\eqref{eq:Gibbs-loc} and~\eqref{eq:Var-loc} occur. ii) If~\eqref{eq:loop-deloc} occurs, then~\eqref{eq:log var deloc} (and hence also ~\eqref{eq:Var-deloc} and~\eqref{eq:Gibbs-deloc}) occurs.
\end{theorem}

Theorem~\ref{thm:3 equiv dichotomies} is a fairly simple consequence of intermediate steps in the proof of Theorem~\ref{thm:loop dichotomy}; see Section~\ref{sec:proof of main results}. Let us now see how the two immediately imply Theorem~\ref{thm:main thm 2nd}:

\begin{proof}[Theorem~\ref{thm:main thm 2nd}]
By Theorem~\ref{thm:loop dichotomy}, either \eqref{eq:loop-loc} or~\eqref{eq:loop-deloc} occurs. In the first case, by Theorem~\ref{thm:3 equiv dichotomies}, the variance is bounded~\eqref{eq:Var-loc}, and in the second it diverges logarithmically~\eqref{eq:log var deloc}. On the other hand, by Theorem~\ref{thm:Piet's non-quantitative deloc thm}, the variance diverges, so~\eqref{eq:log var deloc} holds. 
\end{proof}

The main open question is thus to prove Theorem~\ref{thm:loop dichotomy}. This is done via the so-called renormalization inequality. For technical reasons, the inequality will be more conveniently derived for a quantity $b_n$ closely related to $a_n$.

\begin{proposition}[Renormalization inequality]
\label{thm:renorm}
There exists $C > 0$ such that every $n$, at least one of the following two holds, where $R'=14$:
\begin{align*}
\text{i)}
 \qquad b_n \geq \tfrac{1}{2C}; \qquad \text{or} \qquad \text{ii)}\qquad 
b_{R' n} \leq C b_n^2.
\end{align*}
\end{proposition}

\begin{remark}
\label{rem:equiv renorm ineqs}
With an adjustment of the constant $C$, the two alternatives can both be shown to imply $b_{R' n} \leq C' b_n^2$ (since $b_{R' n} \leq 1$); the two-case formulation is nevertheless how the result is proven and used, and we hence persist to give this ``suboptimal'' formulation. % in the theorem.
\end{remark}

For completeness, we recall how the renormalization inequality implies the dichotomy of Theorem~\ref{thm:loop dichotomy}, even if this argument is very standard by now.

\begin{proof}[the dichotomy of Theorem~\ref{thm:loop dichotomy} for $b_n$]
Let us first study a subsequence of $n$:s that are powers of $R'$, $n_k = (R')^k$. By Proposition~\ref{thm:renorm}, this subsequence has the property that, if for some $k$ it holds true that
$b_{n_{k}} < \tfrac{1}{\alpha C} $ for some $\alpha > 2$, then also $b_{n_{k + 1}} \leq C b_{n_k}^2 < \tfrac{1}{\alpha^2 C}$ and, using this inductively, $b_{n_{k+\ell}} < \tfrac{1}{\alpha^{2^\ell}C}$.

Now, if $b_{n_k} \geq \tfrac{1}{2C}$ for all $k$, then the analogue of~\eqref{eq:loop-deloc} holds for $b_{n_k} $ along the subsequence $n_k$. If not, there exists $k_0$ such that $b_{n_{k_0}} < \tfrac{1}{2C}$, and by the previous paragraph,
\begin{align*}
b_{n_{k_0 + \ell }} < \tfrac{1}{2^{2^\ell}C}, \qquad \text{for all } \ell \in \bbN.
\end{align*}
For $n = n_{k_0 + \ell}$, one has $\ell = \log_{R'} n - k_0$, and %thus $2^\ell = 2^{-k_0}n^{\log_{R'} 2}$, 
 denoting $c = 2^{-k_0}$ and $\alpha = \log_{R'} 2$, we have
\begin{align*}
b_n < \tfrac{1}{ C} 2^{-cn^{\alpha}}  \qquad \text{for $n$ of the form } n = n_{k_0 + \ell},
\end{align*}
i.e., the analogue of~\eqref{eq:loop-loc} holds for $b_{n_k} $ along the subsequence $n_k$.

It now remains to ``fill the subsequence'' which can be done, e.g., via a standard ``necklace argument'': for any $N \in [n_{k+2}, n_{k+3}]$ and $n \in [n_k, n_{k+1}]$, a loop as in the definition of $b_N$ can be generated by a fixed number $\ell$ of translated loops of the size in $b_{n}$. By the FKG, then, $b_N \geq (b_{n})^{\ell}$. Case (i) for $n_k$ then implies case (i) for all $N$ via the special case $b_N \geq (b_{n_{k+1}})^{\ell}$. Case (ii) for $n_k$ implies case (ii) for all $n$ via $b_n \leq (b_{n_{k+2}})^{1/\ell}$. This concludes the proof.
\end{proof}

The remaining questions are thus to prove Proposition~\ref{thm:renorm} and --- simpler --- Theorem~\ref{thm:3 equiv dichotomies}. This is done in the following four sections.

\section{Crossings in discrete topological quardilaterals}

\subsection{Quadrilaterals}

A \textit{simply-connected (discrete) domain} $D= (V,E)$ of $\bbG$ consists of the vertices and edges that lie inside a simple closed loop $\ell$ on the dual $\bbG^*$.
Given such $D $ and $\ell$, and let $\mathsf{bl}, \mathsf{br}, \mathsf{tr}, \mathsf{tl}$ (standing for bottom-left, bottom-right etc.) be four distinct marked faces on $\ell$, in counterclockwise order. The tuple $(D; \mathsf{bl}, \mathsf{br}, \mathsf{tr}, \mathsf{tl})$ is called a \textit{discrete topological quadrilateral}, or a \textit{quad}, for short (if the marked faces are clear in the context, we will simply refer to ``the quad $D$''). We denote by $E_\mathsf{left}^*$ the dual-edges on the counter-clockwise arc of $\ell$ between $\mathsf{tl}$ and $\mathsf{bl}$, by $E_\mathsf{left}$ the primal-edges (of $\bbG \setminus D$) crossing them, by $V_\mathsf{left}^*$ the dual-vertices on the counterclockwise arc of $\ell$ between or including $\mathsf{tl}$ and $\mathsf{bl}$, and by $V_\mathsf{left}$ or the vertices of $F$ adjacent to $E_\mathsf{left}$ in $\bbG$. The $\mathsf{bottom}$, $\mathsf{right}$, and $\mathsf{top}$ sides are given similar notations. 

For $A, B \subset \bbV$, denote $A \stackrel{\calE}{\longleftrightarrow} B$ if there is a path connecting $A$ and $B$ on $\bbG$ and using only edges of $\calE \subset \bbE$. Analogous connections on the dual $\bbG^*$ are denoted by $\longleftrightarrow_*$.

%We start with two simple but useful properties about crossings.

\begin{lemma}
Let $D = (V, E)$ be a quad on $\bbG$ and $\calE$ a (random) subset of  $E$; then
\begin{align}
\label{eq:quad cross duality}
\{  V_\mathsf{left} \stackrel{\calE}{\longleftrightarrow} V_\mathsf{right} \}^c 
&=
\{  V_\mathsf{bottom}^* \stackrel{ (\calE^c)^* }{\longleftrightarrow}_* V_\mathsf{top}^* \} \qquad \text{and} \\
%\end{align}
%where the second event refers to a dual-path using only dual-edges that cross $\calE^c$.
%A slightly more interesting property is the fact that dual-crossings are easier than primal-crossings, formally
%\begin{align}
\label{eq:quad cross dual vs primal}
\{  V_\mathsf{left} \stackrel{\calE}{\longleftrightarrow} V_\mathsf{right} \}
& \subset
\{  V_\mathsf{left}^* \stackrel{\calE^*}{\longleftrightarrow}_* V_\mathsf{right}^* \}.
\end{align}
\end{lemma}

\begin{proof}
Equation~\eqref{eq:quad cross duality} is a standard duality property. For~\eqref{eq:quad cross dual vs primal}, on  the occurrence of $\{  V_\mathsf{left} \stackrel{\calE}{\longleftrightarrow} V_\mathsf{right} \}$, let $\gamma$ be a path on $\bbG$ with first edge on $E_\mathsf{left}$, last on $E_\mathsf{right}$, and the remaining edges being in $\calE \subset D$ and generating this event.
%On $\{  V_\mathsf{bottom}^* \stackrel{ (\calE^c)^* }{\longleftrightarrow}_* V_\mathsf{top}^* \}$, let $\eta$ be a path on $\bbG^*$, connecting  $V_\mathsf{bottom}^*$  to $V_\mathsf{top}^*$ along $ (\calE^c)^* $ (in particular, inside $\ell$). By topological obstruction, the two paths must cross, so they cannot coexist.
 Analogously, on $\{  V_\mathsf{left}^* \stackrel{\calE^*}{\longleftrightarrow}_* V_\mathsf{right}^* \}^c = \{  V_\mathsf{bottom} \stackrel{\calE^c}{\longleftrightarrow} V_\mathsf{top} \}$ (where we used~\eqref{eq:quad cross duality}), let $\gamma'$ on $\bbG$ have its first edge on $E_\mathsf{bottom}$, last on $E_\mathsf{top}$, and the rest in $\calE^c$ and generating the latter event. We will show that such $\gamma$ and $\gamma'$ cannot co-exist, and hence $\{  V_\mathsf{left} \stackrel{\calE}{\longleftrightarrow} V_\mathsf{right} \}
\subset
\{  V_\mathsf{left}^* \stackrel{\calE^*}{\longleftrightarrow}_* V_\mathsf{right}^* \}^{cc}$, and the proof is complete. If $\gamma$ and $\gamma'$ did co-exist, then by topological obstruction, they must share a vertex in $D$. On the other hand, they are manifestly edge-disjoint and hence also vertex-disjoint, possibly apart from endpoint vertices.\footnote{Note that $\bbG$ being of maximum degree $3$ is crucial here.} But all non-endpoint vertices are in $D$, so they cannot share a vertex in $D$, a contradiction. 
\end{proof}

%Indeed, suppose that $\{  V_\mathsf{left} \stackrel{\calE}{\longleftrightarrow} V_\mathsf{right} \}$ occurs and let $\gamma_0$ be a path on $\bbG$ with first edge on $E_\mathsf{left}$, last on $E_\mathsf{right}$, and the remaining edges being in $\calE$ and generating this event. Analogously, on $\{  V_\mathsf{left}^* \stackrel{\calE^*}{\longleftrightarrow}^* V_\mathsf{right}^* \}^c = \{  V_\mathsf{bottom} \stackrel{\calE^c}{\longleftrightarrow} V_\mathsf{top} \}$, let $\gamma_1$ on $\bbG$ have its first edge on $E_\mathsf{bottom}$, last on $E_\mathsf{top}$, and the rest in $\calE$ and generating this event. Suppose that such $\gamma_0$ and $\gamma_1$ co-existed; by topological obstruction, they must share a vertex in $D$. On the other hand, they are manifestly edge-disjoint and hence also vertex-disjoint apart from possibly endpoint vertices\footnote{Note that $\bbG$ being of degree $3$ is crucial here.}. But all non-endpoint vertices are in $D$, so they cannot share a vertex in $D$, a contradiction. 

\subsection{Crossings of symmetric quads}

Let $\tau$ be a symmetry of the lattice $\bbG$ (e.g., a reflection with respect to the vertical axis). We say that $(D; \mathsf{bl}, \mathsf{br}, \mathsf{tr}, \mathsf{tl})$ is \textit{symmetric with respect to $\tau$}, if $\tau(D)=D$ and furthermore, $\tau$ fixes two of the corner faces, and exchanges the remaining two.
Then, for a boundary condition $\xi: \partial D \to \Zodd$,
%(resp. a height function $h: D \to \Zodd$), 
denote by $\tau. \xi $ the boundary condition $ \xi \circ \tau^{-1}$.
%(resp. by $\tau. h$ the height function $h \circ \tau^{-1}$).
In this situation $2-\xi \succeq \tau. \xi$ means, informally speaking, the boundary condition ``tends to be more below $1$ than above''. The following lemma formalizes this intuition.

Below  an in continuation, we denote by $\{ h \omega \leq 0 \}$ (resp. $\{ h \omega \geq 1\}$) the set of edges $e=\langle u, v \rangle$ on which $h(u)\omega(e), h(v)\omega(e) \leq 0$, i.e., $h(u), h(v) \leq 1$ and if $h(u)=h(v)=1$ then also $B_e = 0$ (resp. its complement, on which $h(u)\omega(e), h(v)\omega(e)  \geq 1$, i.e., $h(u), h(v) \geq 1$ and if $h(u)=h(v)=1$ then also $B_e = 1$). With a slight abuse of notation, we also denote, e.g., $\{ h \omega \leq 0 \}^* \subset \bbE^*$ simply by $\{ h \omega \leq 0 \}$. The edges $\{ h \omega \geq 0 \}$ and their complement $\{ h \omega \leq -1\}$ are defined analogously.
% For brevity, we use the same notations for the corresponding subsets of $\bbE^*$. %Note that $\{ h \omega \leq 0 \}$ is decreasing in $(h, B)$.

\begin{lemma}
\label{lem:sym quad cross}
Let $D$ be symmetric quad with respect to $\tau$ and $\xi$ an interval-valued boundary condition on $\partial D$, such that $2-\xi \succeq \tau. \xi$. Then,
\begin{align*}
\bbP^\xi_D [V_\mathsf{left}^* \stackrel{h \omega \leq 0}{\longleftrightarrow}_* V_\mathsf{right}^*] \geq 1/2.
\end{align*}
\end{lemma}

\begin{proof}
Note that $\{ V_\mathsf{left}^* \stackrel{h \omega \leq 0}{\longleftrightarrow}_* V_\mathsf{right}^* \}^c = \{ V_\mathsf{bottom} \stackrel{h \omega \geq 1}{\longleftrightarrow} V_\mathsf{top} \}$, by~\eqref{eq:quad cross duality}. Combining with the symmetry,
\begin{align*}
1- \bbP^{\xi}_D [ V_\mathsf{left}^* \stackrel{h \omega \leq 0}{\longleftrightarrow}_* V_\mathsf{right}^* ]
&=
 \bbP^{\xi}_D [  V_\mathsf{bottom} \stackrel{h \omega \geq 1}{\longleftrightarrow} V_\mathsf{top}  ] 
\\
&=
 \bbP^{\tau.\xi}_D [   V_\mathsf{left} \stackrel{h \omega \geq 1}{\longleftrightarrow} V_\mathsf{right}  ] .
\end{align*}
Next, a path satisfying $h \omega \geq 1$ equivalently means that $h \geq 1$ on all vertices and, additionally, $B_e=1$ on edges between two heights $1$. The existence of such paths is manifestly an increasing event in $(h, B)$, so by the FKG,
\begin{align*}
 \bbP^{\tau.\xi}_D [  V_\mathsf{left} \stackrel{h \omega \geq 1}{\longleftrightarrow} V_\mathsf{right}  ] \leq  \bbP^{ 2- \xi}_D [   V_\mathsf{left} \stackrel{h \omega \geq 1}{\longleftrightarrow} V_\mathsf{right}  ]. 
\end{align*}

Next, given $(h, B, \omega)$ under $ \bbP^{ 2- \xi}_D $, we define $\tilde{h} = 2-h$ (so that $\tilde{h} \sim \bbP^{\xi}_D$).
% and $B^* = 1-B$ (so that $B^*_e$ are i.i.d. $\Ber (1/c)$). 
Then, define a coupling measure $\bbP$ to also include independent $\Ber (1-1/\mathbf{c}_e)$ variables $(\tilde{B}_e)_{e \in E}$ (i.e., with the same law as $B_e$), each $\tilde{B}_e$ coupled to $B_e$ so that $B_e = 1 \Rightarrow \tilde{B}_e = 0$.\footnote{Note that the assumption $\mathbf{c}_e \leq 2$ is crucial here.}
% (so that $(\tilde{h}, \tilde{B}) \sim \bbP^{\xi}_D \otimes_{e \in E} \Ber(1-1/c)$), 
Finally define $\tilde{\omega}$ from $\tilde{h}$ and $\tilde{B}$ in the usual manner (so that $(\tilde{h}, \tilde{B}, \tilde{\omega}) \sim \bbP^{\xi}_D$). In this coupling $\bbP$, a path of $h \geq 1$ with additionally $B_e=1$ on edges between two heights $1$ is also a path of $\tilde{h} \leq 1$ with additionally $\tilde{B}_e=0$ on edges between two heights $1$, i.e.,
\begin{align*}
\{  V_\mathsf{left} \stackrel{h \omega \geq 1}{\longleftrightarrow} V_\mathsf{right} \} \subset \{  V_\mathsf{left} \stackrel{\tilde{h} \tilde{\omega} \leq 0}{\longleftrightarrow} V_\mathsf{right} \}.
\end{align*}
Finally, we observe that by~\eqref{eq:quad cross dual vs primal}
\begin{align*}
\{  V_\mathsf{left} \stackrel{\tilde{h} \tilde{\omega} \leq 0}{\longleftrightarrow} V_\mathsf{right} \}
\subset
\{  V_\mathsf{left}^* \stackrel{\tilde{h} \tilde{\omega} \leq 0}{\longleftrightarrow}_* V_\mathsf{right}^* \}.
\end{align*}

%indeed, suppose that $\{  V_\mathsf{left} \stackrel{\tilde{h} \tilde{\omega} \leq 0}{\longleftrightarrow} V_\mathsf{right} \}$ occurs and let $\gamma_0$ be a path on $\bbG$ with first edge on $E_\mathsf{left}$, last on $E_\mathsf{right}$, and the remaining edges being inside $D$, having $\tilde{h} \tilde{\omega} \leq 0$, and generating this event. Analogously, on $\{  V_\mathsf{left}^* \stackrel{\tilde{h} \tilde{\omega} \leq 0}{\longleftrightarrow}^* V_\mathsf{right}^* \}^c = \{  V_\mathsf{bottom} \stackrel{\tilde{h} \tilde{\omega} \geq 1}{\longleftrightarrow} V_\mathsf{top} \}$, let $\gamma_1$ on $\bbG$ have its first edge on $E_\mathsf{bottom}$, last on $E_\mathsf{top}$, and the rest in $D$ with $\tilde{h} \tilde{\omega} \geq 1$, and generating this event. Suppose that such $\gamma_0$ and $\gamma_1$ co-existed; by topological obstruction, they must share a vertex in $D$. On the other hand, they are manifestly edge-disjoint and hence also vertex-disjoint apart from possibly endpoint vertices\footnote{Note that $\bbG$ being of degree $3$ is crucial here.}. But all non-endpoint vertices are in $D$, so they cannot share a vertex in $D$, a contradiction. 
%
With the two previous set inclusions and the observation $(\tilde{h}, \tilde{B}, \tilde{\omega})  \sim \bbP^{\xi}_D$, we have
\begin{align*}
\bbP^{ 2- \xi}_D [  V_\mathsf{left} \stackrel{h \omega \geq 1}{\longleftrightarrow} V_\mathsf{right}  ]
= 
\bbP [   V_\mathsf{left} \stackrel{h \omega \geq 1}{\longleftrightarrow} V_\mathsf{right}  ]
\leq
\bbP [  V_\mathsf{left}^* \stackrel{\tilde{h} \tilde{\omega} \leq 0}{\longleftrightarrow}_* V_\mathsf{right}^*  ]
=
\bbP^\xi_D [  V_\mathsf{left}^* \stackrel{h \omega \leq 0}{\longleftrightarrow}_* V_\mathsf{right}^*  ].
\end{align*}
Combining all the displayed (non-set) inequalities now yields the claim.
\end{proof}

\subsection{Comparing symmetric quads to narrower ones}

%For the results of the present subsection, we present two variants: when the maximal degree of the dual $\bbG^*$ is (i) at most $7$; or (ii) at most $11$. The former is easier, but it essentially only applies for the honeycomb lattice.

%Let $Q$ and $D$ be two simply-connected domains of $\bbG$.
Note that the intersection graph $Q \cap D$ of two simply-connected domains $ Q$ and $D$ consists of connected components $C$ that are themselves simply-connected discrete domains, each with a bounding dual-loop $\ell(C)$ consisting of dual-edges on $\ell(Q)$ or $\ell (D)$. Suppose now that $Q$ has a quad structure, that $C$ contains a vertical primal crossing of $Q$, and that $\ell(C) \setminus \ell(D) \subset E^*_{\mathsf{top}}(Q) \cup E^*_{\mathsf{bot}}(Q)$. Then, exactly two segments of $\ell(C) \cap \ell(D)$ cross $Q$ vertically (potentially overlapping with $E^*_{\mathsf{left}}(Q) \cup E^*_{\mathsf{right}}(Q)$), one from bottom to top and one from top to bottom (orienting $\ell (C)$ counter-clockwise). We call these the \textit{left and right walls} of $C$, respectively.
%, and denote by $\overline{C}$ the discrete sub-domain of $Q$ between the left and right walls (note that potentially $\overline{C} \supsetneq C$).
Finally, if also $D$ has a quad structure, we say that $D$ is \textit{narrower} than $Q$, if there exists such a $C$ and furthermore $\ell(C) \setminus \ell(Q) \subset E^*_{\mathsf{left}}(D) \cup E^*_{\mathsf{right}}(D)$\footnote{
Equivalently, every dual-edge on $\ell(C)$ belongs to $E^*_{\mathsf{top}}(Q) \cup E^*_{\mathsf{bottom}}(Q)$ or $E^*_{\mathsf{left}}(D) \cup E^*_{\mathsf{right}}(D)$.
} and the left (resp. right) wall of $C$ is contained in $E^*_{\mathsf{left}}(D)$ (resp. $E^*_{\mathsf{right}}(D)$). 

\begin{lemma}
\label{lem:narrow quad cross}
Let $D$ be a quad that is narrower than a symmetric quad $Q$
%, and let $C$ be a fixed component of $D \cap Q$ inducing this narrower structure.
and $\xi$ an interval-valued boundary condition on $\partial D$ which is $ \{ \pm 1 \}$ on $V_{\mathsf{left}}(D) \cup V_{\mathsf{right}}(D)$ and $\xi \preceq \{ 1, 3 \}$ on the rest of $\partial D$.
%\begin{itemize}[noitemsep]
%\item[i)] $\xi = \{ 1, 3 \}$ on the rest of $\partial D$; or
%\item[ii)] $\xi \preceq \{ 5, 7 \}$ on the rest of $\partial D$; or
%\end{itemize}
 Then, we have
\begin{align*}
\bbP^\xi_D [  V_\mathsf{left}^*(D) \stackrel{h \omega \leq 0}{\longleftrightarrow}_* V_\mathsf{right}^*(D)  ] \geq \tfrac{1}{2}.
\end{align*}
\end{lemma}

\begin{figure}
\begin{center}
\includegraphics[width=0.4\textwidth]{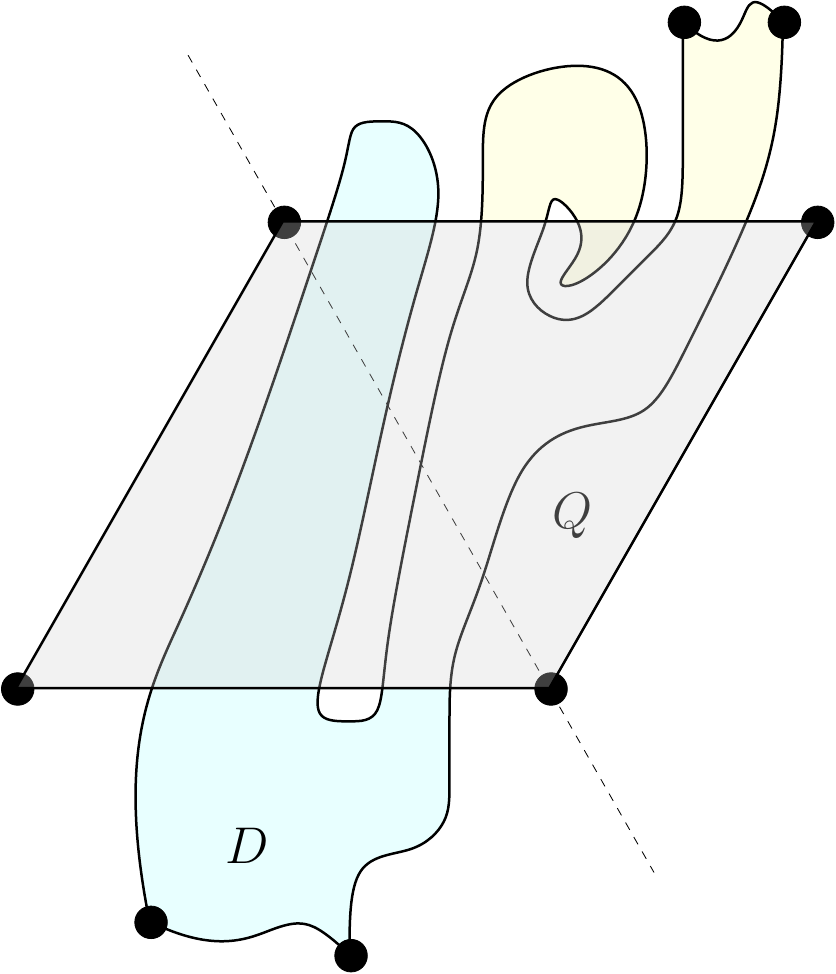}
\end{center}
\caption{
\label{fig:plus domain}
A schematic illustration for the proof of Lemma~\ref{lem:narrow quad cross}. The discretized lozenge $Q$ is a symmetric quad with respect to the reflection $\tau$ around its diagonal line (dashed). The present quad $D$ is narrower than $Q$, and there is a unique component $C$ of $D \cap Q$ inducing this structure (uncoloured). The components of $D\setminus C$ are adjacent to $C$ either via the top of $Q$ ($D_{\mathsf{top}}$, yellow), or via the bottom of $Q$ ($D_{\mathsf{bottom}}$, blue). The glueing $Q \bigcupplus D_{\mathsf{top}} \bigcupplus D_{\mathsf{bottom}} \bigcupplus \tau(D_{\mathsf{top}}) \bigcupplus \tau(D_{\mathsf{bottom}})$ will produce a ``discrete Riemann surface'' which is a symmetric quad (with the same corners as $Q$).
}
\end{figure}

\begin{proof}
 Note first that $\{ V_\mathsf{left}^*(D) \stackrel{h \omega \leq 0}{\longleftrightarrow}_* V_\mathsf{right}^*(D) \}$ is a decreasing event in $(h, B)$. It thus suffices to prove the claim assuming that $\xi = \{ 1, 3 \}$ on $\partial D \setminus V_{\mathsf{left}}(D) \cup V_{\mathsf{right}}(D)$. With this boundary condition, we have $\{ V_\mathsf{left}^* (D) \stackrel{h \omega \leq 0}{\longleftrightarrow}_* V_\mathsf{right}^* (D) \} = \{ V_\mathsf{left}^* (D) \stackrel{\omega = 0}{\longleftrightarrow}_* V_\mathsf{right}^* (D) \}$, which is is manifestly a decreasing event in $(|h|, \omega)$.

We will now aim to use the FKG for $(|h|, \omega)$.
%Let $C$ be a subdomain-crossing of $Q$ by $D$ that implies $D$ being narrower.
Let $D_{\mathsf{top}}$ (resp. $D_{\mathsf{bottom}}$) be the components of $D \setminus C$ that are adjacent to $C$ via the top (resp. bottom) of $Q$ (so $D\setminus C$ consists of $D_{\mathsf{top}}$ and $D_{\mathsf{bottom}}$); see Figure~\ref{fig:plus domain}.
Note that the connected components of $D \setminus C$ are discrete simply-connected domains, and every component of $D_{\mathsf{top}}$ contains exactly one sub-segment of $\ell(C) \cap E^*_{\mathsf{top}} (Q)$ on its bounding dual loop; we let $Q \bigcupplus D_{\mathsf{top}}$ denote the operation of gluing the graph $D_{\mathsf{top}}$ to the graph $Q$ along these shared boundary segments on $ E^*_{\mathsf{top}} (Q)$. Hence the graph $Q \bigcupplus D_{\mathsf{top}}$ is planar and has locally a ``honeycomb structure'', but it may not be embeddable to the plane so that this honeycomb structure remains embedded as regular hexagons. In what follows, it is however more beneficial to think of $Q \bigcupplus D_{\mathsf{top}}$ as a ``discrete Riemann surface'' and embed it in a ``honeycomb manner''. Define via  analogous gluing
\begin{align*}
D' = Q \bigcupplus D_{\mathsf{top}} \bigcupplus D_{\mathsf{bottom}} \bigcupplus \tau(D_{\mathsf{top}}) \bigcupplus \tau(D_{\mathsf{bottom}}).
\end{align*}
With a slight abuse of notation, $\tau$ can be seen as a graph isomorphism from $D'$ to itself. Give $D'$ a quad structure by marking the corner faces of $Q$, and let $\xi'$ be $\xi' = \{ \pm 1 \}$ on $V_{\mathsf{left}}(D') \cup V_{\mathsf{right}}(D')$ and $\xi' = \{ 1, 3 \}$ on the rest of $\partial D'$. Similarly as above, with this boundary condition, $$\{ V_\mathsf{left}^*(D') \stackrel{h \omega \leq 0}{\longleftrightarrow}_* V_\mathsf{right}^*(D') \}
=
\{ V_\mathsf{left}^* (D') \stackrel{\omega = 0}{\longleftrightarrow}_* V_\mathsf{right}^* (D') \}.$$

Lemma~\ref{lem:sym quad cross} (actually, a slight extension with $\tau$ being a graph isomorphism rather than a symmetry of $\bbG$) yields 
\begin{align*}
\bbP^{\xi'}_{D'} [ V_\mathsf{left}^*(D') \stackrel{h \omega \leq 0}{\longleftrightarrow}_* V_\mathsf{right}^*(D')  ] \geq \tfrac{1}{2},
\end{align*}
and the FKG and CBC for $(|h|, \omega)$ yield
\begin{align*}
\bbP^{\xi}_{D} [ V_\mathsf{left}^* (D) \stackrel{\omega = 0}{\longleftrightarrow}_* V_\mathsf{right}^* (D) ] \geq \bbP^{\xi'}_{D'} [V_\mathsf{left}^* (D') \stackrel{\omega = 0}{\longleftrightarrow}_* V_\mathsf{right}^* (D') ] .
\end{align*}
Combining the three previously displayed equations proves the claim.
\end{proof}

We will need another similar result, however stronger in the sense that the crossings to be found are \textit{inside the small domain} $C$; the price to pay for this extra information is the assumption that the domain $C$ is \textit{on the bottom} of $D$, in the sense that $E^*_{\mathsf{bottom}}(D) \subset \ell(C)$ (or \textit{on the top}, defined analogously). The proof is similar to Lemma~\ref{lem:sym quad cross} but a bit more involved; we postpone it to Appendix~\ref{app:roskis}.
%  If $\ell(C) \cap E^*_{\mathsf{bot}}(Q) \subset E^*_{\mathsf{bot}}(D)$, we say that $C$ is on the bottom of $D$.

\begin{lemma}
\label{lem:narrow quad cross 2}
%Let the maximal degree of $\bbG^*$ be (i) at most $7$; or (ii) at most $11$.
Let $D$ be a quad that is narrower than a symmetric quad $Q$, with a component $C$ inducing this narrower structure being on the bottom of $D$; let $\xi$ be an interval-valued boundary condition on $\partial D$ which is $ \{ \pm 1 \}$ on $V_{\mathsf{left}}(D) \cup V_{\mathsf{right}}(D)$ and $\xi \preceq \{ 5, 7 \}$ on the rest of $\partial D$, and denote by $\calE(h, B)$ the set of edges between heights $h \leq 3$ with additionally $B_e=0$ for edges between two heights $3$. We then have
%; or (ii) the set of edges $\langle u, v \rangle$ $\min \{ h(u), h(v) \} \leq 3$. Then, in the respective cases (i)--(ii), we have
\begin{align*}
\bbP^\xi_D [V_\mathsf{left}^* (D) \stackrel{\calE(h, B)^* \cap C}{\longleftrightarrow}_* V_\mathsf{right}^* (D)] \geq \tfrac{1}{2}.
\end{align*}
\end{lemma}

The main result of this section is the following.
%(This result is more useful than the previous lemmas since the height gap is so large that dual-crossings with such height gaps yield absolute-value events with the bound on maximal dual degree.)

\begin{corollary}
\label{cor:narrow quad triple cross}
Let $D$ be a quad that is narrower than three symmetric quads $Q_1$, $Q_2$, $Q_3$, furthermore so that the respective connected components $C_i \subset Q_i \cap D$, $i \in \{ 1, 2, 3\}$, that generate the narrower structure can be chosen disjoint and $C_1$ on the bottom and $C_3$ on the top of $D$.
Let $\xi$ be an interval-valued boundary condition on $\partial D$ which is $ \{ \pm 1 \}$ on $V_{\mathsf{left}}(D) \cup V_{\mathsf{right}}(D)$ and $\xi \preceq \{ 5, 7 \}$ on the rest. Then, we have
\begin{align*}
\bbP^\xi_D [
V_\mathsf{left}^* (D) \stackrel{h \omega \leq 0}{\longleftrightarrow}_* V_\mathsf{right}^* (D) ] \geq \tfrac{1}{8}.
\end{align*}
\end{corollary}

\begin{proof}
Again, it suffices to consider the maximal $\xi$ allowed in the statement.
Define $\calE(h, B)$ as in Lemma~\ref{lem:sym quad cross 2}; by that lemma,
\begin{align*}
\bbP^\xi_D [V_\mathsf{left}^* (D) \stackrel{\calE(h, B)^* \cap C_1}{\longleftrightarrow}_* V_\mathsf{right}^* (D)], \;
\bbP^\xi_D [V_\mathsf{left}^* (D) \stackrel{\calE(h, B)^* \cap C_3}{\longleftrightarrow}_* V_\mathsf{right}^* (D)] \geq \tfrac{1}{2}.
\end{align*}
Denote the two events on the left by $\calH_1$ and $\calH_3$, respectively. Since both are decreasing in $(h, B)$, FKG gives
\begin{align*}
\bbP^\xi_D [\calH_1 \cap \calH_3] \geq 1/4,
\end{align*}
and it thus suffices to show that 
\begin{align*}
\bbP^\xi_D [ V_\mathsf{left}^* (D) \stackrel{h \omega \leq 0}{\longleftrightarrow}_* V_\mathsf{right}^* (D) \; | \; \calH_1 \cap \calH_3] \geq 1/2.
\end{align*}
For this purpose, set $\tilde{h} = 2- h$ and define $\tilde{\omega}$ via $\tilde{h}$ and $B_e$ in the usual manner; note that $\calE(h, B) = \{ \tilde{h} \tilde{ \omega} \geq 0 \}$. To reveal whether $\calH_1 \cap \calH_3$ occurs, we explore the cluster of $\tilde{h} \tilde{\omega} < 0$ in $C_1$ (resp. $C_3$) adjacent to the bottom (resp. top) boundary of $D$: formally, let $V_1$ consist of $V_\mathsf{bottom}$ and the vertices of $D$ connected to $V_\mathsf{bottom}$ by edges of $\tilde{h} \tilde{\omega} < 0$, and let $E_1$ consist of these primal-edges and the primal-edges crossing the dual-paths bounding the former. Define $V_3$ and $E_3$ analogously via $C_3$. Now, the event $V_1 = A_1, E_1 = B_1$ is determined by $\tilde{h}_{|A_1}, \tilde{\omega}_{|B_1}$; let $\Xi$ be the set of $(A_1, B_1, A_3, B_3, \alpha_1, \nu_1, \alpha_3, \nu_3)$ such that $\{ \tilde{h}_{|A_1} = \alpha_1, \tilde{\omega}_{|B_1} = \nu_1, \tilde{h}_{|A_3} = \alpha_3, \tilde{\omega}_{|B_3} = \nu_3 \}$ induces $\calH_1 \cap \calH_3$ with $V_1 = A_1, E_1 = B_1, V_3 = A_3, E_1 = B_3$. Then,
\begin{align*}
\bbP^\xi_D & [ \calH_1 \cap \calH_3] = \sum_{\Xi} \bbP^\xi_D [ \tilde{h}_{|A_1} = \alpha_1, \tilde{\omega}_{|B_1} = \nu_1, \tilde{h}_{|A_3} = \alpha_3, \tilde{\omega}_{|B_3} = \nu_3 ] \\
\bbP^\xi_D &[ V_\mathsf{left}^* (D) \stackrel{h \omega \leq 0}{\longleftrightarrow}_* V_\mathsf{right}^* (D) \cap \calH_1 \cap \calH_3]\\
&=
 \sum_{\Xi} \bbP^\xi_D [ \tilde{h}_{|A_1} = \alpha_1, \tilde{\omega}_{|B_1} = \nu_1, \tilde{h}_{|A_3} = \alpha_3, \tilde{\omega}_{|B_3} = \nu_3 ]
 \\
 & \qquad \bbP^\xi_D [ V_\mathsf{left}^* (D) \stackrel{h \omega \leq 0}{\longleftrightarrow}_* V_\mathsf{right}^* (D) \;| \; \tilde{h}_{|A_1} = \alpha_1, \tilde{\omega}_{|B_1} = \nu_1, \tilde{h}_{|A_3} = \alpha_3, \tilde{\omega}_{|B_3} = \nu_3 ],
\end{align*}
and it suffices to show that the last conditional probability is at least $1/2$ for all  $(A_1, B_1, A_3,$ $ B_3, \alpha_1, \nu_1, \alpha_3, \nu_3)$. Note that the component of unexplored vertices $V \setminus A_1 \setminus A_3$ that contains $C_2$ is a simply-connected discrete domain $D'$ restricted by $\ell(D)$ and the dual-paths bounding the $\tilde{h} \tilde{\omega} < 0$ clusters. By SMP (Lemma~\ref{lem:SMP with percolation}), the conditional measure $\bbP^\xi_D [ \cdot \;| \; \tilde{h}_{|A_1} = \alpha_1, \tilde{\omega}_{|B_1} = \nu_1, \tilde{h}_{|A_3} = \alpha_3, \tilde{\omega}_{|B_3} = \nu_3 ]$ restricted to this component is a random Lipschitz model, with the boundary condition $\xi'$ given by $h \in \{ \pm 1\}$ on $E^*_{\mathsf{left}}(D)$ and $E^*_{\mathsf{right}}(D)$, $h \in \{ 1, 3\}$ on the unexplored tips of $\tilde{ \omega} = 0$ edges, and given $\alpha_1$ and $\alpha_3$ the top and bottom vertices of $V$ from which the exploration of $\tilde{h} \tilde{\omega} < 0$ never proceeded forward (whence  $\xi' \leq 1 $ on these vertices). This boundary condition is suited for the application of Lemma~\ref{lem:sym quad cross}, giving
\begin{align*}
\bbP^{\xi'}_{D'} [V_\mathsf{left}^* (D) \stackrel{h \omega \leq 0}{\longleftrightarrow}_* V_\mathsf{right}^* (D)] \geq 1/2.
\end{align*}
This concludes the proof.
\end{proof}

\section{The Pushing Theorem}

\subsection{Statement}

\begin{figure}
\begin{center}
\includegraphics[width=0.5\textwidth]{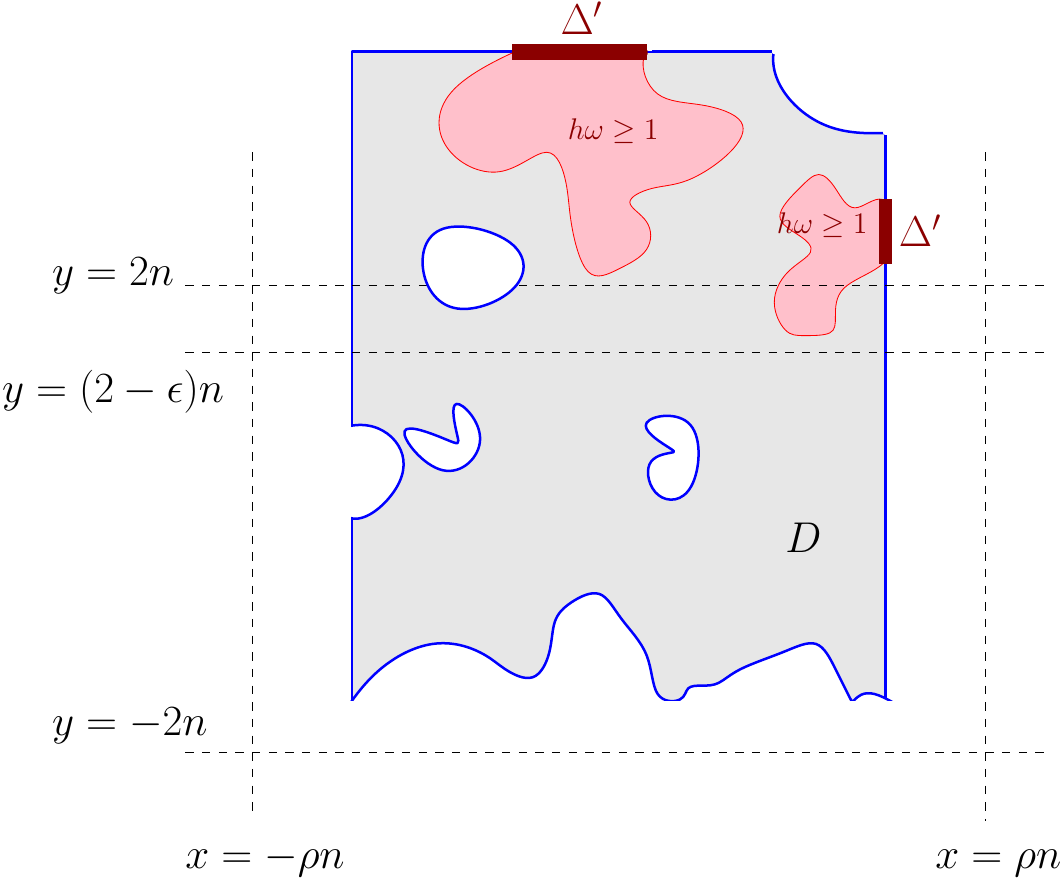}
\end{center}
\caption{
\label{fig:pushing}
Illustration for Proposition~\ref{prop:push}, through which Theorem~\ref{prop:pushing} is typically applied.
}
\end{figure}

We will next use the crossing probabilities in symmetric quads to derive a fairly general uniformly positive crossing probability for discrete rectangles. 

A \textit{(discrete) rectangle} $R_{n,m}$ is a discrete topological quadrilateral, centered at the origin face and extending $n$ horizontal (resp. $m$ vertical) translational periods of $\bbG$ to both left and right (resp. both up and down), and invariant under reflection around the horizontal axis. For an explicit definition the honeycomb lattice case, suppose that $\bbG$ is embedded in the complex plane so that the normal vectors of the sides of a hexagon are the sixth roots of unity and the origin is in a middle of an origin face. Scale the embedding to side length $1/3$ for the hexagons; hence the horizontal lines $y \in \bbZ$ cross the midpoints of every second a row of faces (depict the vertical translational period), and the distance of two neighbouring midpoints (the horizontal translational period) is $1/\sqrt{3}$.
Then let $R_{n,m}$ denote the rectangle consisting of faces with midpoints in $[-1/\sqrt{3} \cdot n, 1/\sqrt{3} \cdot n]\times[-m,m]$, thus of width $2n+1$ hexagons and height $4m+1$ rows. For another example, on the square-octagon lattice with origin in the middle of and octagonal face and a translational symmetry $\bbZ^2$, $R_{n,m}$ would consist of the faces with midpoints in $[- n, n]\times[-m,m]$.
%, and equip this set of faces with the natural structure of a quad.

Horizontal and vertical crossings of a rectangle $R$ by edge set $\calE$ are denoted $\calH_\calE (R)$ and $\calV_\calE (R)$, respectively, and dual-crossings by $\calH^*_\calE (R)$ and $\calV^*_\calE (R)$.

\begin{theorem}
\label{prop:pushing}
For any $\rho > 0$, there exists $c= c(\rho)$ such that the following holds for any $n$. Let $D=(V,E)$ be a connected subgraph of $\bbG$, generated by its vertices and invariant under reflection $\tau$ with respect to the horizontal axis.
Let $\xi \preceq \{ 5, 7 \}$ be an interval-valued boundary condition on $\partial D$,
%, taking only the values $\{ \pm 1 \}$ and $\{ 1, 3\}$, and 
such that 
 $2-\xi \succeq \tau.\xi $ and $\xi \preceq \{ \pm 1\}$ on $\partial D \cap  R_{\rho n, n}$. Then, we have
\begin{align*}
\bbP_D^{\xi} [\calV_{h \omega \geq 1} ( R_{\rho n, n} )] 
% = \bbP_D^{\xi} [\calH^*_{\omega = 0} ( \Lambda_{n, m} )] 
\leq 1 - c.
\end{align*}
\end{theorem}

\subsection{Consequences}

Let us start with an introductory example that we don't formulate as a separate corollary.
Set $D = R_{\rho n, 3n}$, $R = R_{\rho n, n}$,
and let $\xi$ be set-valued $\{ 1, 3 \}$ on the top boundary of $D$ and $\{ \pm 1 \}$ on the side and bottom boundaries. The above theorem then applies, so $ \bbP_D^{\xi} [\calH^*_{h \omega \leq 0} (R)] > c$. Note also that due to the boundary condition,
\begin{align*}
 \calH^*_{h \omega \leq 0} (R) 
\subset
\calH^*_{h \omega \leq 0} (R_+) 
=
 \calH^*_{\omega = 0} (R_+),
\end{align*}
where $R_+$ denotes the two-thirds of $D$ that are in or above $R$. 
%we denoted $R_+ :=[ - 1/\sqrt{3} n , 1/\sqrt{3} n]\times[-m, + \infty] \cap D$. In the case $n'=n$, the latter event can be detected b
Now, to detect the event $\calH^*_{\omega = 0} (R_+)$, explore the primal-clusters of $\omega$ adjacent to below $R_+$, and their top boundary, i.e., the lowest dual-crossing generating $ \calH^*_{\omega = 0} (R_+)$. By SMP (Lemma~\ref{lem:SMP with percolation}) this ``pushes the $\{ \pm 1 \}$ boundary condition above the bottom of $R_+$'', hence the name of the theorem.

Let $\{ y \leq m\}$ denote the set of vertices with vertical coordinate at most $m$ vertical translational periods (generalizing the case of the honeycomb lattice, where the vertical period is one length unit). The idea of the above example holds in a larger generality (see Figure~\ref{fig:pushing}):

\begin{proposition}
\label{prop:push}
For every $\rho, \epsilon > 0$ and $k \in \bbN$, there exists $c=c(\rho, \epsilon, k)$ such that the following holds for all $n > k/\epsilon $.
Let $D$ be a connected subgraph of $\bbG$, generated by its vertices and contained entirely between the vertical lines $x=\pm \rho n$ horizontal periods and entirely above the horizontal line $y = -2n$ vertical periods. Let $\Delta' \subset \partial D$ be a set of boundary vertices entirely above the horizontal line $y=2n$
% $R = \Lambda_{\rho n, n}$, 
and let $\xi$ be an interval-valued boundary condition on $\partial D$ taking values  $\preceq \{ 2k \pm 1 \}$ on $\Delta'$ and $\xi \preceq \{ \pm 1 \}$ on $ \partial D \setminus \Delta'$. Then,
\begin{align*}
\bbP_D^\xi [ \Delta' \stackrel{h \omega \geq 1}{\longleftrightarrow} \{ y \leq (2-\epsilon)n\} ]
% = \bbP_D^\xi [ \Delta' \stackrel{\omega = 1}{\longleftrightarrow} \{ y=(2-\epsilon)n\} ] 
\leq 1-c.
\end{align*}
\end{proposition}

We work towards the proof of this result via intermediate steps.

\begin{lemma}
For every $\rho$, there exists $c=c(\rho)$ such that the following holds for all $n$.
Let $D$ and $\Delta'$ be as above, and let
% $R = \Lambda_{\rho n, n}$, and let
$\xi$ be an interval-valued boundary condition on $\partial D$, being $\preceq \{ 1, 3\}$ on $\Delta'$ and $\xi \preceq \{ \pm 1 \}$ on $ \partial D \setminus \Delta'$. Then,
\begin{align*}
\bbP_D^\xi [ \Delta' \stackrel{h \omega \geq 1}{\longleftrightarrow} \{ y \leq -n\} ]
% = \bbP_D^\xi [ \Delta' \stackrel{\omega = 1}{\longleftrightarrow} \{ y=-n\} ] 
\leq 1-c.
\end{align*}
\end{lemma}

\begin{proof}
%The equality part follows since the sign of $h$ is determined to be positive on any cluster of $\omega = 1$ adjacent to $B$. We thus prove the inequality.
%
Let us start by easy simplifying observations. First, if $D$ does not contain any vertical crossing of $R$, then the claim is trivial, so we assume it does. Second, the event $\{ \Delta' \stackrel{h \omega \geq 1}{\longleftrightarrow} \{ y \leq -n\} \}$ is increasing in $(h, B)$, so it suffices to prove the claim for the maximal boundary condition $\xi$ allowed by the assumptions, i.e., we may assume $\xi_{|\Delta'} = \{1, 3 \}$ and $\xi_{|\partial D \setminus \Delta'} = \{ \pm 1 \}$.  

%In particular, then
%\begin{align*}
%\bbP_D^\xi [ \calV_{\omega \geq 1} (R) ] & \leq \bbP_D^\xi [ \calV_{h \omega \geq 1} (R) ] + \bbP_D^\xi [ \calV_{h \omega \leq -1} (R) ] \\
%\leq 2 \bbP_D^\xi [ \calV_{h \omega \geq 1} (R) ],
%\end{align*}
%where we used the Union bound in the first step, and in the second step the fact that given $|h|$ (and the absolute-value condition $\xi$), the conditional law of the signs of $h$ is a certain ferromagnetic Ising model, so the vertical plus-crossing is more probable than a negative one.
%Note that
%\begin{align*}
%\bbP_D^\xi [ \calV_{h \omega \geq 1} (R) ] \leq \bbP_D^\chi [ \calV_{h \omega \geq 1} (R) ]
%\end{align*}
%since $\calV_{h \omega \geq 1} (R) $ is increasing in $(h, B)$. It thus suffices to prove the claim for the boundary condition $\chi$.

Let $\tau$ now denote the reflection with respect to the horizontal axis. By the first simplifying assumption, $D' := D \cup \tau(D)$ is a connected subgraph of $\bbG$. Furthermore, $\tau(D)$ remains below the line $y=2n$, so $\Delta' \subset \partial D'$. Let $\xi'$ be the boundary condition on $\partial D'$ given by $\xi'_{|\Delta'} = \{1, 3 \}$ and $\xi'_{|\partial D' \setminus \Delta'} = \{ \pm 1 \}$. Next, we compute
\begin{align*}
\bbP_D^\xi [ \Delta' \stackrel{h \omega \geq 1}{\longleftrightarrow} \{ y \leq -n\} ] 
&=
\bbP_D^\xi [ \Delta' \stackrel{\omega = 1}{\longleftrightarrow} \{ y \leq -n\} ] \qquad \text{(due to boundary condition)}
\\
& \leq 
\bbP_{D'}^{\xi'} [ \Delta' \stackrel{\omega = 1}{\longleftrightarrow} \{ y \leq -n\} ] \qquad \text{(by FKG for $(|h|, \omega)$)} \\
& = \bbP_{D'}^{\xi'} [ \Delta' \stackrel{h \omega \geq 1}{\longleftrightarrow} \{ y \leq -n\} ]\qquad \text{(due to boundary condition)} \\
& \leq \bbP_{D'}^{\xi'} [ \calV_{h \omega \geq 1} (R_{\rho n, n}) ] \qquad \text{(by geometric setup)}  \\
& \leq 1- c(\rho) \qquad \text{(by Theorem~\ref{prop:pushing})}.
\end{align*}
This concludes the proof.
\end{proof}

\begin{lemma}
For every $\rho, \epsilon > 0$, there exists $c=c(\rho, \epsilon)$ such that the following holds for all $n > 1/\epsilon$.
Let $D$, $\Delta'$, and $\xi$ be as in the previous lemma. Then,
\begin{align*}
\bbP_D^\xi [ \Delta' \stackrel{h \omega \geq 1}{\longleftrightarrow} \{ y \leq (2 - \epsilon )n\} ] 
%= \bbP_D^\xi [ \Delta' \stackrel{\omega = 1}{\longleftrightarrow} \{ y=(2 - \epsilon )n\} ]
 \leq 1-c.
\end{align*}
\end{lemma}

\begin{proof}
%Since $\{ B \stackrel{h \omega = 1}{\longleftrightarrow} \{ y= (2 - \epsilon )n\} \}$ is increasing in $(h, B)$, it suffices to prove the claim when the boundary set $B$ is $\partial D \cap \{ y > 2n \}$.
In the previous lemma, we saw that
\begin{align*}
\bbP_D^\xi [ \Delta' \stackrel{h \omega \geq 1}{\longleftrightarrow} \{ y \leq -n\} ] \leq 1-c'(\rho).
\end{align*}
Let us thus reveal $h$ below the horizontal line $y=-n$ and then explore the primal-component of $h \omega \geq 1$ adjacent to below this line, and its bounding dual-paths of $h \omega \leq 0$. (For the formalization of a similar argument, see proof of Corollary~\ref{cor:narrow quad triple cross}.) With probability at least $c'(\rho)$, this primal-component will be disconnected from $\Delta'$.
%by dual-paths of $h \omega \leq 0$ (if $B$ is not connected, there may be several dual-paths for different components of $B$).
Let $D'$ be the component(s) of the unexplored domain that contains $\Delta'$.
% that is above these dual-paths. 
By the SMP, this exploration induces a boundary condition $\xi'$ on $\partial D'$ which coincides with $\xi$ on $ \partial D' \cap \partial D$, is $\{ \pm 1 \}$ on the unexplored tips of $\omega = 0$ edges, and negative on the vertices adjacent to the line $y=-n$ from which the exploration never proceeded. Thus $D'$ and $\xi'$ are also adapted to the application of the previous lemma (after shifting the origin suitably), but the entire $D'$ is now above the line $y=-n$, i.e., the ``vertical extent of the domain has been reduced from $4n$ to $3n$''. By a bounded number of iterated applications, we reduce the height by any power of $3/4$, eventually to $\epsilon n$ vertical periods (which makes sense since $n > 1/\epsilon$).
\end{proof}

\begin{proof}[Proposition~\ref{prop:push}]
Note first that by CBC for $(h,B)$, it suffices to prove the claim for the maximal $\xi$ allowed.
The previous lemma gives the desired result directly in the case $k=1$. Assume now inductively that the claim holds for $k$ and prove it for $k+1$, i.e., when, e.g., $\xi \preceq \{ 2k +3, 2k + 1\}$. Let $\hat{h} = h-2k$, $\hat{\xi} = \xi - 2k$ and define $\hat{\omega}$ based on $\hat{h}$ and $B$ in the usual manner. Now, $D$, $\Delta'$, and $\hat{\xi}$ are suited for the application of the previous lemma \textit{with} $c(\rho, \epsilon/k)$, giving
\begin{align*}
\bbP_D^{\hat{\xi}} [ \Delta' \stackrel{\hat{h} \hat{\omega} \geq 1}{\longleftrightarrow} \{ y \leq (2 - \epsilon/k )n\} ] \leq 1-c(\rho, \epsilon/k).
\end{align*}

Let us then explore the components of $\hat{h} \hat{\omega} \geq 1$ adjacent to $\Delta'$ (on which $\hat{\xi} = \{ 1, 3\}$). On the complement of the event $\{ \Delta' \stackrel{\hat{h} \hat{\omega} \geq 1}{\longleftrightarrow} \{ y \leq (2 - \epsilon/k )n\} \}$ above, these components will be restricted above dual-paths of $\hat{\omega}= 0$ that remain above the line $y = (2 - \epsilon/k)n$. By the SMP, the exploration induces a boundary condition $\xi'$ on the boundary of the unexplored domain $D'$, where $\xi'$ coincides with $\xi$ on $ \partial D' \cap \partial D$ and is of the form $\hat{h} \in \{ \pm 1 \}$ (equivalently, $h \in \{ 2k \pm 1\}$) on the dual curves of $\hat{\omega}= 0$. In particular, $\xi' $ is not $\{ \pm 1 \}$ on a set $\Delta'' \subset \partial D'$ consisting of $\Delta' \cap \partial D'$ and the bounding dual curves above. Now, $D'$, $\Delta''$ and $\xi'$ are suited for the application of the present lemma with parameters $\rho, \epsilon, k $ (which was inductively supposed to hold), giving
\begin{align*}
\bbP_{D'}^{\xi'} [ \Delta'' \stackrel{h \omega \geq 1}{\longleftrightarrow} \{ y \leq (2 - \epsilon )n\} ] \leq 1-c(\rho, \epsilon, k).
\end{align*}
Averaging over $D'$ and $\xi'$ concludes the proof.
\end{proof}

\subsection{Proof of Theorem~\ref{prop:push}}

By CBC for $(h, B)$ it suffices to prove the claim when $\xi_{| R_{ \rho n, n}} = \{ \pm 1 \}$.
Also, without loss of generality, we will assume that $\rho$ is an even integer and 
$n $ is divisible by $3$. 
Denote for the rest of this proof $\bbP^\xi_D =: \bbP$, $R=R_{ \rho n, n}$ and
\begin{align*}
\bbP [ \calV_{h \omega \geq 1} (R) ] = 1-p.
\end{align*}
 The strategy will be to show that there exists a $c=c(\rho)$ such that $p \leq c$ will lead to a contradiction; in the proof we thus imagine $p$ as a very small positive number. The proof is divided into steps and lemmas.
 
\subsubsection{Many geometrically controlled crossings of alternating signs}
 
The first lemma shows that also vertical crossings ``of opposite sign'' must be likely.
\begin{lemma}
We have
\begin{align}
\label{eq:sym and FKG}
\bbP [\calV_{ h \omega \leq 0}(R)]  \geq \bbP [\calV_{ h \omega  \geq 1}(R)]. % = 1-p,
\end{align} 
\end{lemma}

\begin{proof}
Compute
\begin{align*}
\bbP^\xi_D [\calV_{ h \omega  \geq 1}(R)] 
&= \bbP^{\tau.\xi}_D [\calV_{ h \omega \geq 1}(R)] \qquad \text{(by symmetry)} \\
& \leq \bbP^{2-\xi}_D [\calV_{ h \omega \geq 1}(R)],
\end{align*}
where the second step used the FKG for and the assumption $\tau.\xi \preceq 2- \xi$, as the event $\calV_{ h \omega \geq 1}(R)$ is increasing in $(h, B)$. Finally, given $(h, B, \omega) \sim \bbP^{2-\xi}_D $, set $\tilde{h} = 2-h$, and let $\tilde{B}$ be distributed same as $B$ but coupled so that $B_e = 1 \Rightarrow \tilde{B}_e = 0$, and define $\tilde{\omega}$ from $\tilde{h}$ and $\tilde{B}$; in this coupling $(\tilde{h}, \tilde{B}, \tilde{\omega}) \sim \bbP^{\xi}_D$ and $\calV_{ h \omega \geq 1}(R) \subset \calV_{ \tilde{h} \tilde{\omega} \leq 0}(R)$, and thus
\begin{align*}
\bbP^{2-\xi}_D [\calV_{ h \omega \geq 1}(R)] \leq \bbP^{\xi}_D [\calV_{  h \omega \leq 0}(R)].
\end{align*}
%and we conclude~\eqref{eq:sym and FKG}. 
\end{proof}

%We will next use the square-root trick to subdivide $\calV_{h \omega \geq 1} (R)$ and $\calV_{h \omega \leq 0} (R)$ into subevents that are still very likely but give more geometric control over the vertical crossings.
Let next $R_{\mathsf{top}}$,  $R_{\mathsf{mid}}$, and  $R_{\mathsf{bot}}$, respectively, denote the top, middle, and bottom thirds of $R$, i.e., rectangles of width $(2\rho n+1)$ hexagons and height $(4n/3 + 1)$ rows, overlapping in their lowest/highest row of faces.\footnote{We explicate the geometry for the hexagonal lattice. Other lattices can be treated so that one horizontal period corresponds to one hexagon, while the vertical one is two rows of hexagons, e.g., width $(2\rho n+1)$ hexagons meaning spanning $2\rho n$ translational periods between left-most and right-most faces.} Subdivide the four distinct top and bottom segments of these sub-rectangles into $3 \rho / 2$ subsegments (dual-paths) of width $2n/3 +1$ hexagons each, overlapping at their extremal faces, and denoted by $I_1, I_2, \ldots, I_{3 \rho / 2}$ on the bottom-most row, then $J_1, \ldots, J_{3 \rho / 2}$, then $K_1,  \ldots, K_{3 \rho / 2}$, and finally $L_1, \ldots, L_{3 \rho / 2}$ on the top-most row. For a vertical primal-crossing $\gamma$ of $R$ in $R$, directed from bottom to top and generating $\calV_{h \omega \leq 0} $ (resp. directed top to bottom and generating $\calV_{h \omega \geq 1}$), let us denote by $\eta$ the sub-path of $\gamma$ that ends at the first hitting time of $V_{\mathsf{top}}(R_{\mathsf{mid}})$ (resp. $V_{\mathsf{bottom}}(R_{\mathsf{mid}})$) and begins at the last vertex of of $V_{\mathsf{bottom}}(R_{\mathsf{mid}})$ (resp. $V_{\mathsf{top}}(R_{\mathsf{mid}})$) before that. Then, let $\calV^{i, j, k, l}(h \omega \leq 0) $ (resp. $\calV^{i, j, k, l}(h \omega \geq 1)$) be the event that such a $\gamma$ lands on $I_i$ and $L_l$ and the related $\eta$ between $J_j$ and $K_k$.\footnote{For several technical details, it is important that when we talk about the \textit{landing} of primal-curves, we interpret, e.g., $I_1, \ldots, I_{3 \rho / 2}$ as disjoint, horizontally ordered, and translationally equivalent collections of primal-edges \textit{in $R$} adjacent to $V_{\mathsf{bottom}}(R)$.}

\begin{lemma}
\label{lem:geom controlled crossings}
Suppose that $p$ is small enough, $p < c(\rho)$. Then, there exist $i_0$ and $j_0$ such that, denoting $N = (3 \rho /2)^4$, we have
\begin{align*}
\bbP [\calV^{i_0, j_0, j_0, i_0}(h \omega \leq 0)], \bbP [\calV^{i_0, j_0, j_0, i_0}(h \omega \geq 1)] \geq 1-p^{1/N}.
\end{align*}
\end{lemma}

\begin{proof}
It is clear that also $\calV^{i, j, k, l} (h \omega \geq 1)$ are increasing in $(h,B)$, and that
\begin{align*}
\calV_{h \omega \geq 1} (R) = \bigcup_{i, j, k , l= 1}^{3 \rho / 2}  \calV^{i, j, k, l} (h \omega \geq 1) .
\end{align*}
By FKG and the square-root trick, there exist indices $(i_0, j_0, k_0, l_0)$ such that
\begin{align*}
\bbP [\calV^{i_0, j_0, k_0, l_0}(h \omega \geq 1)] \geq 1-p^{1/N},
\end{align*}
where $N = (3 \rho /2)^4$ is the total number of index tuples $(i, j, k, l)$.  By reflection symmetry and FKG, also
\begin{align*}
\bbP [\calV^{l_0, k_0, j_0, i_0}(h \omega \leq 0)] \geq \bbP [\calV^{i_0, j_0, k_0, l_0}(h \omega \geq 1)];
\end{align*}
the proof is identical to that of~\eqref{eq:sym and FKG}.

Next, we will prove that $i_0=l_0$ and $j_0 = k_0$ necessarily.
%\begin{align}
%\label{eq:geom control of cross}
%i_0 = l_0
%\qquad \text{and} \qquad
%j_0 = k_0,
%\end{align}
% and thus in particular, the $h \omega \geq 1$ and $h \omega \leq 0$ crossings inducing $\calV^{i_0, j_0, k_0, l_0}(h \omega \geq 1)$ and $\calV^{l_0, k_0, j_0, i_0}(h \omega \leq 0)$, respectively, are actually likely to pass through the \textit{same} narrow slits $I_{i_0}, J_{j_0}, K_{j_0}$, and $L_{i_0}$. 
Indeed, by the Union bound, 
\begin{align*}
\bbP [\calV^{i_0, j_0, k_0, l_0}(h \omega \geq 1) \cap \calV^{l_0, k_0, j_0, i_0}(h \omega \leq 0)] \geq 1 - 2 p^{1/N},
\end{align*}
and if $p$ is small enough, $p < c(\rho)$, the two occur simultaneously with a positive probability. But the paths inducing such events are by definition edge-disjoint and hence also vertex-disjoint, possibly apart from endpoint vertices, since $\bbG$ is of degree $3$. %Moreover, the curves are actually completely vertex-disjoint, since any possible endpoint vertex right outside of $R$ can only be reached via a unique edge crossing the dual-loop around $R$, and this edge can only figure in one of the two paths.
On the other hand, studying the two sub-curves crossing $R_{\mathsf{mid}}$ between $J_{j_0} \leftrightarrow K_{k_0}$ and $J_{k_0} \leftrightarrow K_{j_0}$ (and recalling that these are interpreted as horizontally ordered collections of primal-edges), we observe that $j_0 \neq k_0$ would topologically force the sub-curves to vertex-intersect inside $R_{\mathsf{mid}}$, a contradiction, so we must have $j_0 = k_0$. Similarly from the crossings of the entire $R$, i.e., $I_{i_0} \leftrightarrow L_{l_0}$ and $I_{l_0} \leftrightarrow L_{i_0}$, we observe that $i_0 = l_0$. This concludes the proof.
\end{proof}

Next, we will need even more geometric control of the crossings. Let $\calV^{i, j, k, l}_{\alpha, \beta, \gamma}(h \omega \geq 1)$ (resp. $\calV^{i, j, k, l}_{\alpha, \beta, \gamma}(h \omega \leq 0)$) be the event that there exists a vertical primal-crossing $\gamma$ inducing $\calV^{i, j, k, l}(h \omega \geq 1)$ (resp. $\calV^{i, j, k, l}(h \omega \leq 0)$) and furthermore
\begin{itemize}
\item the unique subpath of $\gamma$ crossing $R_{\mathsf{bot}}$ vertically (which lands on $I_i$ on the bottom) passes strictly farther right than the right extreme of $I_i$ if $\alpha = +1$ (resp. passes strictly farther left that the left extreme if $\alpha = -1$, resp. remains right above $I_i$ if $\alpha = 0$); and
\item
the unique subpath of $\gamma$ crossing $R_{\mathsf{top}}$ vertically (which lands on $L_l$ on the top) passes strictly farther right than the right extreme of $L_l$ if $\gamma = +1$ (resp. passes strictly farther left that the left extreme if $\gamma = -1$ , resp. remains right below $L_l$ if $\gamma = 0$ ); and
\item the related subcurve $\eta \subset \gamma$ (which lands on $J_j$ and $K_k$) passes strictly farther right than the right extreme of $J_j$ if $\beta = +1$ (resp. passes strictly farther left that the left extreme if $\beta = -1$, resp. remains right above $J_j$ if $\beta = 0$).
\end{itemize}
Another application of the square-root trick 
%(since $\calV^{i_0, j_0, j_0, i_0}_{\alpha_0, \beta_0, \gamma_0}(h \omega \leq 0) $ are decreasing in $(h, B)$)
 guarantees that there exist $\alpha_0, \beta_0, \gamma_0$ such that
\begin{align}
\label{eq:one geom controlled crossing}
\bbP [\calV^{i_0, j_0, j_0, i_0}_{\alpha_0, \beta_0, \gamma_0}(h \omega \geq 1) ] \geq 1- p^{1/(27N)},
\end{align}
and arguing once again similarly to~\eqref{eq:sym and FKG} shows that
\begin{align*}
\bbP [\calV^{i_0, j_0, j_0, i_0}_{\gamma_0, \beta_0, \alpha_0}(h \omega \leq 0) ]
\geq
\bbP [\calV^{i_0, j_0, j_0, i_0}_{\alpha_0, \beta_0, \gamma_0}(h \omega \geq 1) ].
\end{align*}
The last lemma in this first part of the proof is to create many crossings of the above type, with ``alternating signs''.

\begin{lemma}
Suppose that $p$ is small enough, $p < c(\rho)$.
There exist horizontally ordered collections of primal-edges (``subintervals'') $I^1, \ldots, I^8 \subset I_{i_0}$, intersecting at their extremal edges, so that the following holds. Let $\calE_q^+$ (resp. $\calE_q^-$) denote the existence of a crossing inducing $\calV^{i_0, j_0, j_0, i_0}_{\alpha_0, \beta_0, \gamma_0} (h \omega \geq 1 )$ (resp. $ \calV^{i_0, j_0, j_0, i_0}_{\gamma_0, \beta_0, \alpha_0} (h \omega \leq 0 )$) and furthermore landing on $I^q $ on $I_{i_0}$ and its reflection $L^q := \tau(I^q)$ on $L_{i_0}$. We have
\begin{align*}
\bbP[\calE^\pm_q] \geq 1 - p^{1/(4^3 \cdot 27N)}, \qquad \text{for all } 1 \leq q \leq 8.
\end{align*}
\end{lemma}

\begin{proof}
Let $x$ and its reflection $\tau(x)$ be primal-edges landing on the slits $I_{i_0}$ and $L_{i_0}$, respectively. Let $\calV_{a, b}^x (h \omega \geq 1)$ (resp. $\calV_{a, b}^x (h \omega \leq 0)$), where $a$ and $b$ are order relations ``$\leq$'', ``$<$'', ``$\geq$'', or ``$>$'', be the event that there is a vertical crossing with $h \omega \geq 1$ of $R$, inducing $\calV^{i_0, j_0, j_0, i_0}_{\alpha_0, \beta_0, \gamma_0} (h \omega \geq 1)$ (resp. $\calV^{i_0, j_0, j_0, i_0}_{\gamma_0, \beta_0, \alpha_0} (h \omega \leq 0)$), and furthermore so that it lands on $I_{i_0}$ through or left of (resp. strictly left of, resp. through or right of, resp. strictly right of) $ x$ if $a=$``$\leq$'' (resp. ``$<$'', ``$\geq$'', or ``$>$''), and $b$ determined similarly by comparing landing on $L_{i_0}$ to $\tau(x)$.
Clearly, fixing $x$ to be left-most edge in $I_{i_0}$, $(a, b) = (\geq, \geq)$ maximizes $\bbP [\calV_{a, b}^{x} (h \omega \geq 1)]$ among the order relations. On the other hand, by yet another square-root trick and~\eqref{eq:one geom controlled crossing}, for any $x$
\begin{align*}
\max_{a,b \in \{ \geq, < \}} \bbP [ \calV_{a, b}^x (h \omega \geq 1) ] \geq 1 - p^{1/( 4 \cdot 27 N)}.
\end{align*}
In particular, starting with $x$ being the left-most possible location on $I_{i_0}$ and sliding $x$ to the right one edge at a time, there will appear a right-most edge $x_0 \in I_{i_0}$, where still
\begin{align*}
\bbP [ \calV_{\geq, \geq}^{x_0} (h \omega \geq 1) ] \geq 1 - p^{1/( 4 \cdot 27 N)}.
\end{align*}
For this $x_0$, by definition, there exist $a_0, b_0 \in \{ \leq, \geq \}$ of which at least one is ``$\leq$'', so that
\begin{align*}
\bbP [ \calV_{a_0, b_0}^{x_0} (h \omega \geq 1) ] \geq 1 - p^{1/( 4 \cdot 27 N)}.
\end{align*}
We claim that it necessarily occurs that $a_0 = b_0 =$``$\leq$''. Indeed, repeating the proof of~\eqref{eq:sym and FKG},
\begin{align*}
\bbP [ \calV_{b_0, a_0}^{x_0} (h \omega \leq 0) ] \geq \bbP [ \calV_{a_0, b_0}^{x_0} (h \omega \geq 1) ],
\end{align*}
but if only one out of  $a_0, b_0 $ is $\leq$, then the crossings inducing $\calV_{a_0, b_0}^{x_0} (h \omega \geq 1)$ and $\calV_{b_0, a_0}^{x_0} (h \omega \leq 0) $ cannot coexist by the same argument as in the proof of Lemma~\ref{lem:geom controlled crossings}, a contradiction (for small enough $p < c(\rho)$). This provides a subdivision of $I_{i_0}$ and $L_{i_0}$ into two subintervals; the claim follows by repeating the subdivision procedure.
\end{proof}

\subsubsection{Contradiction via bridging}
 
By the previous lemma and the union bound,
\begin{align*}
\bbP [ \calE^+_1 \cap \calE^-_2 \cap \calE^+_3 \cap \calE^+_5 \cap \calE^-_6 \cap \calE^+_7] &\geq 1 - 6 p^{1/( 4^3 \cdot 27 N)} \\
\bbP [ \calE^+_4 ] &\geq 1 -  p^{1/( 4^3 \cdot 27 N)}.
\end{align*}
The rest of the proof will be to use the former probability to upper-bound the probability of $\calE^+_4 $ so that, if $p$ were small enough, the two formulas above would be contradictory.

Let now $\tilde{h} = 6-h$, and define a percolation $\nu$ on $E$ by letting
\begin{align}
\label{eq:def of the perco nu}
\nu_{\langle u, v \rangle} =
\begin{cases}
0, \qquad \{ |\tilde{h}(u)|, |\tilde{h}(v)| \} = \{ 5, 7 \} \\
B_{\langle u, v \rangle}, \qquad |\tilde{h}(u)|= |\tilde{h}(v)| \in \{ 5, 7 \} \\
1, \qquad \text{otherwise}.
\end{cases}
\end{align}
so $\omega_e =0 \Rightarrow \nu_e = 0$, but $\nu$ also has ``false zeroes'' appearing at heights $\tilde{h} \in \{ -5, -7 \}$; these false zeroes will later on play a crucial role in that they remove from $\nu$ any information about the sign of $\tilde{h}$ that $\omega$ would have contained.
% Note also that a crossing with $h \omega \geq 1$ as in $\calE_q^+$ (resp. $h \omega \leq 0$ as in $\calE^-_q$) equivalently means that $\tilde{h} \leq 5$ with additionally $\nu_e= 1$ on edges between two heights $5$ (resp. $\tilde{h} \geq 5$ with additionally $\nu_e= 0$ on edges between two heights $5$). 

Define the event $\calV^{i, j, k, l \; *}_{\alpha, \beta, \gamma} (\nu = 0)$ analogously to $\calV^{i, j, k, l }_{\alpha, \beta, \gamma} (h \omega \geq 1)$, except in terms of dual-paths with $\nu = 0$ on the crossed primal-edges.

\begin{lemma}
\label{lem:find a crossing}
Suppose that $ \calE^+_1 \cap \calE^-_2 \cap \calE^+_3$ occurs. Then, there exists a vertical dual-crossing of $R$ with $\nu = 0$, that lands on $ I^2 \cup I^3$ and $L^2 \cup L^3$ on the bottom and top, respectively. Furthermore the left-most such dual-crossing induces $\calV^{i_0, j_0, j_0, i_0 \; *}_{\alpha_0, \beta_0, \gamma_0} (\nu = 0)$. 
\end{lemma}

The obvious analogue of the above lemma, with ``left'' and ``right'' reversed, holds for $ \calE^+_5 \cap \calE^-_6 \cap \calE^+_7$. Let $\gamma^*_{\mathsf{left}}$ (resp. $\gamma^*_{\mathsf{right}}$) be the left-most (resp. right-most) vertical dual-crossing of $R$ with $\nu_e= 0$ that traverses $I_{i_0} \setminus I^1 \leftrightarrow L_{i_0} \setminus L^1$ (resp. $I_{i_0} \setminus (I^7 \cup I^8) \leftrightarrow L_{i_0} \setminus (L^7 \cup L^8)$), such exist on $\calE^+_1 \cap \calE^-_2 \cap \calE^+_3 \cap \calE^+_5 \cap \calE^-_6 \cap \calE^+_7$, and are dual-vertex-disjoint by the proof of the previous lemma (separation by the curve $\gamma_3$ in that proof). In that case, let $\calX $ be the quad bounded by $\gamma^*_{\mathsf{left}}$, $\gamma^*_{\mathsf{right}}$, and the top and bottom boundary segments of $R$ between them. Let $\Xi $ be the collection of all sub-quads of $R$, whose top and bottom are contained in the top and bottom of $R$, respectively, and the left (resp. right) side is of the type $\calV^{i_0, j_0, j_0, i_0 \; *}_{\gamma_0, \beta_0, \alpha_0}$ and furthermore lands on $ I^2 \cup I^3$ and $ L^2 \cup L^3$ (resp. on $I^5 \cup I^6 $ and $L^5 \cup L^6 $) on the bottom and top, respectively.  In this notation, the conclusion of the previous lemma reads
\begin{align*}
\{ \calE^+_1 \cap \calE^-_2 \cap \calE^+_3 \cap \calE^+_5 \cap \calE^-_6 \cap \calE^+_7 \}
\subset \{ \calX \in \Xi \} ,
\qquad \text{so} \qquad
\bbP [\calX \in \Xi] \geq 1 - 6 p^{1/( 4^3 \cdot 27 N)}
\end{align*}
and by the Union bound,
\begin{align}
\label{eq:contradiction1}
\bbP [\calE_4^+ \cap \{ \calX \in \Xi \}] \geq 1 - 7 p^{1/( 4^3 \cdot 27 N)}.
\end{align}

The proof of the entire theorem is now readily finished with the following ``bridging lemma''.
\begin{lemma}
\label{lem:bridging for the push lemma}
For any $X \in \Xi$, we have
\begin{align}
\label{eq:contradiction2}
\bbP [\calV_{h \omega \geq 1}(X) \; | \; \calX = X] \leq \tfrac{31}{32}.
\end{align}
\end{lemma}

\begin{proof}[Theorem~\ref{prop:pushing}]
Compute
\begin{align*}
\bbP [\calE_4^+ \cap \{ \calX \in \Xi \} ] 
&= \sum_{X \in \Xi} \bbP [ \calE_4^+ \; | \; \calX = X ] \bbP [ \calX = X ] \\
& \leq \sum_{X \in \Xi} \bbP [\calV_{h \omega \geq 1}(X) \; | \; \calX = X ] \bbP [ \calX = X ] \qquad \text{(by inclusion of events)} \\
& \leq \sum_{X \in \Xi} \tfrac{31}{32} \bbP [ \calX = X ]  \qquad \text{(by ~\eqref{eq:contradiction2})}\\
& = \tfrac{31}{32} \bbP [ \calX \in \Xi ]  \leq \tfrac{31}{32}.
\end{align*}
This upper bound contradicts~\eqref{eq:contradiction1} if $p<c(\rho)$ is small enough.
\end{proof}

\begin{proof}[Lemma~\ref{lem:bridging for the push lemma}]
Fix $X \in \Xi$ throughout the proof.
Note first that the event $\{ \calX = X \}$ can be determined based on $\nu_{|X^c}$ only.\footnote{We interpret here and below the subgraph $X\subset D$ as a vertex/edge set in the obvious manner.} Denote $(X, \chi, \beta) \in \Xi$ if $  \nu_{|X^c} = \beta$ generates $\calX = X$, and the configuration $ \chi$ for $ | \tilde{h} |$ on $X^c \cup V_{\mathsf{bottom}}(X) \cup V_{\mathsf{top}}(X) \setminus \partial D = Y$ can occur simultaneously with $  \nu_{|X} = \beta$; hence
\begin{align*}
 \sum_{(\chi, \beta): \; (X, \chi, \beta) \in \Xi} \bbP [ | \tilde{h}_{|Y } | = \chi, \nu_{| X^c } = \beta \; | \; \calX = X] = 1
\end{align*}
and
\begin{align*}
\bbP
& [  \calV_{h \omega \geq 1}(X) \; | \; \calX = X ] \\
& = \sum_{(\chi, \beta): \; (X, \chi, \beta) \in \Xi} \bbP [ \calV_{h \omega \geq 1}(X) \; | \; | \tilde{h}_{|Y } | = \chi, \nu_{| X^c } = \beta ] \bbP [ | \tilde{h}_{|X } | = \chi, \nu_{| X } = \beta \; | \; \calX = X].
\end{align*}
It thus suffices to show that
\begin{align*}
\bbP[ \calV_{h \omega \geq 1}(X) \; | \; | \tilde{h}_{|Y} | = \chi, \nu_{| X^c } = \beta] \leq \tfrac{31}{32} \qquad \text{for any } (X, \chi, \beta) \in \Xi.
\end{align*}

By the bound on dual degree, the sign of $\tilde{h}$ is fixed on $\gamma^*_{\mathsf{left}}$ and $\gamma^*_{\mathsf{right}}$. Let $s_{\mathsf{left}}, s_{\mathsf{right}} \in \{ \pm 1 \}$ denote these two random signs. Note that for any $(X, \chi, \beta) \in \Xi$,
\begin{align*}
\bbP[ s_{\mathsf{left}} = s_{\mathsf{right}} = +1 \; | \; | \tilde{h}_{|Y } | = \chi, \nu_{| X^c } = \beta] \geq 1/4.
\end{align*}
This stems from the fact that, given whatever extension of $|\tilde{h}|$ to the entire domain $D$, the conditional law of the signs of $\tilde{h}$ obey a certain ferromagnetic Ising model (Lemma~\ref{cor:signs are Ising}). (The boundary condition $\xi$ for $h$ reveals the sign of $\tilde{h}$ to be positive in some places, but this will only increase the probability of plus signs in the Ising model; recall also that given $|\tilde{h}|$, $\nu$ is conditionally independent of the sign of $\tilde{h}$, so the additional information $\nu_{| X^c } = \beta$ will not change the conditional law of the signs of $\tilde{h}$.) It thus suffices to show that
\begin{align*}
\bbP[ \calV_{h \omega \geq 1} (X) \; | \; | \tilde{h}_{|Y } | = \chi, \nu_{| X^c } = \beta, s_{\mathsf{left}} = s_{\mathsf{right}} = +1] \leq \tfrac{7}{8} \qquad \text{for any } (X, \chi, \beta) \in \Xi.
\end{align*}

Next, by (repeating the proof of) the SMP (Lemma~\ref{lem:SMP with percolation}), we note that
 \begin{align*}
 \bbP [\cdot \; | \; | \tilde{h}_{|Y } | = \xi, \nu_{| X^c} = \beta, s_{\mathsf{left}} = s_{\mathsf{right}} = +1 ]
 = \bbP_{\tilde{D}}^{\chi'} [\cdot ],
 \end{align*}
where $\tilde{D}$ is the graph $D$ with the primal-edges crossing $\gamma^*_{\mathsf{left}} \cup \gamma^*_{\mathsf{right}}$ removed and $\chi'$ is the following boundary condition given on $Y \cup V_{\mathsf{left}}(X) \cup  V_{\mathsf{right}}(X) \cup \partial D$: it is given by $|\tilde{h}| = \chi$ on $Y$,  $6-\tilde{h} = \xi$ on $\partial D$, and $\tilde{h} \in \{ 5, 7 \}$ on $V_{\mathsf{left}}(X) \cup  V_{\mathsf{right}}(X)$. This is an absolute value boundary condition for height function $\tilde{h}$.

Let us consistently interpret $B = 0$ and $h=1$ on $X \setminus D$ below; whence $\calV_{h \omega \geq 1} (X)^c = \calH_{h \omega \leq 0 }^* (X)$; in terms of $\tilde{h}$, this means there is a horizontal dual-crossing of $X$ such that $\tilde{h} \geq 5$ (with the boundary condition of $\bbP_D^{\chi'}$ equivalently, $|\tilde{h}| \geq 5$) with additionally $B_e = 0$ on edges $e$ between two heights $\tilde{h} = 5$ (equivalently, absolute-heights $|\tilde{h}| = 5$). As such, $\calH_{h \omega \leq 0}^*(X)$ is increasing in $(|\tilde{h}|, -B)$. Hence, by Proposition~\ref{prop:pos assoc of percolated height fcns}\footnote{Observe that the assumptions of this proposition hold for $\tilde{D}$ while they wouldn't for $D$.} and then~\eqref{eq:CBC-|h|} we have
\begin{align*}
\bbP_{\tilde{D}}^{\chi'} [\calH_{h \omega \leq 0}^*(X) ] \geq \bbP_{D \cap X}^{\chi'_{|D \cap X}} [\calH_{h \omega \leq 0}^*(X) ] \geq \bbP_{D \cap X}^{\psi} [\calH_{h \omega \leq 0}^*(X) ] ,
\end{align*} 
where $\psi$ is the boundary condition for $\heightfcns_{D \cap X}$ on $\partial (D \cap X)$ which is $\tilde{h} \in \{ 5, 7 \}$ (i.e., equal to $\chi'$) on $V_{\mathsf{left}}(X) \cup  V_{\mathsf{right}}(X)$ and $\partial D \cap X$, and its smallest extension with $\tilde{h} \succeq \{ \pm 1 \}$ on the rest of $\partial (D \cap X)$. With this boundary condition, $\calH_{h \omega \leq 0}^*(X) = \calH_{\omega = 0}^*(X)$, a decreasing event in $(|h|, \omega)$ and by the FKG and CBC for $(|h|, \omega)$,
\begin{align*}
\bbP_{D \cap X}^{\psi} [\calH_{h \omega \leq 0}^*(X) ] \geq \bbP_{ X}^{\psi'} [\calH_{h \omega \leq 0}^*(X) ],
\end{align*}
where $\psi'$ is  $\tilde{h} \in \{ 5, 7 \}$ on $V_{\mathsf{left}}(X) \cup  V_{\mathsf{right}}(X)$ and its smallest extension with $\tilde{h} \succeq \{ \pm 1 \}$ on the rest of $\partial X $. In terms of $h$, this boundary condition is exactly matched for the application of Corollary~\ref{cor:narrow quad triple cross}. The proof of Lemma~\ref{lem:bridging for the push lemma} is thus concluded by constructing below the symmetric quads $Q_1, Q_2, Q_3$ required for that corollary, which then yields
\begin{align*}
\bbP_X^{\psi} [\calH_{h \omega \leq 0}^*(X) ] \geq 1/8.
\end{align*}

\paragraph{Construction of $Q_1$ and $C_1$}
If $\alpha_0 = 0$, we let $Q_1$ be a suitable lozenge (or square, for some other lattices), containing $I_{i_0}$ in its bottom and $J_{i_0}$ in its top boundary (the dimensions have been matched so that such exists).

If $\alpha_0 = +1$, $\gamma^*_{\mathsf{left}}$ passes strictly right of the vertical line $S$ marking the right extremity of $I_{i_0}$.
Let $\gamma^*_{\mathsf{left, tr}}$ be the truncation of $\gamma^*_{\mathsf{left}}$ (started from the bottom) at the first hitting of $S$. Let $\tau'$ denote reflection around $S$, and let $Q_1$ be the symmetric quad constrained by $\gamma^*_{\mathsf{left, tr}}$, $\tau' (\gamma^*_{\mathsf{left, tr}})$, and a suitable bottom boundary segment of $R$, with marked faces at the end points of $\gamma^*_{\mathsf{left, tr}}$, $\tau' (\gamma^*_{\mathsf{left, tr}})$ and the intersection of $S$ and the bottom of $R$. The case $\alpha_0 = -1$ is similar, with left and right reversed.

In all three cases, we then let $C_1$ be the component of $Q_1$ that contains the bottom boundary segment of $R$ between $\gamma^*_{\mathsf{left}}$ to $\gamma^*_{\mathsf{right}}$. Clearly, $C_1$ is on the bottom of $X$ and induces $X$ being narrower than $Q_1$. The construction of $Q_3$ and $C_3$ is analogous.

\paragraph{Construction of $Q_2$ and $C_2$}
Analogously to earlier notations, let $\eta^*_{\mathsf{left}}$ and $\eta^*_{\mathsf{right}}$ be the first crossings of $R_{\mathsf{mid}}$ of $\gamma^*_{\mathsf{left}}$ and $\gamma^*_{\mathsf{right}}$, respectively. In the three cases $\beta_0 \in \{ 0, \pm 1\}$, construct the symmetric quad $Q_2$ similarly to $Q_1$ above, but replacing $\gamma^*_{\mathsf{left, tr}}$ and $\gamma^*_{\mathsf{right, tr}}$ by $\eta^*_{\mathsf{left}}$ and $\eta^*_{\mathsf{right}}$, respectively. Clearly, the discrete quad $Q_2'$ between $\gamma^*_{\mathsf{left}}$ and $\gamma^*_{\mathsf{right}}$ in $R_{\mathsf{mid}}$ is then narrower than $Q_2$. To show that $X$ is narrower than $Q_2$, it hence suffices to show that it is narrower than $Q_2'$ (the reader may verify that being narrower indeed is a transitive relation).

For this purpose, recall that each connected component of $X \cap Q_2'$ is a discrete simply-connected domain with exactly two segments of $\partial X$ on its boundary that cross $Q_2'$ vertically: the left and right \textit{walls} of that component. On the other hand, from the first crossing $\eta^*_{\mathsf{left}}$ onwards, $\gamma^*_{\mathsf{left}}$ crosses $Q_2'$ vertically an even number of times: equally many upward (i.e., left walls of some component) and downward (i.e., right walls). In particular, counting in $\eta^*_{\mathsf{left}}$, $\gamma^*_{\mathsf{left}}$ makes one more upward than downward crossing of $Q_2'$, and hence, out of the components of $X \cap Q_2'$, at least one must have an upward crossing of $\gamma^*_{\mathsf{left}}$ as its left wall and an upward crossing of $\gamma^*_{\mathsf{right}}$ as its right wall. Any such component induces $X$ being narrower than $Q_2'$

This verifies the assumptions of Corollary~\ref{cor:narrow quad triple cross}, concluding the proof of Lemma~\ref{lem:bridging for the push lemma}.
%, and hence the entire proof of Theorem~\ref{prop:pushing}. 
\end{proof}

\section{Russo--Seymour--Welsh type results}

Russo--Seymour--Welsh (RSW) type results mean, very loosely speaking, results for percolation models or level sets of random fields, that allow one to ``create long crossings from short ones'' in a scale-invariant manner. The purpose of the present subsection is to derive such results for the random Lipschitz model (with percolation); see Figure~\ref{fig:RSW}.
% These results will be both useful later, and crucial for the proof of the second Crack Lemma.

\begin{figure}
\begin{center}
\includegraphics[width=0.65\textwidth]{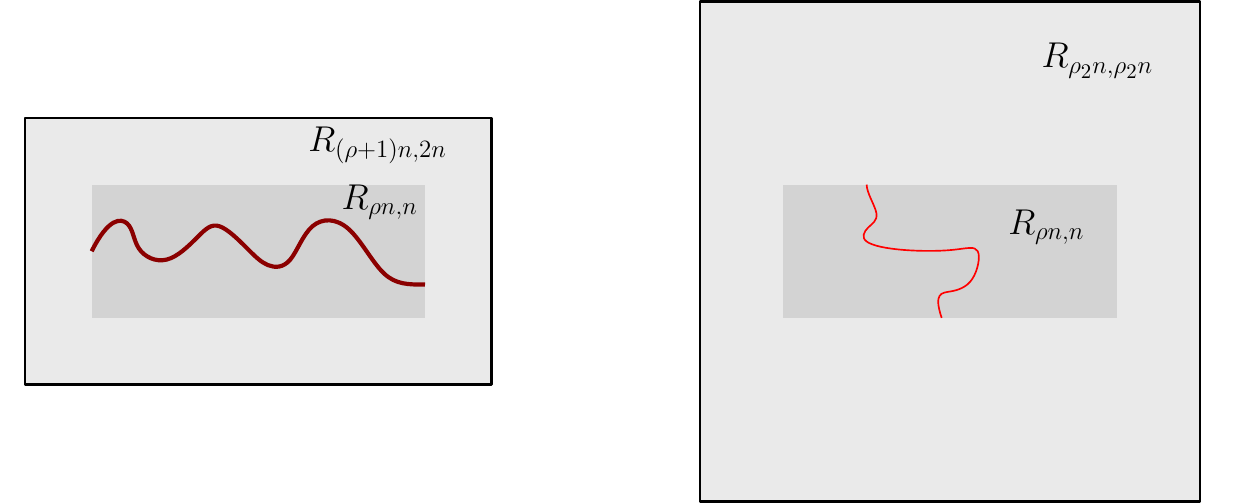}
\end{center}
\caption{
\label{fig:RSW}
Theorem~\ref{thm:RSW} (illustrated here in the case $\rho'=\rho$, $\epsilon=1$) lower-bounds the probability of a high crossing in the difficult direction of a rectangle in a difficult domain, in terms of a polynomial of the probability for a lower crossing in the easy direction in an easy domain.
}
\end{figure}

\begin{theorem}
\label{thm:RSW}
For every $\rho, \epsilon > 0$ and $\rho_2 \geq \rho$, there exist $C= C(\rho, \rho_2, \epsilon ) , c= c(\rho, \rho_2, \epsilon ) > 0$ such that for all $n \geq  4/\epsilon$ and $\rho' \leq \rho$, we have
\begin{align*}
c \bbP^{ \pm 1 }_{R_{\rho_2 n, \rho_2 n}} [\calV_{\omega = 1}(R_{\rho n, n})]^{C} 
\leq
\bbP^{ \pm 1 }_{ R_{(\rho'+ \epsilon )n, 2\epsilon  n}} [\calH_{\calE}^* ( R_{\rho' n, \epsilon n})],
\end{align*}
where $\calE$ denotes the set of edges with heights $h \geq 5$ at the endpoints, with additionally $B_e = 0$ on edges $e$ between two heights $h=5$.
\end{theorem}

\begin{remark}
Note that, due to the bound on dual degree, $\calH_{\calE}^* ( R_{\rho' n, \epsilon n}) \subset \calH_{\omega = 1} ( R_{\rho' n, \epsilon n})$.
\end{remark}

We immediately record the following simple but important corollary.

\begin{corollary}
\label{lem:crack alternative}
For every $\rho, \epsilon > 0$ and $\rho_2 \geq \rho$, there exists $c= c(\rho, \rho_2, \epsilon )$ such that for all $n$, at least one of the following two holds: (i)
\begin{align*}
\bbP^{ \pm 1 }_{R_{\rho_2 n, \rho_2 n}} [\calH^*_{\omega = 0}(R_{\rho n, n})] > c;
\end{align*}
or (ii) for all $\rho' \leq \rho$, we have
\begin{align*}
\bbP^{ \pm 1 }_{ R_{(\rho'+ \epsilon )n, 2\epsilon  n}} [\calH_{\calE}^* ( R_{\rho' n, \epsilon n})] > c.
\end{align*}
\end{corollary}

\begin{proof}
Denote $\bbP^{ \pm 1 }_{R_{\rho_2 n, \rho_2 n}} [\calH^*_{\omega = 0}(R_{\rho n, n})] := p_n$. By~\eqref{eq:quad cross duality} and the above theorem,
\begin{align*}
\bbP^{ \pm 1 }_{ R_{(\rho'+ \epsilon )n, 2\epsilon  n}} [\calH_{\calE}^* ( R_{\rho' n, \epsilon n})] \geq c'(1-p_n)^{C}.
\end{align*}
The claim thus holds with $c=\min_{x \in [0, 1]}\{ x, c'(1-x)^{C} \}$.
\end{proof}

The rest of the present section is dedicated to the proof of Theorem~\ref{thm:RSW}.

\subsection{From vertical to horizontal crossings}

We start with the key step, which creates a horizontal crossing from a vertical crossing. This step will crucially rely on treating the model on the infinite strip $\Strip(m) = \Lambda_{\infty, m}$; the horizontal shift invariance is instrumental, while the Pushing Theorem and FKG will in later subsections allow us go back to finite domains.

A boundary condition $\xi$ on a subset of $\Strip(m)$ is admissible if $\xi_{|R_{n, m}}$ is admissible on $\heightfcns_{R_{n, m}}$ for all $n$ (large enough). It is readily shown that there then exists a unique weak limit of $\bbP^{\xi_{|D}}_{D}$ when $D \uparrow \Strip(m)$, by the finite energy of the model, and this limit is taken as the definition of the strip measure $\bbP^{\xi}_{\Strip(m)}$. Note that the FKG and SMP properties for the strip measure (including the independent Bernoulli percolation $B$ and $\omega$ defined via $(h,B)$ as usual) directly follow from those of the finite-domain measures.

\begin{lemma}
\label{lem:from vert to hor crossing}
For every $\rho > 0$, there exist $C= C(\rho), c= c(\rho) > 0$ such that the following holds. Let $m>n$ be integers, and let $R$ be any shifted version of the rectangle $R_{ \rho n, n}$ contained in $\Strip(m)$. Then, for any $s \in \{ 1,3, 5\}$ we have 
\begin{align*}
c \bbP^{\{ s, s+2 \}}_{\Strip(m)} [\calV^{*}_{h \omega \leq 0}(R)]^{C} 
\leq
\bbP^{\{ s, s+2 \}}_{\Strip(m)} [\calH^*_{h \omega \leq 0} (R)]  
.
\end{align*}
\end{lemma}

\begin{proof}
Denote $\bbP^{\{ s, s+2 \}}_{\Strip(m)} =: \bbP$ throughout the proof and,
% The proof is similar to, but simpler than, that of the Pushing theorem. Again, 
without loss of generality, assume that $\rho$ is an even integer and $n \geq 24$ and divisible by $3$. Subdivide $R$ into three translates of $R_{\rho n, n/3}$,
%rectangles of width $2\rho n +1$ hexagons and height $4n/3 + 1$ rows, 
denoted $R_{\mathsf{top}}$,  $R_{\mathsf{mid}}$, and  $R_{\mathsf{bot}}$, and overlapping in their lowest/highest row of faces, and then subdivide these top/bottom boundaries (dual-paths) into $M$ dual-vertex disjoint subsegments of width $\lfloor 2n/45 \rfloor$ horizontal translational periods between endpoint faces, and overlapping at their extremal faces.\footnote{
%Similarly as the proof of the Pushing theorem we will interpret these subsegments as disjoint, translationally equivalent sets of primal-edges when needed.
It will later be crucial that these subsegments are $15$ times shorter than in the proof of the Pushing theorem.
} Denote these subsegments by $I_1, \ldots, I_{M}$ on the bottom-most row, then $J_1, \ldots, J_M$, then $K_1, \ldots, K_M$, and finally $L_1, \ldots, L_M$ on the top-most row, where $M \sim 45 \rho$. Define the subevents $\calV^{i,j,k,l}(h \omega \leq 0)$, $1 \leq i, j, k, l \leq M$, of $\calV^{*}_{h \omega \leq 0}(R)$ analogously as in the Pushing theorem. Then, we define $\calV^{i,j,k,l}_{\alpha, \beta, \gamma}(h \omega \leq 0)$ to be the event that there exists a vertical dual-crossing $\gamma^*$ of $R$, inducing $\calV^{i, j, k, l}(h \omega \leq 0)$ and furthermore such that
\begin{itemize}
\item the unique subpath of $\gamma^*$ crossing $R_{\mathsf{bot}}$ passes strictly farther right than the right extreme  of $I_{i-6} \cup \ldots \cup I_{i+6}$ if $\alpha = +1$ (resp. passes strictly farther left than the left extreme if $\alpha = -1$, resp. remains right above them if $\alpha = 0$); and
\item
the unique subpath of $\gamma^*$ crossing $R_{\mathsf{top}}$ passes strictly farther right than the right extreme  of $L_{l-6} \cup \ldots \cup L_{l+6}$ if $\gamma = +1$ (resp. passes strictly farther left that the left extreme if $\gamma = -1$, resp. remains right above them if $\gamma = 0$); and
\item the related subcurve $\eta^*$ of $\gamma^*$ passes strictly farther right than the right extreme of $J_{i-6} \cup \ldots \cup J_{i+6}$ if $\beta = +1$ (resp. passes strictly  farther left that the left extreme if $\beta= -1$, resp. remains right above them if $\beta = 0$).
\end{itemize}

%Since we are interested in the range of small $\bbP^{\{ \pm 1 \}}_D [\calV_{\omega = 1} (R)]$, we use this time 
By the Union bound,
%\footnote{Unlike the Pushing Proposition, we are this time interested in the range of small probabilities for the vertical crossing events; whence the Union Bound is stronger than the Square-root Trick.}, 
there exists $(i_0, j_0, k_0, l_0, \alpha_0, \beta_0, \gamma_0)$ such that
\begin{align}
\label{eq:RSWa1}
\bbP [\calV^{i_0,j_0,k_0,l_0}_{\alpha_0, \beta_0, \gamma_0}( h \omega \leq 0 )] \geq \tfrac{1}{N} \bbP [ \calV^*_{h \omega \leq 0} (R) ] ,
\end{align}
where $N = M^4 \cdot 3^3$ is the total number of tuples $(i,j,k,l, \alpha, \beta, \gamma)$. Let $\calE^*_{h \omega \leq 0} (\ell)$ denote the horizontal shift of the event $\calV^{i_0,j_0,k_0,l_0}_{\alpha_0, \beta_0, \gamma_0}( h \omega \leq 0 )$, so that the related crossing lands on $I_{\ell}$ on the bottom, instead of $I_{i_0}$.
%; define also $\calE^*_{h \omega \leq 0} (-1)$ and $\calE^*_{h \omega \leq 0} (M+1)$ as the natural further shifts.
Note that the events $\calE^*_{h \omega \leq 0} (\ell)$ are decreasing in $(h, B)$. By the shift symmetry and FKG,
\begin{align}
\label{eq:RSWa2}
\bbP [\bigcap_{\substack{ \ell=2 \\ \ell \text{ even} }}^{8} \calE^*_{h \omega \leq 0} (\ell) ] 
\geq 
\left( \bbP [\calV^{i_0,j_0,k_0,l_0}_{\alpha_0, \beta_0, \gamma_0}( h \omega \leq 0 )]  \right)^{4}.
\end{align}
Let now $\calB^*_{h \omega \leq 0} (\ell)$ be the ``bridge event'' that there is a dual crossing of $h \omega \leq 0$ in the horizontal strip of $R$ (i.e., the union of all horizontal shifts of $R$), from the semi-infinite line on left of $I_{\ell}$ to the semi-infinite line on the right. If for any $\ell_0 \in \{3, 5, 7 \}$,
\begin{align}
\label{eq:RSWa3}
\bbP [ \calB^*_{h \omega \leq 0} (\ell_0 ) \; | \; \bigcap_{\substack{ \ell=2 \\ \ell \text{ even} }}^{8} \calE^*_{h \omega \leq 0} (\ell)  ] \geq 1/3,
\end{align}
then the proof is trivially finished by computing
\begin{align}
\label{eq:RSWa4}
\bbP [ \calH^*_{h \omega \leq 0}(R)  ]
\geq
 \bbP [ \bigcap_{\ell = 1}^M \calB^*_{h \omega \leq 0} (\ell )  ]
\geq
 \bbP  [ \calB_{h \omega \leq 0}^* (\ell_0)  ]^M,
\end{align}
where the first step was an inclusion of events and the second one relied on shift invariance and FKG. Now, combining~\eqref{eq:RSWa1}--\eqref{eq:RSWa4}, we obtain the claim.
% with for $c=$ and $C=$.

For the rest of the proof, we thus assume that~\eqref{eq:RSWa3} does \textit{not} hold true; hence
\begin{align}
\label{eq:RSWassn2}
\bbP [ \calB^*_{h \omega \leq 0} (3)^c \cap \calB^*_{h \omega \leq 0} (7)^c  \; | \; \bigcap_{\substack{ \ell=2 \\ \ell \text{ even} }}^{8} \calE^*_{h \omega \leq 0} (\ell)  ] \geq 1/3.
\end{align}
Let now $\tilde{h} = 6 - h$, and define $\nu$ on $E$ by~\eqref{eq:def of the perco nu}.
%\begin{align*}
%\nu_{\langle u, v \rangle} =
%\begin{cases}
%0, \qquad \{ |\tilde{h}(u)|, |\tilde{h}(v)| \} = \{ 5, 7 \} \\
%B_e, \qquad |\tilde{h}(u)|= |\tilde{h}(v)| \in \{ 5, 7 \} \\
%1, \qquad \text{otherwise}.
%\end{cases}
%\end{align*}
%
%
%In terms of $\tilde{h}$ and $\nu$, the event $\calE^*_{h \omega \leq 0} (2) \cap \calB^*_{h \omega \leq 0} (3)^c  \cap \calE^*_{h \omega \leq 0} (4)$ implies that in the strip of $R$, there is a vertical primal crossing of $ | \tilde{h} | \leq 5$ with additionally $\nu_e= 1$ on edges between two absolute heights $| \tilde{h} |  = 5$, squeezed between the dual-crossings inducing $ \calE^*_{h \omega \leq 0} (2) \cap \calE^*_{h \omega \leq 0} (4)$ (just modify the vertical crossing with $h \omega \geq 1$ to wind around any cluster of $\tilde{h} \leq -5$). This crossing determines a cluster of $\nu = 1$ which is bounded (strictly) between the $h \omega \leq 0$ dual paths generating $\calE^*_{h \omega \leq 0} (2) \cap \calE^*_{h \omega \leq 0} (4)$; the right boundary of this cluster, a dual crossing of $\nu= 0$, is likewise bounded (but not strictly) and traverses between $I_{i_0+x} \cup I_{i_0 +x+ 1}$ and $L_{l_0+x} \cup L_{l_0 +x+ 1}$, where $3-i_0$.
%
Recall from the Pushing theorem the idea that $\omega = 0 \Rightarrow \nu = 0$, but $\nu$ has been modified from $\omega$ to remove information about $\sign(\tilde{h})$.
Define the event $\calW^{i,j,k,l}_{\alpha, \beta, \gamma}(\nu = 0)$ that there exists a vertical dual-crossing $\gamma^*$ of the horizontal strip of $R$, with $\nu = 0$ and furthermore such that
\begin{itemize}
\item the unique subpath of $\gamma^*$ crossing the strip of $R_{\mathsf{bot}}$ lands on $I_{i-1} \cup I_i \cup I_{i+1}$ on the bottom\footnote{Negative indices for $I_i$ are allowed here since we are on the infinite strip, and interpreted in the obvious manner.} and, if $\alpha = +1$ (resp. $\alpha = -1$, resp. $\alpha = 0$), it passes strictly farther right than the right extreme of $I_{i+5}$ (resp. passes strictly farther left that the left extreme of $I_{i-5}$, resp. remains right above $I_{i-7} \cup \ldots \cup I_{i+7}$); and
\item
the unique subpath of $\gamma^*$ crossing the strip of $R_{\mathsf{top}}$ vertically lands on $L_{l-1} \cup L_l \cup L_{l+1}$ on the top and, if $\gamma = +1$ (resp. $\gamma = -1$, resp. $\gamma = 0$), passes strictly farther right than the right extreme of $L_{l+5}$ (resp. passes strictly farther left that the left extreme of $L_{l-5}$, resp. remains right below $L_{l-7} \cup \ldots \cup L_{l+7}$); and
\item the related $\eta^*$ (the first crossing of the strip of $R_{\mathsf{mid}}$) traverses between $J_{j-1} \cup J_j \cup J_{j+1}$ and $K_{k-1} \cup K_k \cup K_{k+1}$ and, if $\beta = +1$ (resp. $\beta = -1$, resp. $\beta = 0$), passes strictly farther right than the right extreme $J_{j+5}$ (resp. passes strictly  farther left that the left extreme of $J_{j-5}$, resp. remains right above $J_{i-7} \cup \ldots \cup J_{i+7}$).
\end{itemize}

\begin{lemma}
Suppose that $\calE^*_{h \omega \leq 0} (2) \cap \calB^*_{h \omega \leq 0} (3)^c  \cap \calE^*_{h \omega \leq 0} (4)$ occurs. Then, there exists a vertical dual-crossing of the strip of $R$ with $\nu = 0$ that lands on $I_{i_0 + x} \cup I_{i_0 + x  + 1}$ on the bottom and on $L_{l_0 + x} \cup L_{l_0 + x+ x}$ on the top, where $x=3-i_0$. Furthermore, the left-most such dual-crossing induces the event $\calW^{i_0+x, j_0+x, k_0+x, l_0+x \; *}_{\alpha_0, \beta_0, \gamma_0}( \nu = 0)$.
\end{lemma}

The proof is essentially identical to that of Lemma~\ref{lem:find a crossing} (indeed $\calE^*_{h \omega \leq 0} (2) \cap \calB^*_{h \omega \leq 0} (3)^c  \cap \calE^*_{h \omega \leq 0} (4)$ means that there is a primal crossing of $h \omega \geq 1$ from $I_{3}$ to the top of the rectangle of $R$, and this primal crossing cannot cross any dual-crossing generating $\calE^*_{h \omega \leq 0} (2)$ or $\calE^*_{h \omega \leq 0} (4)$.). The obvious analogue of this lemma holds on the event $\calE^*_{h \omega \leq 0} (6) \cap \calB^*_{h \omega \leq 0} (7)^c  \cap \calE^*_{h \omega \leq 0} (8)$, with ``left'' and ``right'' reversed and $x= 7-i_0$.
Let $\gamma^*_{\mathsf{left}}$ (resp. $\gamma^*_{\mathsf{right}}$) be such left-most (resp. right-most) vertical dual-crossings; on $\calB^*_{h \omega \leq 0} (3)^c \cap \calB^*_{h \omega \leq 0} (7)^c  \bigcap_{\substack{ \ell=2 \\ \ell \text{ even} }}^{8} \calE^*_{h \omega \leq 0} (\ell)$ they exist. In that case, let $\calX $ be the quad bounded by $\gamma^*_{\mathsf{left}}$, $\gamma^*_{\mathsf{right}}$, and the top and bottom boundary segments of the strip of $R$ between them.\footnote{We will actually have to allow ``degenerate quads'', where $\gamma^*_{\mathsf{left}}$ and $\gamma^*_{\mathsf{right}}$ may intersect but not cross transversally. We will omit the related (trivial) special considerations.} Let $\Xi $ be the collection of all such sub-quads of this strip, whose
% top and bottom are contained in the top and bottom of $R$, respectively, and the 
left (resp. right) side both is of the type $\calW^{i_0+x, j_0+x, k_0+x, l_0+x \; *}_{\alpha_0, \beta_0, \gamma_0}( \nu = 0)$ with $x = 3-i_0$ (resp. $x = 7-i_0 $).  The previous lemma yields
\begin{align*}
\bigg\{ \calB^*_{h \omega \leq 0} (3)^c \cap \calB^*_{h \omega \leq 0} (7)^c  \bigcap_{\substack{ \ell=2 \\ \ell \text{ even} }}^{8} \calE^*_{h \omega \leq 0} (\ell) \bigg\}
\subset \{ \calX \in \Xi \},
\end{align*}
so combined with~\eqref{eq:RSWassn2}, we obtain
\begin{align}
\label{eq:RSWb1}
\bbP [\calX \in \Xi] \geq \tfrac{1}{3} \bbP [ \bigcap_{\substack{ \ell=2 \\ \ell \text{ even} }}^{8} \calE^*_{h \omega \leq 0} (\ell) ].
\end{align}

To conclude the proof,one then shows identically to Lemma~\ref{lem:bridging for the push lemma} that
\begin{align*}
\bbP [\calV_{h \omega \geq 1} (X) \; | \; \calX = X] \leq \tfrac{31}{32}, \qquad \text{for any } X \in \Xi,
\end{align*}
and noticing that $ \calV_{h \omega \geq 1} (X)^c = \calH_{h \omega \leq 0}^*(X) \subset \calB^*_{h \omega \leq 0} (5)$ (on the event $\calX = X$), one deduces
\begin{align}
\label{eq:RSWb2}
\bbP [ \calB^*_{h \omega \leq 0} (5) ] \geq \tfrac{1}{32} \bbP [\calX \in \Xi].
\end{align}
Finally, combining inequalities~\eqref{eq:RSWa1}--\eqref{eq:RSWa2}, \eqref{eq:RSWb1}--\eqref{eq:RSWb2} and in the end \eqref{eq:RSWa4}, we obtain the claimed inequality under the assumption~\eqref{eq:RSWassn2}. This completes the proof.
\end{proof}

\subsection{Increasing the height on the crossings}

The next step is to increase the height on the horizontal crossing. The price to pay for this increase is an increase in the size of the domain, which will be handled in later subsections.

\begin{proposition}
\label{prop:RSW height increase}
For every $\rho > 0$ and $\rho_2 \geq 3$, there exist $C = C(\rho, \rho_2), c = c (\rho, \rho_2) > 0$ and $\rho_3 = \rho_3 (\rho, \rho_2) \in \bbZ_+$ such that for all $n \in \bbN$, we have
\begin{align*}
c \bbP_{R_{\rho_2 n, \rho_2 n}}^{\pm 1} [\calV_{ \omega = 1} (R_{\rho n, n})]^C
\leq
\bbP_{\Strip(\rho_3 n)}^{\pm 1} [\calH^*_\calE (R_{\rho n, n})],
\end{align*}
where the random edge set $\calE$ is defined as in Theorem~\ref{thm:RSW}.
\end{proposition}

We work towards this proposition in several steps. In the following, for $0<s \in \Zodd$, let $\{ h \geq_B s \}$ denote the random set of edges on which $h \geq s$ and additionally for an edge $e$ between  two heights $s$, $B_e = 1$; and let $\calE_s = \calE_s(h, B)$ denote the set of edges with heights $h \geq s$ at the endpoints, with additionally $B_e = 0$ on edges $e$ between two heights $h=s$. (Hence $\{ h \geq_B 1 \} = \{ h \omega \geq 1 \}$ and the set $\calE$ in Theorem~\ref{thm:RSW} coincides with $\calE_5$.) For two rectangles $R \subset R'$, denote by $\calO^*_{\calE_s}(R, R')$ the event that there is a dual-loop in $R'$, on the dual-edges $\calE_s^*$, that surrounds $R$. We now have the following (see Figure~\ref{fig:RSW2} (left)):

\begin{figure}
\begin{center}
\raisebox{0.02\textheight}{\includegraphics[height=0.108\textheight]{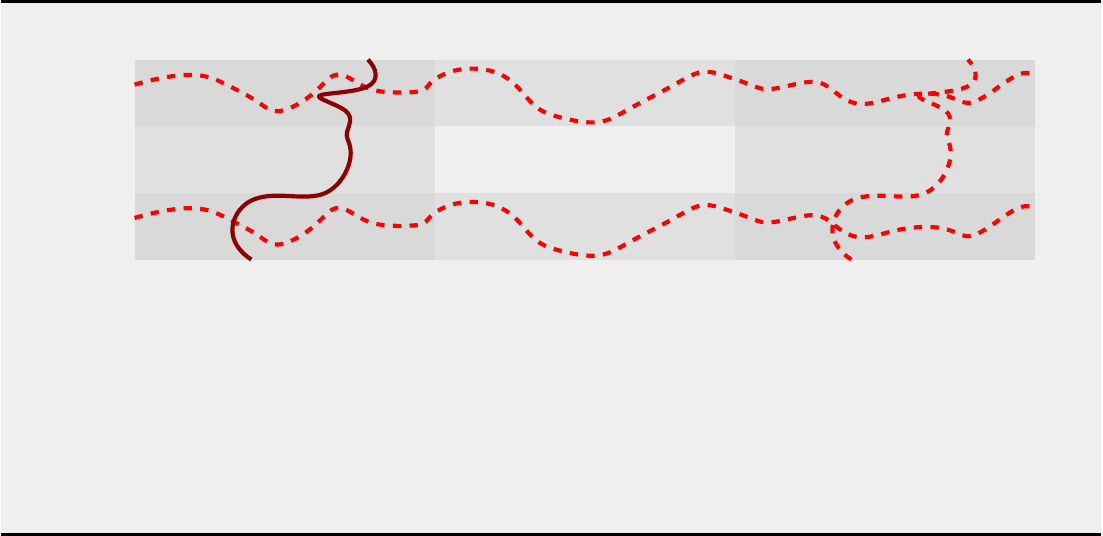}} \qquad \qquad
\includegraphics[height=0.13\textheight]{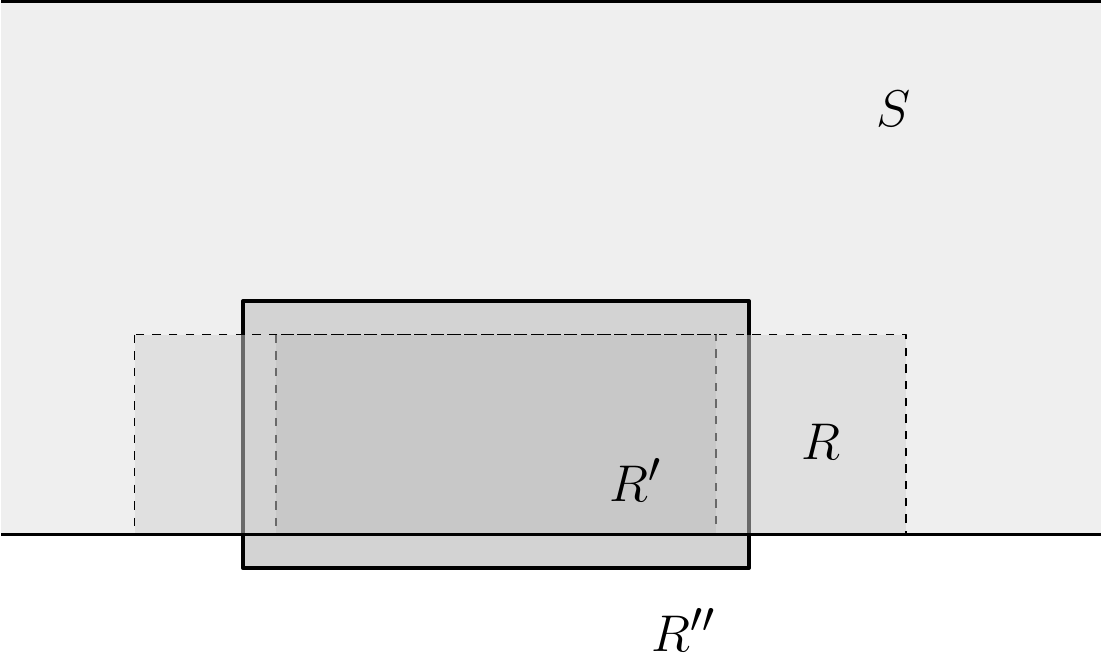}
\end{center}
\caption{
\label{fig:RSW2}
\label{fig:RSW bdary push}
Two simple illustrations on RSW results. Left:
Lemma~\ref{lem:from vert to hor crossing} tells one how to ``construct horizontal crossings from vertical ones''. A step in the proof of Lemma~\ref{lem:from vert crossing to loop} is to similarly ``construct loops''. Right: The geometric setup in the proof of Proposition~\ref{lem:RSW push boundaries}.
}
\end{figure}

%\begin{figure}
%\begin{center}
%\includegraphics[width=0.4\textwidth]{loop_lemma.pdf}
%\end{center}
%\caption{
%\label{fig:RSW2}
%Lemma~\ref{lem:from vert to hor crossing} tells one how to ``construct horizontal crossings from vertical ones''. A step in the proof of Lemma~\ref{lem:from vert crossing to loop} is to similarly ``construct loops''.
%}
%\end{figure}
%
%\begin{figure}
%\begin{center}
%\includegraphics[width=0.4\textwidth]{RSW_boundary_push.pdf}
%\end{center}
%\caption{
%\label{fig:RSW bdary push}
%The geometric setup in the proof of Proposition~\ref{lem:RSW push boundaries}.
%}
%\end{figure}

\begin{lemma}
\label{lem:from vert crossing to loop}
For every $\rho > 0$ and $\rho_2 \geq 3$, there exist $C = C(\rho, \rho_2), c = c (\rho, \rho_2) > 0$ such that the following holds. Let $n \in \bbN$ and $(R, R', S)$ be a translate of the triplet $(R_{\rho n, n}, R_{ 3 \rho n, 3n}, \Strip(\rho_2 n) )$, and let $m $ be large enough so that $S \subset \Strip(m)$. Then, for any $s \in \{ 1, 3, 5\}$, we have
\begin{align*}
c \bbP_{\Strip(\rho_2 n)}^{\pm 1} [\calV_{h \geq_B s} (R_{\rho n, 3 n})]^C
\leq
\bbP_{\Strip(m)}^{\pm 1} [\calO^*_{\calE_s} (R, R')].
\end{align*}
\end{lemma}

\begin{proof}
Denote $\bbP_{\Strip(\rho_2 n)}^{\pm 1} := \bbP_1$ and $\bbP_{\Strip(m)}^{\pm 1} =: \bbP_2$.
Denote by $\{ | h | \geq_B s \}$  the edges on which $|h| \geq s$ and additionally $B=1$ between two absolute-heights $|h|=s$. Note that this edge set is increasing in $(|h|, B)$, and that the sign of $h$ is fixed on any cluster of $\{ | h | \geq_B s \}$.
Subdivide $R'$ into three translates $R_1, R_2, R_3$ of $R_{\rho n, 3 n}$, indexed from left to right. By sign flip symmetry, then FKG and finally inclusion, we have for all $1 \leq i \leq 3$
\begin{align*}
\bbP_2 [\calV_{h \geq_B s} (R_i)] 
\geq \tfrac{1}{2} \bbP_2 [\calV_{|h| \geq_B s} (R_i)] 
\geq \tfrac{1}{2} \bbP_1 [\calV_{|h| \geq_B s} ( R_{\rho n, 3 n} )] 
\geq \tfrac{1}{2} \bbP_1 [\calV_{h \geq_B s} ( R_{\rho n, 3 n} )] .
\end{align*}
Next, let $\tilde{B}$ be equally distributed with $B$ and coupled so that $B_e=1 \Rightarrow \tilde{B}_e = 0$. In this coupling, $\{ h \geq_B s \} \subset \calE_s (h, \tilde{B})$. Combining this observation with~\eqref{eq:quad cross dual vs primal},
\begin{align*}
\bbP_2 [\calV_{h \geq_B s} (R_i)] \leq 
\bbP_2 [\calV^*_{\calE_s (h, \tilde{B})} (R_i)]  = \bbP_2 [\calV^*_{\calE_s } (R_i)]  .
\end{align*}

Then subdivide $R'$ into three translates $R_4, R_5, R_6$ of $R_{3 \rho n, n}$, indexed from bottom to top. By Lemma~\ref{lem:from vert to hor crossing}, we have for all $4 \leq i \leq 6$
\begin{align*}
\bbP_2 [\calH^*_{\calE_s} (R_i)] \geq c'  \bbP_2 [\calV^*_{\calE_s } (R_i)]^{C'}
\end{align*}
for suitable $C' = C'(\rho), c' = c' (\rho) > 0$, and by inclusion
\begin{align*}
\bbP_2 [\calV^*_{\calE_s } (R_i)] \geq \bbP_2 [\calV^*_{\calE_s } (R_2)].
\end{align*}

Finally, by inclusion and FKG for $(h, -B)$,
\begin{align*}
\bbP_{2} [\calO^*_{\calE_s} (R, R')]
&
\geq
\bbP_{2} [\calV^*_{\calE_s} (R_1) \cap \calV^*_{\calE_s} (R_3) \cap \calH^*_{\calE_s} (R_4) \cap \calH^*_{\calE_s} (R_6) ]
\\
&
\geq \bbP_2 [ \calV^*_{\calE_s} (R_1)  ] \bbP_2 [  \calV^*_{\calE_s} (R_3)] \bbP_{2} [\calH^*_{\calE_s} (R_4) ] \bbP_2 [  \calH^*_{\calE_s} (R_6) ] .
\end{align*}
The claim now follows by combining all the displayed inequalities.
\end{proof}

\begin{lemma}
For every $\rho \leq \rho_2$ and $s\in \{1, 3, 5\}$, there exist $C = C(\rho, \rho_2, s), c = c (\rho, \rho_2, s) > 0$ and $\rho_3 = \rho_3 (\rho, \rho_2, s) \in \bbZ_+$ such that, for all $n \in \bbN$, we have
\begin{align*}
c \bbP_{R_{\rho_2 n, \rho_2 n}}^{\pm 1} [\calV_{ \omega = 1} (R_{\rho n, n})]^C
\leq
\bbP_{\Strip(\rho_3 n)}^{\pm 1} [\calV_{h \geq_B s} (R_{\rho n, n})].
\end{align*}
\end{lemma}

\begin{proof}
The proof is an induction on $s$. In the base case $s=1$, by FKG and sign flip symmetry,
\begin{align*}
\bbP_{R_{\rho_2 n, \rho_2 n}}^{\pm 1} [\calV_{ \omega = 1} (R_{\rho n, n})]
\leq
\bbP_{\Strip(\rho_2 n)}^{\pm 1} [\calV_{\omega = 1} (R_{\rho n, n})]
\leq
2 \bbP_{\Strip(\rho_2 n)}^{\pm 1} [\calV_{h \omega \geq 1} (R_{\rho n, n})].
\end{align*}
Recalling that $\{ h \omega \geq 1\} = \{ h \geq_B 1 \}$, this proves the base case, with $\rho_3 = \rho_2$.

Assume that the claim holds for $s \in \{ 1, 3 \}$, i.e., there exist $C( s), c ( s),\rho_3 ( s) $ (we omit $\rho, \rho_2$ for lighter notation) such that
\begin{align}
\label{eq:ind1}
c(s) \bbP_{R_{\rho_2 n, \rho_2 n}}^{\pm 1} [\calV_{ \omega = 1} (R_{\rho n, n})]^{C(s)}
\leq
\bbP_{\Strip(\rho_3(s) n)}^{\pm 1} [\calV_{h \geq_B s} (R_{\rho n, n})],
\end{align}
and aim to prove it for $s+2$.
Denote $(R, R', R'')=:(R_{\rho n, n}, R_{\rho_2 n, \rho_2 n}, 2 R_{\rho_2 n, \rho_2 n}) $ for short; by Lemma~\ref{lem:from vert crossing to loop} and a standard necklace argument, for any $\rho_3(s+2)$ large enough (fix any for definiteness),
\begin{align}
\label{eq:ind2}
\bbP_{\Strip(\rho_3(s+2) n)}^{\pm 1} [\calO_{\calE_s}^* (R', R'')] \geq 
c' \bbP_{\Strip(\rho_3(s) n)}^{\pm 1} [\calV_{h \geq_B s} (R_{\rho n, n})]^{C'}.
\end{align}
for suitable $c', C' > 0$ depending on $\rho$ and $\rho_3(s)$.

Next, let $h \sim \bbP_{\Strip(\rho_3(s+2) n)}^{\pm 1}$, and define $\tilde{h} = s + 1 - h$ and $\tilde{\omega}$ based on $(\tilde{h}, B)$ as usual. Note that $\calE_s(h, B) = \{ \tilde{h} \tilde{\omega } \leq 0 \}$. On the event $\calO_{\calE_s}^* (R', R'')$, explore the primal-component of $ \tilde{h} \tilde{\omega } \geq 1 $ starting from $\partial R''$ and its bounding dual curves of $\tilde{h} \tilde{\omega } \leq 0$. Let $X$ be the unique unexplored component that contains $R'$ (a simply-connected domain bounded by a dual curve); by the SMP, the exploration generates, in terms of $\tilde{h}$, a boundary condition $\tilde{\xi} \preceq \{ \pm 1 \}$ on $\partial X$ (in terms of $h$, $\xi = 2- \tilde{\xi} \succeq \{ s, s+2 \}$). Now, compute
\begin{align*}
\bbP^\xi_X [\calV_{h \geq_B s+2}(R)] 
&
\geq
\bbP^{ \{ s, s+2 \} }_X [\calV_{h \geq_B s+2}(R)] \qquad \text{(by CBC for $(h, B)$)}
\\
&
=
\bbP^{ \pm 1 }_X [\calV_{h \omega \geq 1}(R)] \qquad \text{(by height shift)}
\\
&
\geq
\tfrac{1}{2} \bbP^{ \pm 1 }_X [\calV_{ \omega = 1}(R)] \qquad \text{(by sign flip symmetry)}
\\
&
\geq
\tfrac{1}{2} \bbP^{ \pm 1 }_{R'} [\calV_{ \omega = 1}(R)] \qquad \text{(by FKG for $(|h|, \omega)$).}
\end{align*}
Averaging over $X$, one obtains
\begin{align}
\label{eq:ind3}
\bbP_{\Strip(\rho_3(s+2) n)}^{\pm 1} [\calV_{h \geq_B s+2}(R) \; | \; \calO_{\calE_s}^* (R', R'')]
\geq \tfrac{1}{2} \bbP^{ \pm 1 }_{R'} [\calV_{ \omega = 1}(R)].
\end{align}
Combining~\eqref{eq:ind1}--\eqref{eq:ind3}, the claim also holds for $s+2$. This concludes the proof.
\end{proof}

\begin{proof}[Proposition~\ref{prop:RSW height increase}]
By the previous lemma (case $s=5$), there exist $C' = C'(\rho, \rho_2), c' = c' (\rho, \rho_2) > 0$ and $\rho_3 = \rho_3 (\rho, \rho_2) \in \bbZ_+$ such that 
\begin{align*}
c' \bbP_{R_{\rho_2 n, \rho_2 n}}^{\pm 1} [\calV_{ \omega = 1} (R_{\rho n, n})]^{C'}
\leq
\bbP_{\Strip(\rho_3 n)}^{\pm 1} [\calV_{h \geq_B 5} (R_{\rho n, n})].
\end{align*}
Letting again $\tilde{B}$ be equally distributed with $B$ and coupled so that $B_e=1 \Rightarrow \tilde{B}_e = 0$; note that $\{ h \geq_B 5 \} \subset \calE_5 (h, \tilde{B}) = \calE (h, \tilde{B})$. Combining with~\eqref{eq:quad cross dual vs primal},
\begin{align*}
\bbP_{\Strip(\rho_3 n)}^{\pm 1} [\calV_{h \geq_B 5} (R_{\rho n, n})] \leq
\bbP_{\Strip(\rho_3 n)}^{\pm 1} [\calV_{\calE}^* (R_{\rho n, n})].
\end{align*}
Finally, by Lemma~\ref{lem:from vert to hor crossing}, there exist $C'' = C''(\rho), c'' = c'' (\rho) > 0$ such that
\begin{align*}
c'' \bbP_{\Strip(\rho_3 n)}^{\pm 1} [\calV_{\calE}^* (R_{\rho n, n})]^{C''}
\leq
\bbP_{\Strip(\rho_3 n)}^{\pm 1} [\calH_{\calE}^* (R_{\rho n, n})].
\end{align*}
Combining the three above displayed inequalities concludes the proof.
\end{proof}

\subsection{Pushing the crossings}

In the next two subsections, we give two lemmas: one to ``push a horizontal crossing downwards'', and another to ``push the top boundary downwards''. An iterated use of these lemmas will compensate the increase in domain size that appeared in Proposition~\ref{prop:RSW height increase}.

\begin{lemma}
\label{lem:RSW push crossings}
For every $\rho> 0$, there exist constants $C= C(\rho), c= c(\rho) > 0$ such that the following holds for any $n$. Let $R = R_{ \rho n, 9n}$ and let $R'$ be a translate of the rectangle $R_{ \rho n, 7n}$ so that it lies on the bottom of $\Strip(11n)$. Then, we have
\begin{align*}
\bbP^{\pm 1 }_{\Strip(11n)} [\calH_{\calE}^* (R')] 
\geq 
c \bbP^{ \pm 1 }_{\Strip(11n)} [\calH_{\calE}^* (R)]^{C} .
\end{align*}
\end{lemma}

\begin{proof}
Let us denote $\bbP^{ \pm 1 }_{\Strip(11n)} =: \bbP$.
Let $R_{\mathsf{top}}$, $R_{\mathsf{mid}}$, and $R_{\mathsf{bot}}$ denote the top, middle and bottom thirds of the rectangle $R$. Let $\calH^{\mathsf{tm}} $ (resp. $\calH^{\mathsf{mb}} $, resp. $\calH^{\mathsf{all}}$) be the subevent of $\calH_{\calE}^* (R)$ that there is a horizontal $\calE^*$ dual-crossing of $R$ that is contained in $R_{\mathsf{top}} \cup R_{\mathsf{mid}}$ (resp. contained in $R_{\mathsf{mid }} \cup R_{\mathsf{bot}}$, resp. visits all three sub-rectangles). By the Union Bound,
\begin{align*}
\bbP [\calH_{\calE}^* (R)] \leq \bbP [\calH^{\mathsf{tm}}] + \bbP [\calH^{\mathsf{mb}}] + \bbP [\calH^{\mathsf{all}}],
\end{align*}
so at least one of the events on the right-hand side has a probability larger or equal to $\tfrac{1}{3} \bbP [\calH_{\calE}^* (R)]$. If $\calH^{\mathsf{all}}$ is not such an event, then by symmetry, $\calH^{\mathsf{mb}}$ is one, and with an inclusion of events, we deduce
\begin{align*}
%\label{eq:RSWpush1}
\bbP [\calH_{\calE}^* (R')] \geq \bbP [\calH^{\mathsf{mb}}] \geq \tfrac{1}{3}\bbP [\calH_{\calE}^* (R)],
\end{align*}
which proves the claim.
 If, $\calH^{\mathsf{all}}$ is such an event
%\begin{align*}
%\bbP [\calH^{\mathsf{all}}_{\omega = 1}(R)] \geq \tfrac{1}{3} \bbP [\calH_{\omega = 1}(R)]
%\end{align*}
then we infer by inclusion of events that
\begin{align*}
\bbP [ \calV^*_{\calE}(R^{\mathsf{mid}}) ] \geq \bbP [\calH^{\mathsf{all}}] \geq \tfrac{1}{3} \bbP [\calH_{\calE}^* (R)],
\end{align*}
and by another inclusion of events and Lemma~\ref{lem:from vert to hor crossing} that
\begin{align*}
%\label{eq:RSWpush2}
\bbP [\calH_{\calE}^* (R')] \geq \bbP [\calH_{\calE}^* ( R^{\mathsf{mid}} )] \geq c' \bbP [ \calV^*_{\calE}(R^{\mathsf{mid}}) ]^{C'},
\end{align*}
for suitable $C'= C'(\rho), c'= c'(\rho) > 0$. This concludes the proof. 
%Combining the two cases presented in the three previous equalities,
%%~\eqref{eq:RSWpush1}--\eqref{eq:RSWpush2}, 
%we obtain the claimed relation with $c=  c' (\tfrac{1}{3})^C$ and $C = C'$.
\end{proof}

\subsection{Pushing the boundaries}

\begin{proposition}
\label{lem:RSW push boundaries}
For every $\rho, \epsilon> 0$ and $\rho_2 \geq 1$, there exists
$c= c(\rho, \rho_2, \epsilon) > 0$ such that the following holds for any $n \geq 4/\epsilon$. Let $R$ be a translate of the rectangle $R_{\rho n, n}$ on the bottom of $\Strip(\rho_2 n)$, let $w \leq \rho n$ and  $R' := R_{w, n}$ and $R'' = R_{ w+\epsilon n, (1+ \epsilon )n}$. Then, we have 
\begin{align*}
c \bbP^{ \pm 1 }_{\Strip( \rho_2 n )} [\calH_{\calE}^*(R)] 
\leq
\bbP^{ \pm 1 }_{R''} [\calH_{\calE}^* (R')] 
.
\end{align*}
\end{proposition}

The rest of this subsection constitutes the proof of this result.
Throughout the subsection, denote $\Strip( \rho_2 n )=: S$ and assume that $R'$ and $R''$ are shifted so that $R'$ is still in the middle of $R''$ and that $R'$ lies inside $R$ (Figure~\ref{fig:RSW bdary push} (right)). Let $\calE'$ denote the set of edges between absolute-heights $|h|\geq 5$, with additionally $B = 0$ on edges between two absolute-heights $|h| = 5$. 
Let $\calU$ be the event that there is a horizontal dual-crossing of $R'$ on the dual-edges crossing $\calE'$, i.e., generating $\calH_{\calE'}^* (R')$, whose component of $\omega =1 $ (indeed, by the bound on the dual degree, the primal-edges crossed by such a dual-path are contained in a single component of $\omega =1 $) is furthermore disconnected from the top and side boundaries of $R''$ (hence by some dual-path of $\omega = 0$).

\begin{lemma}
\label{lem:P[U] upper bd}
In the setup and notation given above, we have
\begin{align}
\label{eq:push1}
\bbP^{  \pm 1 }_S [\calU ] 
% \leq  \bbP^{ \pm 1 }_{R''} [ \calH_{\omega = 1} (R') ]
 \leq \bbP^{ \pm 1 }_{R''} [ \calH_{\calE'}^* (R') ].
\end{align}
\end{lemma}

\begin{proof}
Explore the primal clusters of $\omega=1$ in $R'' \cap S$ that are adjacent to the top and side boundaries of $R''$, and the dual curves of $\omega= 0$ bounding them. % Explore also $h$ on the primal clusters.  
Let $\calX \subset R'' \cap S$ be the random set edges where $\omega$ is explored; the event $\calX = X$ can thus be inferred from $\omega_{|X}$. %Note that $\calU$ can only occur if at least one component of $X^c$ crosses $R'$ horizontally.
Let $\Xi(X)$ be the set of such $ \alpha$ that $ \omega_{|X} = \alpha $ generates the event $\calX = X$, so
%Thus,
%\begin{align*}
%\bbP^{ \{ \pm 1 \} }_S [ \calB ] = \sum_{(X, \xi, \beta) \in \Xi} \bbP^{ \{ \pm 1 \} }_S [ h_{|X} = \xi, \omega_{|X}= \alpha  ],
%\end{align*}
%and
\begin{align*}
\bbP^{  \pm 1 }_S [\calU \cap \{ \calX = X \} ]
=
\sum_{\alpha \in \Xi (X)} \bbP^{ \pm 1 }_S [ \omega_{|X}= \alpha  ] \bbP^{ \pm 1 }_S [ \calH^*_{\calE' \cap X^c}(R') \; | \; \omega_{|X}= \alpha  ].
\end{align*}
%where $\calH_{\calE}^*(R'\cap X^c)$ denotes the event that there is a dual-crossing generating $\calH_{\calE}^*(R')$ in $X^c$.
Now, by the SMP, the event $ \omega_{|X}= \alpha $ generates a boundary condition $\{ \pm 1 \}$ for the unexplored components, the height configurations in the single individual components being independent, and this law for heights on $X^c$ is also obtained by setting $\omega_{|X} = 0$. Compute
\begin{align*}
\bbP^{ \pm 1 }_S [ \calH^*_{\calE' \cap X^c}(R')  \; | \; \omega_{|X}= \alpha  ] 
&= 
\bbP^{ \pm 1 }_S [ \calH^*_{\calE' \cap X^c}(R')  \; | \; \omega_{|X}= 0 ]  \qquad \text{(by SMP)}\\
&
= 
\bbP^{ \pm 1 }_{R''} [ \calH^*_{\calE' \cap X^c}(R')  \; | \; \omega_{|X}= 0] \qquad \text{(by SMP)} \\
& \leq 
\bbP^{ \pm 1 }_{R''} [ \calH^*_{\calE' \cap X^c}(R')]  \qquad \text{(by FKG for $(|h|, -B)$)} \\
& \leq \bbP^{ \pm 1 }_{R''} [ \calH^*_{\calE' }(R') ]  \qquad \text{(by inclusion).}
\end{align*}
Combining the two previously displayed equations yields
\begin{align*}
\bbP^{ \pm 1 }_S [\calU \cap \{ \calX = X \} ] \leq 
%\sum_{\alpha \in \Xi (X)} \bbP^{ \{ \pm 1 \} }_S [ \omega_{|X}= \alpha  ] 
%\bbP^{ \pm 1 }_{R''} [ \calH_{\omega = 1} (R') ]
 \bbP^{ \pm 1 }_S [ \calX = X  ] \bbP^{ \pm 1 }_{R''} [ \calH_{\calE'}^* (R') ],
\end{align*}
and summing this over all possible values of $X$ gives~\eqref{eq:push1}.
\end{proof}

\begin{lemma}
\label{lem:P[U] lower bd}
In the setup and notation given in Proposition~\ref{lem:RSW push boundaries} and right below it, there exists $c'= c'(\rho, \rho_2, \epsilon) > 0$ such that
\begin{align}
\label{eq:push2}
\bbP^{ \pm 1 }_S [ \calU  ] 
\geq c' \bbP^{ \pm 1 }_S [ \calH_{\calE'}^* (R) ].
\end{align}
\end{lemma}

\begin{proof}
Set
\begin{align*}
\nu_{\langle u, v \rangle} =
\begin{cases}
0, \qquad \{ |{h}(u)|, |{h}(v)| \} = \{ 5, 7 \} \\
B_{\langle u, v \rangle} , \qquad |{h}(u)|= |{h}(v)| \in \{ 5, 7 \} \\
1, \qquad \text{otherwise}
\end{cases}
\end{align*}
and note that under the boundary condition $\bbP^{\{ \pm 1 \}}_{S} =: \bbP$, $\calH_{\calE'}^*(R') = \calH_{\nu=0}^*(R')$.
Thus, explore the primal-cluster of $\nu=1$ in $R'$ that is adjacent to the bottom of $R'$, and the primal curves of $\nu=0$ bounding it; this reveals the lowest crossing inducing $\calH_{\nu=0}^*(R')$ (if such exists). Let us also explore $h$ on the primal-clusters, hence in $[-5, 5]$. Let $\calY = (V(\calY), E(\calY))$ be the set where $h$ and $\nu$ are thus explored; the event $\calY = Y$ can be inferred from $ \nu_{|Y}$; let $\Psi$ be the set of $(Y, \chi, \beta)$ such that $\{ h_{|Y} = \chi, \nu_{|Y}= \beta \}$ induces $\calY = Y$ and $\calH_{\nu=0}^*(R')$, so
\begin{align*}
\bbP [ \calH_{\nu=0}^*(R') ] 
=
\sum_{(Y, \chi, \beta) \in \Psi} \bbP [ h_{|Y} = \chi, \nu_{|Y}= \beta  ]
\end{align*}
and
\begin{align*}
\bbP [\calU ]
=
\bbP [ \calU \cap \calH_{\nu=0}^*(R') ]
=
\sum_{(Y, \chi, \beta) \in \Psi} \bbP [  h_{|Y}  = \chi, \nu_{|Y}= \beta  ] \bbP [ \calU \; | \; h_{|Y} = \chi, \nu_{|Y}= \beta ]
\end{align*}
Note also that on the unique $\calH_{\nu=0}^*(R')$ crossing found in $\beta$, the sign of $\chi$ is constant (due to the bound on dual degree); let us assume it is $+1$ for definiteness (without loss of generality, due to sign flip symmetry). Then, given $\{ h_{|Y} = \chi, \nu_{|Y}= \beta \}$, $\calU $ happens at least if there is a dual path (``bridge'') of $h \omega \leq 0$ in $(R'' \cap S) \setminus Y$, that disconnects the the bottom of the smaller rectangle $R'$ from the top, left, and right sides of the bigger rectangle $R'' \cap S$. We denote this event by $\calB_{h \omega \leq 0}^*(Y)$; then, in equations
\begin{align*}
\bbP [ \calU \; | \; h_{|Y} = \chi, \nu_{|Y}= \beta ] \geq
\bbP [ \calB_{h \omega \leq 0}^*(Y) \; | \; h_{|Y} = \chi, \nu_{|Y}= \beta ] 
\end{align*}

We will show below that there exists $c' = c'(\rho, \rho_2, \epsilon)$ such that, for any $(Y, \chi, \beta) \in \Psi$,
\begin{align}
\label{eq:push for cross}
\bbP [ \calB_{h \omega \leq 0}^*(Y) \; | \;  h_{|Y} = \chi, \nu_{|Y}= \beta ] \geq c'.
\end{align}
Combining the four previously displayed inequalities then yields~\eqref{eq:push2}.

It thus remains to prove Equation~\eqref{eq:push for cross}. Let $S'$ be the discrete domain obtained by ``cutting out of $S$'' the explored primal-clusters of $\nu=1$ along their dual-path boundaries.
Let $\chi'$ be the unique admissible boundary condition for $\heightfcns_{S'}$ on $\partial S'$ which is $\{ \pm 1 \}$ on $\partial S \cap \partial S'$ and either $\{ 5, 7 \}$ or $\{ -5, -7 \}$ (depending on the neighbouring values of $\chi$) on the unexplored tips of $E(Y)$ edges, and coincides with $\chi$ on $V(Y) \cap \partial S'$. By the SMP,
\begin{align*}
\bbP [  \calB_{h \omega \leq 0}^*(Y)  \; | \; h_{|Y}  = \chi, \nu_{|Y}=\beta ]
= \bbP^{ \chi' }_{S'} [  \calB_{h \omega \leq 0}^*(Y)   ] .
\end{align*}
Let then $L_{\mathsf{left}}$, $L_{\mathsf{mid}}$ and $L_{\mathsf{right}}$ now be three disjoint lozenges/squares of the height of $S$, and immediately to the right of $R'$, and let $L = L_{\mathsf{left}} \cup L_{\mathsf{mid}} \cup L_{\mathsf{right}}$. By approximating $\bbP^{\chi'}_{S'}$ (in the unique infinite unexplored component) by quads, Corollary~\ref{cor:narrow quad triple cross} gives
\begin{align}
\label{eq:loz cross}
\bbP^{ \chi' }_{S'} [\calV^*_{h \omega \leq 0 }(L)] \geq \tfrac{1}{8}.
\end{align}
Let $L'$ similarly consist of three big lozenges immediately to the right of $Y$. 
By symmetry and the FKG (the events below are decreasing in $(h,B)$)
\begin{align*}
\bbP^{ \chi' }_{S'} [\calV^*_{h \omega \leq 0 }(L) \cap \calV^*_{h \omega \leq 0 }(L')]  \geq \tfrac{1}{64}
\end{align*}
We then wish to show
\begin{align}
\label{eq:punchline}
\bbP^{ \chi' }_{S'} [  \calB^*_{h \omega \leq 0}(Y) \; | \; \calV^*_{h \omega \leq 0 }(L) \cap \calV^*_{h \omega \leq 0 }(L') ] \geq c''(\rho, \rho_2, \epsilon),
\end{align}
To do this, explore the primal clusters of $h\omega \geq 1$ adjacent to the right (resp. left) side of $L$ (resp. $L'$), and $\omega$ on the dual paths of $h \omega \leq 0$ bounding them. On  $\calV^*_{h \omega \leq 0 }(L) \cap \calV^*_{h \omega \leq 0 }(L')$, this exploration will find the right-most (resp. left-most) dual path generating the event $\calV^*_{h \omega \leq 0 }(L)$ (resp. $\calV^*_{h \omega \leq 0 }(L')$). Let $D$ denote the unexplored domain between these dual-paths. By the SMP, the exploration generates a boundary condition $\chi''$ for $D$ which extends $\chi'$ with something smaller than $\{ \pm 1\}$ on the vertical crossings. 
Equation~\eqref{eq:punchline} is then a direct consequence of Proposition~\ref{prop:push}.
This concludes the proof.
\end{proof}

\begin{proof}[Proposition~\ref{lem:RSW push boundaries}]
By lemmas~\ref{lem:P[U] upper bd}--\ref{lem:P[U] lower bd}, there exists $c'(\rho, \rho_2, \epsilon) > 0$ such that
\begin{align*}
c' \bbP^{ \pm 1 }_S [ \calH_{\calE'}^* (R) ]
 \leq \bbP^{ \pm 1 }_{R''} [ \calH_{\calE'}^* (R') ].
\end{align*}
The claim then follows since, by inclusion and sign flip symmetry, respectively,
\begin{align}
\label{eq:calE vs calE'}
\bbP^{ \pm 1 }_S [ \calH_{\calE}^* (R) ]
\leq
\bbP^{ \pm 1 }_S [ \calH_{\calE'}^* (R) ]
\qquad
\text{and}
\qquad
\bbP^{ \pm 1 }_{R''} [ \calH_{\calE'}^* (R') ]
\leq
2 \bbP^{\pm 1 }_{R''} [ \calH_{\calE}^* (R') ].
\end{align}
\end{proof}

\subsection{Concluding the proof of Theorem~\ref{thm:RSW}}

We start by combining the two previous subsections to iteratively push the crossings and boundaries downwards. 

\begin{corollary}
\label{cor:RSW strip to strip}
For any $\rho>0$, there exist constants $c = c(\rho), C=C(\rho) > 0$ such that for any $n$ and $m = \lceil \tfrac{7n}{9} \rceil$
\begin{align*}
c \bbP^{ \pm 1 }_{\Strip(11n)} [\calH_{\calE}^* (R_{\rho n, 9n})]^C \leq
\bbP^{\pm 1 }_{\Strip(11m)} [\calH_{\calE}^* (R_{ \rho n, 9m})].
\end{align*}
\end{corollary}

\begin{proof}
Note that $9m \geq 7n$, and let $R$ be a translate of $R_{\rho n, 9m}$ on the bottom of $\Strip(11n)$. By Lemma~\ref{lem:RSW push crossings},
\begin{align*}
c'(\rho) \bbP^{ \pm 1 }_{\Strip(11n)} [\calH_{\calE}^* (R_{ \rho n, 9n})]^{C'(\rho)}
\leq
\bbP^{ \pm 1 }_{\Strip(11n)} [\calH_{\calE}^*  (R)]. 
\end{align*}
Note also that $\rho n = \rho' m$, where $\rho' \leq \tfrac{9}{7} \rho$. Thus, by Proposition~\ref{lem:RSW push boundaries}
\begin{align*}
c''(\rho) \bbP^{ \pm 1 }_{\Strip(11n)} [\calH_{\calE}^* (R)] 
\leq
\bbP^{ \pm 1 }_{R_{(\rho'+1)m,11m}} [\calH_{\calE}^*  (R_{\rho' m, 9m})].
\end{align*}
Finally, by the FKG for $(|h|, \omega)$,
\begin{align*}
\bbP^{ \pm 1 }_{R_{(\rho'+1)m,11m}} [\calH_{\calE'}^*  (R_{\rho' m, 9m})] \leq \bbP^{ \pm 1 }_{\Strip(11m)} [\calH_{\calE'}^* (R_{ \rho' m, 9m})].
\end{align*}
Combining the  displayed inequalities and arguing as in~\eqref{eq:calE vs calE'} proves the claim.
\end{proof}

\begin{proof}[Theorem~\ref{thm:RSW}]
First, by Proposition~\ref{prop:RSW height increase} (and arguing as in~\eqref{eq:calE vs calE'}, which will be done repeatedly below), there exist $c', C', \rho_3$ depending on $\rho, \rho_2$ and such that
\begin{align*}
c' \bbP^{ \pm 1 }_{R{\rho_2 n, \rho_2 n  }} [\calV_{\omega = 1}(R_{\rho n, n})]^{C'}
\leq
\bbP^{ \pm 1 }_{\Strip(\rho_3 n)} [\calH^*_{\calE'} (R_{\rho n, n})] .
\end{align*}
Next, if the ratio $\rho_3$ of the heights of the strip and rectangle on the right-hand side satisfies $\rho_3 \geq 11/9$, let $n' = \lceil \rho_3 n/11 \rceil$; if $\rho_3 < 11/9$ let $n' = \lceil n/9 \rceil$. In both cases, let $\rho'$ then be such that $\rho' n' = \rho n$. By FKG and inclusion of events,
\begin{align*}
\bbP^{ \pm 1 }_{\Strip(\rho_3 n)} [\calH_{\calE'}^*  (R_{\rho n, n})] 
\leq
\bbP^{ \pm 1 }_{\Strip(11 n')} [\calH_{\calE'}^*  (R_{\rho' n', 9n'})].
\end{align*}
We next wish to apply Corollary~\ref{cor:RSW strip to strip} iteratively to make the height of the strip and rectangle on the right-hand side smaller and smaller. For this purpose, note that these dimensions originally depend on $\rho$ and $\rho_2$.
Thus, after a number of iterations depending additionally on $\epsilon$, we obtain
\begin{align*}
c''(\rho, \rho_2, \epsilon) \bbP^{ \pm 1 }_{\Strip(11 n')} [\calH_{\calE'}^*  (R_{\rho' n, 9n'})]^{ C'' (\rho, \rho_2, \epsilon) }
\leq
\bbP^{ \pm 1 }_{\Strip(11 \epsilon n / 9)} [\calH_{\calE'}^* (R_{\rho n, \epsilon n})].
\end{align*}
Finally by Lemma~\ref{lem:RSW push boundaries} and the FKG,
\begin{align*}
c '''(\rho, \rho_2, \epsilon) \bbP^{ \pm 1 }_{\Strip(11 \epsilon n / 9)} [\calH_{\calE'}^* (R_{\rho n, \epsilon n})] \leq \bbP^{ \pm 1 }_{R_{(\rho' + \epsilon ) n, 2\epsilon n}} [\calH_{\calE}^* ( R_{\rho'n, \epsilon n} )].
\end{align*}
Combining the four displayed inequalities concludes the proof.
\end{proof}

\section{Proofs of the main results}
\label{sec:proof of main results}

\subsection{Proof of the renormalization inequality}

Similarly to~\eqref{eq:def of an} let us define
\begin{align*}
%b_n := \bbP^{\{5, 7 \}}_{L_{Rn}} [\calO^*_{\omega =0}(L_n)],
b_n := \bbP^{\pm 1}_{L_{Rn}} [\calO^*_{\calE}(L_n)],
\end{align*}
where $\calE$ is defined as in Theorem~\ref{thm:RSW} and
where $\calO^*_{\calE}(L_n)$ denotes the existence of a dual loop around $L_n$, crossing only $\calE$  primal edges.
% These probabilities are slightly easier to handle than ~\eqref{eq:def of an}, but they satisfy the same phase dichotomy.
The aim of this subsection is to prove Proposition~\ref{thm:renorm} about the quantities $b_n$.
%
%\begin{lemma}
%\label{lem:loop dichotomy}
%%Fix $1 \leq \mathbf{c} \leq 2$.
%There exist $C, c, \alpha > 0$ such that either
%\begin{align*}
%\text{i)} \quad b_n \geq c \quad \text{for all $n$; or} \qquad
%\text{ii)} \quad b_n \leq C \exp(-c n^\alpha) \quad \text{for all $n$}.
%\end{align*}
%\end{lemma}
%
Recall that this was shown in Section~\ref{subsec:outline of main thm} to imply the analogue of the dichotomy theorem~\ref{thm:loop dichotomy} for $b_n$; let us quickly finish Theorem~\ref{thm:loop dichotomy} in its stated form.

\begin{proof}[Theorem~\ref{thm:loop dichotomy}]
It suffices to show that there exist $c', C' > 0$ such that
\begin{align}
\label{eq:two-side bound}
c' a_{n }^{C'} \leq b_n \leq a_n 
\qquad \text{for all } n \in \bbN.
\end{align}
The former inequality follows readily from Theorem~\ref{thm:RSW} and FKG. The latter follows since, due to the bound on dual degree, $\calO^*_{\calE}(L_n) \subset \calO_{\omega = 1}(L_n)$.
\end{proof}

\begin{proof}[Proposition~\ref{thm:renorm}]
Fix $n$ and denote by $c_i$ positive constants (independent of $n$). By Corollary~\ref{lem:crack alternative}, there exists $c_0$ such that at least one of the following two holds:
\begin{align}
\label{eq:horiz frozen cross}
\bbP^{ \pm 1 }_{R_{100 n, 100 n}} [\calH^*_{\omega = 0}(R_{10 n, n/2 })] & > c_0; \quad \text{or}\\
\label{eq:horiz high cross}
\bbP^{ \pm 1 }_{ R_{5n/2, n}} [\calH_{\calE}^* ( R_{2 n,  n/2 } )] & > c_0.
\end{align}
If~\eqref{eq:horiz high cross} holds true, then, constructing $\calO^*_{\calE}(L_n)$ from four suitable translated and rotated versions of $\calH_{\calE}^* ( R_{2n, n/2} )$, one obtains
\begin{align*}
b_n = \bbP^{ \pm 1 }_{L_{3n}} [\calO^*_{\calE}(L_n)] > c_1,
\end{align*}
which is the first alternative of Proposition~\ref{thm:renorm}.

The task is now to prove the second alternative, assuming~\eqref{eq:horiz frozen cross}. Note that by the FKG for $(|h|, \omega)$, for any $D \subset L_{42n}$ (note that $L_{42n}$ appears in the definition of $b_{14n}$), and any translate $R$ of $R_{w, n/2 }$ where $w \leq 10 n$, we then have
\begin{align}
\label{eq:crack 2}
\bbP^{ \pm 1 }_{D} [\calV_{\omega = 1} (R) ] < 1-c_0.
\end{align}
We will consider the following geometric setup. Let $L_{3n}^{\mathsf{left}} = L_{3n}^{0}, L_{3n}^1, \ldots, L_{3n}^4 = L_{3n}^{\mathsf{right}}$ be five horizontal translates of $L_{3n}$, so that each $L_{3n}^{i+1}$ is immediately to the right of $L^i_{3n}$. Let $L_{kn}^{\mathsf{left}}$ for $k \in \bbN$ (resp. $L_{kn}^{\mathsf{right}}$) denote translates of $L_{kn}$ cocentric with $L_{3n}^{\mathsf{left}}$ (resp. $L_{3n}^{\mathsf{right}}$); hence the horizontal span from the left side of $L_{2n}^{\mathsf{left}}$ to the right side of $L_{2n}^{\mathsf{right}}$ is $28n$ hexagons (generally: horizontal translational periods). Finally, let all the above mentioned lozenges/squares be translated so that $L_{2n}^{\mathsf{left}}$ and $L_{2n}^{\mathsf{right}}$ both just fit inside $L_{14n}$, which is cocentric with $L_{42n}$; recall the definition
\begin{align*}
b_{14 n} = \bbP_{L_{42n}}^{\pm 1} [\calO^*_{\calE} (L_{14n})].
\end{align*}
We will denote below, e.g., $\calO_{\tilde{\omega } = 0 }^* (L_{n}^{\mathsf{left}}, L_{2n}^{\mathsf{left}})$ for the event that $\calO_{\tilde{\omega } = 0 }^* (L_{n}^{\mathsf{left}})$ occurs, and furthermore a dual-loop generating this event can be found in $L_{2n}^{\mathsf{left}}$.

\begin{figure}
\begin{center}
\includegraphics[width=0.386\textwidth]{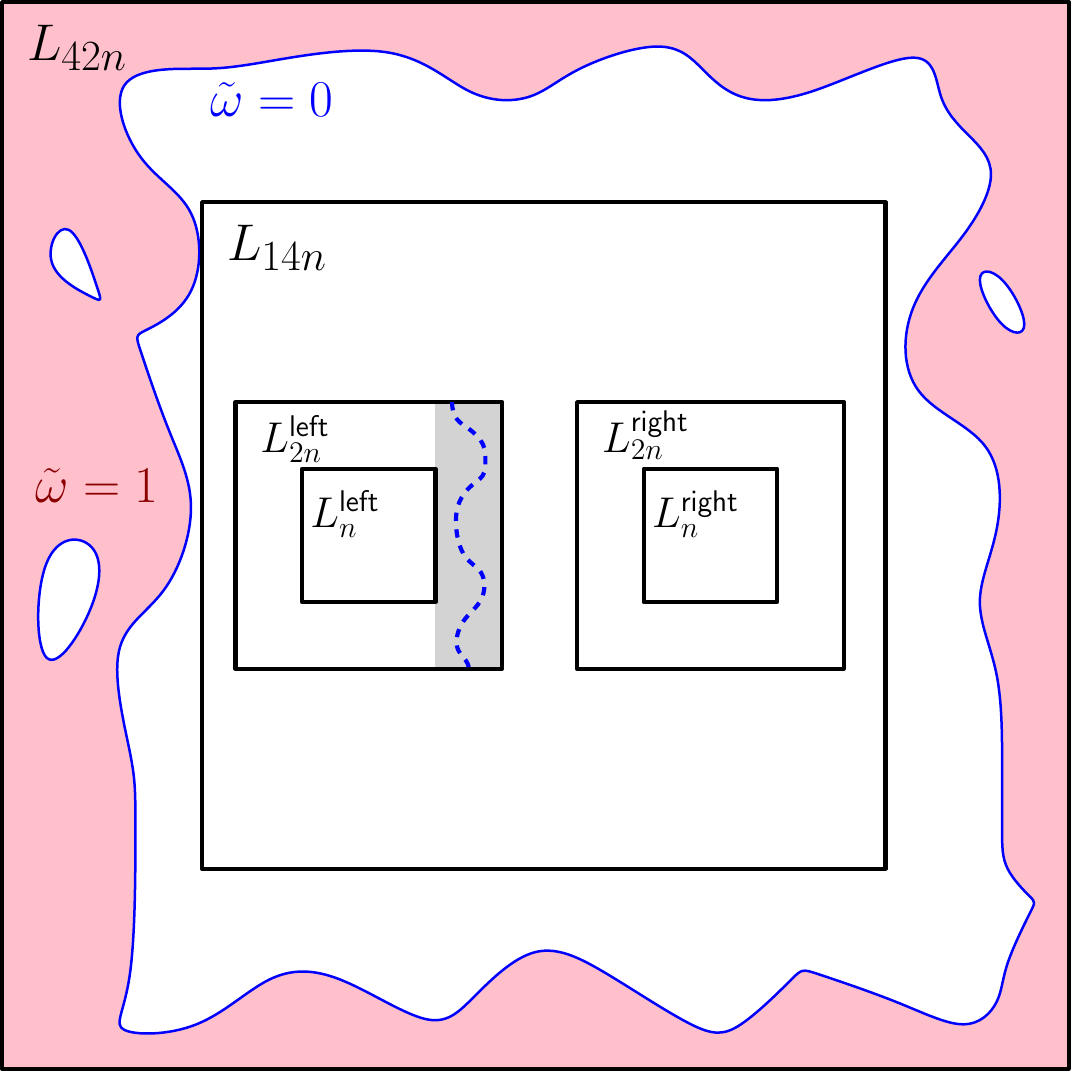} \qquad
\includegraphics[width=0.4\textwidth]{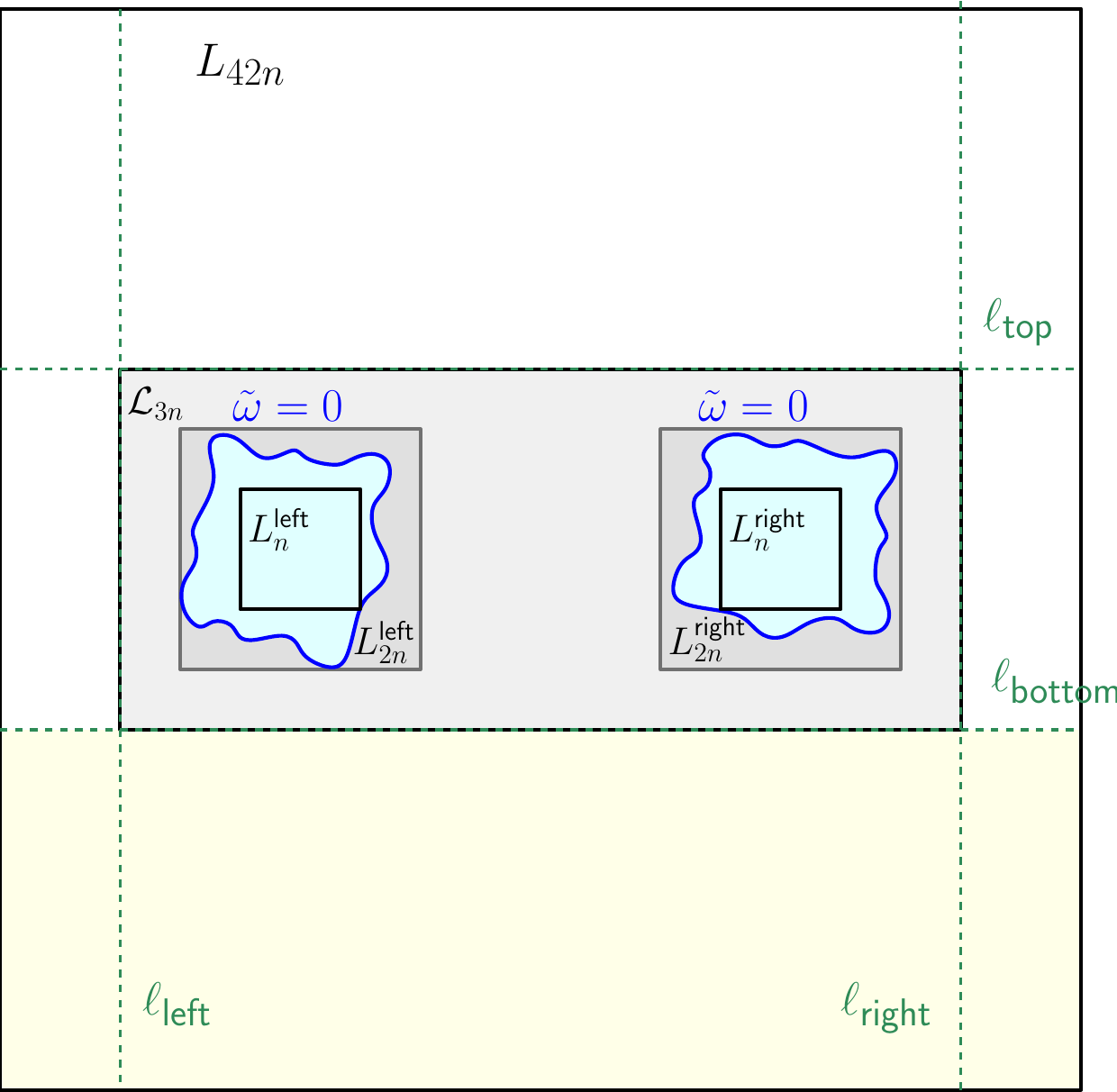}
\end{center}
\caption{
\label{fig:renorm small loops}
\label{fig:renorm boundary push}
Two geometric argument in the proof of the renormalization inequality. \textbf{Left:} inside the domain bounded by an $\tilde{\omega} =0$ dual curve, by SMP and~\eqref{eq:horiz frozen cross}, there is a uniformly positive probability for an $\tilde{\omega}=0$ dual crossing in the long direction of each of the eight rectangles forming the sides of the annuli $L_{2n}^{\mathsf{left}} \setminus L_{n}^{\mathsf{left}}$ and $L_{2n}^{\mathsf{right}} \setminus L_{n}^{\mathsf{right}}$, as illustrated on the right side of $L_{2n}^{\mathsf{left}} \setminus L_{n}^{\mathsf{left}}$. For a simpler illustration, we have drawn the case when $L_m$ are squares, and the squares and the distances between them are not in scale. \textbf{Right:} by Proposition~\ref{prop:push}, there is a uniformly positive probability that the explored $\tilde{\omega } = 0 $ dual-loops (in blue) are separated by dual-curve(s) of $h \omega \leq 0$ from the area below the bottom line $\ell_{\mathsf{bottom}}$ of $\calL_{3n}$ (in light yellow). Arguing similarly for the lines marking the top, left, and right sides of $\calL_{3n}$, one obtains~\eqref{eq:unif pos separation property}.
}
\end{figure}

Next, let $(h, B, \omega) \sim \bbP_{L_{42n}}^{\pm 1} =: \bbP$ and set $\tilde{h} = 6-h$ and define $\tilde{\omega}$ via $(\tilde{h}, B)$. Then, with the above boundary condition, $\calO^*_{\calE} (L_{14n}) = \calO^*_{\tilde{\omega} = 0} (L_{14n})$.
On this event, by Lemma~\ref{lem:SMP with percolation}, exploring the boundary component of $\tilde{\omega} = 1$ and its bounding $\tilde{\omega} =0$ dual curves generates a $\tilde{h} \in \{ \pm 1 \}$ boundary condition in the unexplored domains. Then, by~\eqref{eq:crack 2} and the FKG for $\omega$, (see Figure~\ref{fig:renorm small loops} (left))
\begin{align*}
\bbP [\calO_{\tilde{\omega } = 0 }^* (L_{n}^{\mathsf{left}}, L_{2n}^{\mathsf{left}})\ \cap \calO_{\tilde{\omega } = 0 }^* (L_{n}^{\mathsf{right}}, L_{2n}^{\mathsf{right}}) \; | \; \calO^*_{\tilde{\omega} =0 } (L_{14n})] \geq c_1^8,
\end{align*}
so
\begin{align}
\bbP [\calO_{\tilde{\omega } = 0 }^* (L_{n}^{\mathsf{left}}, L_{2n}^{\mathsf{left}})\ \cap \calO_{\tilde{\omega } = 0 }^* (L_{n}^{\mathsf{right}}, L_{2n}^{\mathsf{right}}) ] \geq 
c_1^8 b_{14 n}.
\end{align}

Next, let us detect the occurrence of the intersection $\calI := \calO_{\tilde{\omega } = 0 }^* (L_{n}^{\mathsf{left}}, L_{2n}^{\mathsf{left}})\ \cap \calO_{\tilde{\omega } = 0 }^* (L_{n}^{\mathsf{right}}, L_{2n}^{\mathsf{right}})$ by exploring  $\tilde{\omega}$ inside $L_{n}^{\mathsf{left}}$ and $L_{n}^{\mathsf{right}}$, and then outward up to the innermost dual-loops of $\tilde{\omega} = 0$ surrounding them. For events that only depend on the exterior of these dual-loops, by the SMP, this exploration generates a boundary condition $\tilde{h} \in {\pm 1}$ (i.e., $h \in \{ 5, 7 \}$) on the dual-loops. Let $\calL_{3n}$ be the ``long lozenge''/long rectangle, obtained as the union of all horizontal translates of $L_{3n}$, when it is sled from $L^{\mathsf{left}}_{3n}$ to $L^{\mathsf{right}}_{3n}$, and let $\calS_{h \omega \leq 0}^*$ be the event that $\calI $ occurs and furthermore there is a dual path or dual-paths of $h \omega \leq 0$ that separate the above-mentioned two innermost loops from $\partial \calL_{3n}$. Using Proposition~\ref{prop:push} and the FKG for $(h, B)$ (in which crossings of $h \omega \leq 0$ are decreasing) and arguing as illustrated in Figure~\ref{fig:renorm boundary push} (right),
\begin{align}
\label{eq:unif pos separation property}
\bbP [\calS_{h \omega \leq 0}^* \; | \; \calI ] \geq c_2.
\end{align}

We next claim that
\begin{align}
\label{eq:split domains}
\bbP [ \calO_{h \omega \leq 0}^* (L^{\mathsf{left}}_{2n}, L^{\mathsf{left}}_{3n}) \cap \calO_{h \omega \leq 0}^* (L^{\mathsf{right}}_{2n}, L^{\mathsf{right}}_{3n}) \; | \;  \calI \cap \calS_{h \omega \leq 0}^* ] \geq c_3.
\end{align}
Assuming this for a moment, the proof is readily finished; namely, observing that
$\calI \subset \calO_{\calE}^* (L_{n}^{\mathsf{left}}, L_{2n}^{\mathsf{left}}) \cap \calO_{\calE}^*  (L_{n}^{\mathsf{right}}, L_{2n}^{\mathsf{right}})$
we then have
\begin{align*}
\bbP & [ \calO_{h \omega \leq 0}^* (L^{\mathsf{left}}_{2n}, L^{\mathsf{left}}_{3n}) \cap \calO_{h \omega \leq 0}^* (L^{\mathsf{right}}_{2n}, L^{\mathsf{right}}_{3n}) \cap \calO_{\calE}^* (L_{n}^{\mathsf{left}}, L_{2n}^{\mathsf{left}}) \cap \calO_{\calE}^*  (L_{n}^{\mathsf{right}}, L_{2n}^{\mathsf{right}}) ] 
& \geq c_3 c_2 c_1^8 b_{14 n}
\end{align*}
and in particular
\begin{align}
\label{eq:almost there}
\bbP  [  \calO_{\calE}^* (L_{n}^{\mathsf{left}}, L_{2n}^{\mathsf{left}}) \cap \calO_{\calE}^*  (L_{n}^{\mathsf{right}}, L_{2n}^{\mathsf{right}})
\; | \; 
\calO_{h \omega \leq 0}^* (L^{\mathsf{left}}_{2n}, L^{\mathsf{left}}_{3n}) \cap \calO_{h \omega \leq 0}^* (L^{\mathsf{right}}_{2n}, L^{\mathsf{right}}_{3n})  ]
\geq c_3 c_2 c_1^8 b_{14 n}.
\end{align}
Now, to detect the conditioning event above, explore $h$ and $\omega$ outside of $ L^{\mathsf{left}}_{3n}$ and $ L^{\mathsf{right}}_{3n}$, and inside them explore from the boundary until finding the exterior-most $h \omega \leq 0$ dual-loops surrounding $ L^{\mathsf{left}}_{2n}$ and $ L^{\mathsf{right}}_{2n}$, respectively. This generates a boundary condition $\preceq \{ \pm 1 \}$ for the unexplored (random) simply-connected domains $\calD_L$ and $\calD_R$ containing $ L^{\mathsf{left}}_{2n}$ and $ L^{\mathsf{right}}_{2n}$, respectively. By SMP, CBC and inclusion,
\begin{align*}
\bbP & [   \calO_{\calE}^* (L_{n}^{\mathsf{left}}, L_{2n}^{\mathsf{left}}) \cap \calO_{\calE}^*  (L_{n}^{\mathsf{right}}, L_{2n}^{\mathsf{right}})
\; | \; \calD_L  = D_L, \calD_R = D_R ] \\
& \leq \bbP_{D_L}^{\pm 1} [\calO_{\calE }^* (L_{n}^{\mathsf{left}}, L_{2n}^{\mathsf{left}})] \bbP_{D_R}^{\pm 1} [\calO_{\calE }^* (L_{n}^{\mathsf{right}}, L_{2n}^{\mathsf{right}})] \\
& \leq \bbP_{D_L}^{\pm 1} [\calO_{\calE'}^* (L_{n}^{\mathsf{left}}, L_{2n}^{\mathsf{left}})] \bbP_{D_R}^{\pm 1} [\calO_{\calE' }^* (L_{n}^{\mathsf{right}}, L_{2n}^{\mathsf{right}})],
\end{align*}
where $\calE'$ is the set of edges on which $|h|\geq 5$ at both endpoints and additionally $B=0$ on edges between two absolute-heights $|h|=5$. By FKG for $(|h|, -B)$ and inclusion, and then sign flip symmetry
\begin{align*}
\bbP_{D_L}^{\pm 1} [\calO_{\calE'}^* (L_{n}^{\mathsf{left}}, L_{2n}^{\mathsf{left}})]
\leq 
\bbP_{L_{3n}^{\mathsf{left}}}^{\pm 1} [\calO_{\calE' }^* (L_{n}^{\mathsf{left}}, L_{3n}^{\mathsf{left}})]
\leq
2 \bbP_{L_{3n}^{\mathsf{left}}}^{\pm 1} [\calO_{\calE }^* (L_{n}^{\mathsf{left}}, L_{3n}^{\mathsf{left}})]= 2b_n.
\end{align*}
The same bound holds for $D_R$, and is independent of the shape of $D_L$ or $D_R$. Averaging over all $D_L$ and $D_R$, we thus have
\begin{align*}
\bbP  [  \calO_{ \calE }^* (L_{n}^{\mathsf{left}}, L_{2n}^{\mathsf{left}}) \cap \calO_{ \calE }^*  (L_{n}^{\mathsf{right}}, L_{2n}^{\mathsf{right}})
\; | \; 
\calO_{h \omega \leq 0}^* (L^{\mathsf{left}}_{2n}, L^{\mathsf{left}}_{3n}) \cap \calO_{h \omega \leq 0}^* (L^{\mathsf{right}}_{2n}, L^{\mathsf{right}}_{3n})  ] \leq 4 b_n^2,
\end{align*}
which combined with~\eqref{eq:almost there} yields
\begin{align*}
b_{14 n} \leq \tfrac{4}{c_1^8 c_2 c_3 } b_n^2.
\end{align*}
This concludes the proof.
\end{proof}

\begin{proof}[Equation~\eqref{eq:split domains}]
In order to detect whether $\calI \cap \calS_{h \omega \leq 0}^* $ occurs,  explore (i) the components of $\tilde{\omega} = 1$ of $L^{\mathsf{left}}_{n}$ and $L^{\mathsf{right}}_{n}$, and their bounding dual-loops of $\tilde{\omega}=0$; and (ii) the component of $h \omega \geq 1$ of $\partial \calL_{3n}$ in $\calL_{3n}$ and its bounding dual-loops of $h \omega \leq 0$. The former finds the interior-most dual-loops of $\tilde{\omega}=0$ surrounding $L^{\mathsf{left}}_{n}$ and $L^{\mathsf{right}}_{n}$, and the latter the exterior-most loop(s) of $h \omega \leq 0$ separating them from $\partial \calL_{3n}$.

Denote by $\calY$ the random unexplored area(s) between there explorations; it thus either consists of ``two topological annuli'' or a ``simply-connected domain with two holes''. By SMP, on $\calY  = Y$, the exploration generates a boundary condition $\chi$ taking values $\preceq \{ \pm 1 \}$ on the ``outer boundary components'' and the set-value $\{ 5, 7\}$ on the ``hole boundary components''. Let $R_{\mathsf{left}}$ be the ``right rectangle of the annulus $ L^{\mathsf{left}}_{3n} \setminus  L^{\mathsf{left}}_{2n}$'' and let  $R_{\mathsf{right}}$ be the ``left rectangle of the annulus $ L^{\mathsf{right}}_{3n} \setminus  L^{\mathsf{right}}_{2n}$''. We claim that
\begin{align}
\label{eq:pust pust}
\bbP_Y^\chi [ \calV^*_{h \omega \leq 0} (R_{\mathsf{left}} ) \cup \calV^*_{h \omega \leq 0} (R_{\mathsf{right}} )] \geq c_3
\end{align}
for all $Y$, where for $(h, B, \omega) \sim \bbP_Y^\chi$ we interpret that $h \omega \leq 0$ on $Y^c$. Averaging over $Y$ then gives~\eqref{eq:split domains}. 

\begin{figure}
\begin{center}
\includegraphics[width=0.8\textwidth]{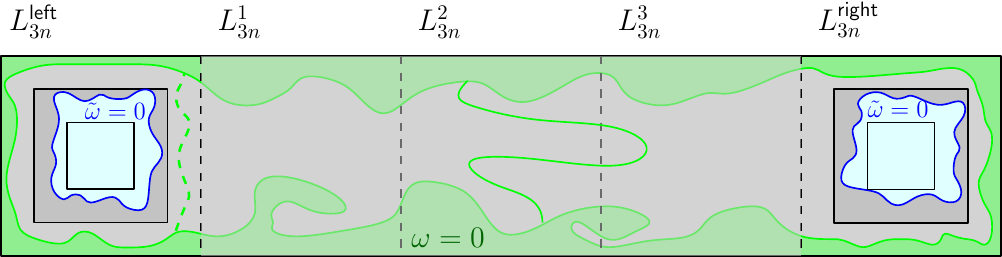}
\end{center}
\caption{
\label{fig:renorm domain split}
A geometric argument in the proof of Equation~\eqref{eq:split domains} (informal illustration): first, the random unexplored domain $\calY$ (in different shades of gray) can be compared to $\calL = L_{3n}^1 \cup L_{3n}^2 \cup L_{3n}^3$, in which the probability of a vertical $h \omega \leq 0$ dual-corssing (in solid green) is lower-bounded by Corollary~\ref{cor:narrow quad triple cross}. Then, by Proposition~\ref{prop:pushing}, such a dual-crossing may be pushed into $L_{3n}^{\mathsf{left}}$ (dashed green).
}
\end{figure}

It thus remains to prove~\eqref{eq:pust pust}. Denote $\calL = L_{3n}^1 \cup L_{3n}^2 \cup L_{3n}^3$. We start by claiming that (see Figure~\ref{fig:renorm domain split})
\begin{align}
\label{eq:pust pust pust}
\bbP_Y^\chi [ \calV^*_{h \omega \leq 0} (\calL ) ] \geq 1/8.
\end{align}
Indeed, letting $\chi'$ be exactly $\{ \pm 1 \}$ on the outer boundary of $Y$, CBC gives
\begin{align*}
\bbP_Y^\chi [ \calV^*_{h \omega \leq 0} (\calL ) ] 
\geq \bbP_Y^{\chi'} [ \calV^*_{h \omega \leq 0} (\calL ) ], 
\end{align*}
and with the boundary condition $\chi'$, $\calV^*_{h \omega \leq 0} (\calL ) = \calV^*_{\calE'(\tilde{h}, B)} (\calL )$, and thus by the FKG for $(|\tilde{h}|, -B)$,
\begin{align*}
\bbP_Y^{\chi'} [ \calV^*_{h \omega \leq 0} (\calL ) ]
=
\bbP_Y^{\chi'} [ \calV^*_{\calE'(\tilde{h}, B)} (\calL ) ]
 \geq 
 \bbP_{\calL \cap Y }^{\chi''} [ \calV^*_{\calE'(\tilde{h}, B)} (\calL ) ]
 =
\bbP_{\calL \cap Y }^{\chi''} [ \calV^*_{h \omega \leq 0} (\calL ) ],
\end{align*}
where $\chi''$ is the boundary condition for $\heightfcns_{\calL \cap Y}$ on $\partial (\calL \cap Y )$ which is, in terms of $h$, on the outer boundary of $Y$ equal to $\chi'$, i.e., set-valued $h \in \{ \pm 1 \}$, and on the rest of $\partial (\calL \cap Y )$ it is the largest admissible extension to $\{ 5, 7\}$. Corollary~\ref{cor:narrow quad triple cross} the gives~\eqref{eq:pust pust pust}.
%\begin{align*}
%\bbP_{\calL \cap Y }^{\chi''} [ \calV^*_{h \omega \leq 0} (\calL ) ] \geq 1/8.
%\end{align*}

The rest of proving~\eqref{eq:pust pust} is easy; by exploring the left-most crossing generating $\calV^*_{h \omega \leq 0} (\calL ) $ and then applying Proposition~\ref{prop:pushing}, one obtains
\begin{align*}
\bbP_Y^\chi [ \calV^*_{h \omega \leq 0} (R_{\mathsf{left}} ) \; | \; \calV^*_{h \omega \leq 0} (\calL ) ] \geq c_4,
\end{align*}
and thus
\begin{align*}
\bbP_Y^\chi [ \calV^*_{h \omega \leq 0} (R_{\mathsf{left}} ) ] \geq c_4/8
\end{align*}
and similarly for $\calV^*_{h \omega \leq 0} (R_{\mathsf{right}} ) $. Finally, by the FKG for $(h, B)$
\begin{align*}
\bbP_Y^\chi [ \calV^*_{h \omega \leq 0} (R_{\mathsf{left}} ) \cup \calV^*_{h \omega \leq 0} (R_{\mathsf{right}} )] \geq
\bbP_Y^\chi [ \calV^*_{h \omega \leq 0} (R_{\mathsf{left}} ) ]
\bbP_Y^\chi [ \calV^*_{h \omega \leq 0} (R_{\mathsf{right}} )] 
\geq c_4^2/64.
\end{align*}
This concludes the proof of~\eqref{eq:pust pust}, and hence also~\eqref{eq:split domains}.
\end{proof}

\subsection{Part (i) of Theorem~\ref{thm:3 equiv dichotomies}}

We start with an interesting extension of Theorem~\ref{thm:loop dichotomy} which crucially uses the RSW theorem~\ref{thm:RSW}.

\begin{lemma}
\label{lem:zero-loops in loop dichotomy}
Suppose that alternative~(ii) of Theorem~\ref{thm:loop dichotomy} occurs. Then, there exist $C, c, \alpha > 0$ such that
\begin{align*}
\bbP_{L_{4n}}^{ \pm 1 } [\calO^*_{\omega = 0}(L_n, L_{2n})] \geq 1-Ce^{-cn^\alpha} \qquad \text{for all }n.
\end{align*}
\end{lemma}

\begin{proof}%[Lemma~\ref{lem:zero-loops in loop dichotomy}]
Start by observing that $ \calO_{\omega = 0}^* (L_n, L_{2n}) ^c$ equivalently means that there is a primal-crossing of $\omega = 1$ from $ \partial L_{2n}$ to $\partial L_n$. Let $R$ denote a translate of $R_{5n/2, n/2}$ that covers the ``bottom rectangle'' of $ L_{2n} \setminus L_n$. If such a primal-crossing lands on the bottom of $ \partial L_{2n}$, it necessarily induces $\calV_{\omega = 1} (R)$. Combining these observation with the symmetric roles of the sides of $L_{2n}$, one obtains
\begin{align*}
1 - \bbP_{L_{4n}}^{ \pm 1  } [\calO^*_{\omega = 0}(L_n, L_{2n})] 
\leq 4 \bbP_{L_{4n}}^{ \pm 1 } [\calV_{\omega = 1} (R)].
\end{align*}
By Theorem~\ref{thm:RSW} and simple FKG arguments, there exist $c', C' > 0$ such that
\begin{align*}
c' \bbP_{L_{4n}}^{ \pm 1 } [\calV_{\omega = 1} (R)]^{C'}
\leq
b_n.
\end{align*}
The claim now follows from alternative~(ii) of Theorem~\ref{thm:loop dichotomy} (as $b_n \leq a_n$ by~\eqref{eq:two-side bound}).
\end{proof}

\begin{proof}[of Part~(i) in Theorem~\ref{thm:3 equiv dichotomies}]
Let $L^j $, $j = 1, 2,\ldots$, be $x$-centered lozenges or squares, with side length $k 2^j$ horizontal periods. Fix $J$ and study $(h, B, \omega) \sim \bbP^{ \pm 1 }_{L^J}=:\bbP_J$; note that there exists $C > 0$, depending only on $\bbG$ and its bi-periodic structure, such that if $|h(x)| \geq Ck$, then $\omega = 1$ on the entire $L^2$. In particular, if $|h(x)| \geq Ck$, there exist a largest value of $j \in \{ 3, 4, \ldots, J -1 \}$ for which $\calO^*_{\omega = 0}(L^{j-2},L^{j-1})$ does not occur (and $\calO^*_{\omega = 0}(L^{j-1},L^{j})$ occurs, where we count $\partial L^J$ as a dual-loop of $\omega = 0$ if $j=J$); let $\calE_\ell$ be the event that this largest value of $j$ is $\ell$. We thus have
\begin{align*}
\bbP_J [|h(x)| \geq Ck]
& \leq
\bbP_J [\bigcup_{\ell=3}^{J-1} \calE_\ell ] \\
& \leq
\sum_{ \ell =3}^{J-1}
\bbP_J [\calO^*_{\omega = 0} (L^{\ell-2},L^{\ell-1})^c \cap \calO^*_{\omega = 0} (L^{\ell-1},L^{\ell})].
\end{align*}
An exploration argument to detect $\calO^*_{\omega = 0} (L^{\ell-1},L^{\ell})$ and FKG yield
\begin{align*}
\bbP_J [\calO^*_{\omega = 0} (L^{\ell-2},L^{\ell-1})^c \; | \; \calO^*_{\omega = 0} (L^{\ell-1},L^{\ell})] \leq 
\bbP^{ \pm 1 }_{L^\ell} [\calO^*_{\omega = 0} (L^{\ell-2},L^{\ell-1})^c]
\leq
C' e^{-c(k2^\ell)^\alpha},
\end{align*}
where the latter inequality follows due to Lemma~\ref{lem:zero-loops in loop dichotomy}. Combining the two previous displayed equations,
\begin{align*}
\bbP_J [|h(x)| \geq Ck]
\leq
\sum_{ \ell =3}^\infty C' e^{-c(k2^\ell)^\alpha}.
\end{align*}
Note that the right-hand side is independent of $J$. It readily implies (e.g., by computing the obvious integral upper-bound of the series above) that $\bbE_J [h(x)^2]$ is uniformly bounded in $J$, i.e.,~\eqref{eq:Var-loc} occurs. This concludes the proof.
\end{proof}

\subsection{Part~(ii) of Theorem~\ref{thm:3 equiv dichotomies}: lower bound}
\label{subsec:pf 3 equiv dichotomies}

\begin{proof}[the lower bound in part~(ii) of Theorem~\ref{thm:3 equiv dichotomies}]
By~\eqref{eq:FKG-|h|}, it suffices to consider the case $D=L_n$, $x=0$. Denote now
\begin{align*}
v_n := \Var^{0}_{L_n} h(0) = \bbE^{ \{ \pm 1 \} }_{L_n} [h(0)^2].
\end{align*}
Note that by the FKG for $|h|$, $v_n$ is thus increasing in $n$. It thus suffices to prove the lower bound of (ii) for powers of $R$ (i.e., the parameter $R=3$ from Theorem~\ref{thm:loop dichotomy}), $n_k = R^k$, one has
\begin{align*}
v_{n_{k+1}} \geq v_{n_k} + c,
\end{align*}
for an absolute constant $c > 0$. Denote below $\bbP^{ \pm 1 }_{L_{n_k}} = : \bbP_k$ for short.

Let thus $(h, B, \omega) \sim \bbP_{k+1}$ and define a percolation process $\nu$ on the edges by
\begin{align*}
\nu_{\langle u, v \rangle} 
 =
\begin{cases}
0, \qquad \{ |h(u)|, |h(v)| \} = \{ 5, 7 \} \\
B_{\langle u, v \rangle}, \qquad |h(u)|= |h(v)| \in \{ 5, 7 \} \\
1, \qquad \text{otherwise}.
\end{cases}
\end{align*}
%in other words, $\hat{\omega}_e= 0 \Rightarrow \nu_e = 0$, but $nu$ also contains more zeroes that remove any information about the sign of $h$. 
Trivially,
\begin{align*}
v_{n_{k+1}} = \bbP_{k+1}[\calO^*_{\nu=0} (L_{n_k})] \bbE_{k+1} [ h(0)^2 \; | \; \calO^*_{\nu = 0} (L_{n_k}) ] + \bbP_{k+1}[\calO^*_{\nu=0} (L_{n_k})^c] \bbE_{k+1} [ h(0)^2 \; | \; \calO^*_{\nu = 0} (L_{n_k})^c ].
\end{align*}
where, by the assumed~\eqref{eq:loop-deloc}
\begin{align*}
\bbP_{k+1} [ \calO^*_{\nu = 0} (L_{n_k})] \geq
\bbP_{k+1} [ \calO^*_{\calE} (L_{n_k})] 
\geq \tfrac{1}{2 C}.
\end{align*}

Now, detect the event $\calO^*_{\nu = 0} (L_{n_k})$ by exploring the cluster of $\nu > 0$ of the boundary $\partial L_{n_{k+1}}$ in $L_{n_{k+1}} \setminus L_{n_k}$, and its bounding $\nu = 0$ dual-loops. If $\calO^*_{\nu = 0} (L_{n_k})$ does not occur, continue by revealing $\nu$ and $|h|$ on the entire $L_{n_{k+1}} \setminus L_{n_k}$ and $\partial L_{n_k}$. Note that now that when $|h|$ is known, $\nu$ equivalently provides information only about $B_{ \langle u,v \rangle} $ on the edges where $|h(u)|= |h(v)| \in \{ 5, 7 \} $. In particular, since $B$ and $h$ are independent, this exploration provides an absolute-value boundary condition. Using then~\eqref{eq:bdary cut} and~\eqref{eq:CBC-|h|}, one obtains
\begin{align*}
\bbE_{k+1} [ h(0)^2 \; | \; \calO^*_{\nu = 0} (L_{n_k})^c ] \geq \bbE_{k} [ h(0)^2 ] = v_k.
\end{align*}

In the complementary case where $\calO^*_{\nu = 0} (L_{n_k})$ does occur, the exploration above generates by SMP a boundary condition $h \in \{ 5, 7\}$ or $h \in \{ -5, -7\}$ for the unexplored simply-connected domain $\calD \supset L_{n_k}$. Then, compute
\begin{align*}
\bbE_\calD^{ \pm \{ 5, 7\} } [h(0)^2] 
= \bbE_\calD^{ \pm 1  } [(h(0) \pm 6)^2] 
= \bbE_\calD^{ \pm 1 } [h(0)^2] + 36
\geq v_k + 36,
\end{align*}
where in the second step we used the sign flip symmetry to deduce $\bbE_\calD^{ \pm 1 } [h(0)] = 0$, and in the third the FKG for $|h|$. Combining the four previously displayed inequalities, we obtain
\begin{align*}
v_{n_{k+1}}  \geq v_k + \tfrac{18}{ C}.
\end{align*}
This proves the claim.
\end{proof}

\subsection{Part~(ii) of Theorem~\ref{thm:3 equiv dichotomies}: upper bound}
\label{subsec:pf 3 equiv dichotomies 2}

Throughout this subsection, by mapping the bi-periodic embedding of $\bbG$ linearly if needed, we will assume that $\bbG$ is $\bbZ^2$ translation invariant. (Only the translational symmetries of $\bbG$ will be needed, and this will hence simplify the presentation.) Denote by $\bbT_N$ denote the ``locally $\bbG$'' graph on the torus, obtained by identifying the top and bottom (resp. left and right) sides of $\bbG \cap [-N, N]^2$. Denote by $\bbE_{\bbT_N}$ the random Lipschitz model on $\bbT_N$, with the boundary condition $h(r) \in \{ \pm 1\}$ imposed only at an auxiliary root node $r$ (which often, such as in the theorem below, need not be specified). The main result of this subsection is the following.

\begin{proposition}
\label{prop:var upper bd on torus}
Suppose that~\eqref{eq:loop-deloc} occurs in Theorem~\ref{thm:loop dichotomy}. Then, there exists $C > 0$ such that for all $N$ and all $x, y \in V( \bbT_N)$ with $d_{\bbT_N} (x, y) \geq 2$, we have
\begin{align*}
\bbE_{\bbT_N} [(h(x) - h(y))^2] \leq \log C d_{\bbT_N} (x, y),
\end{align*}
where $d_{\bbT_N}$ denotes the graph distance.
\end{proposition}

Let us see how this concludes the proof of Theorem~\ref{thm:3 equiv dichotomies}.

\begin{proof}[the upper bound in part~(ii) of Theorem~\ref{thm:3 equiv dichotomies}]
Embed $D$ on a torus $\bbT_N$ (with $N$ large enough to fit $D$ on $\bbT_N$); let $y$ be the closest boundary point to $x$ on $\partial D$, and set $r=y$ for the torus measure. By FKG for $(|h|, \omega)$ and then SMP 
\begin{align}
\label{eq:torus var lower bound}
\bbE_{\bbT_N} [h(x)^2] \geq \bbE_{\bbT_N} [h(x)^2 \; | \; \omega_{D^c} = 0]
= \bbE_{D}^{\pm 1} [h(x)^2].
\end{align}
It is then readily observed that for $r=y$, we have\footnote{
Indeed, by sign flip symmetry (all measures are on $\heightfcns_{\bbT_N}$)
\begin{align*}
\bbE_{\bbT_N} [h(x)^2] = \bbE^{h(y) \in \{ \pm 1\}} [h(x)^2]  = \bbE^{h(y) = 1 } [h(x)^2] = \bbE^{h(y) = -1 } [h(x)^2].
\end{align*}
On the other hand, by another symmetry argument, $ \bbE^{h(y) = \pm 1 } [h(x)] = \pm 1$ so that
\begin{align*}
\bbE^{ h(y) = \pm 1 } [(h(x)-h(y))^2] = \bbE^{ h(y) = \pm 1 } [h(x)^2] - 1.
\end{align*}
This is independent of the sign $h(y) = \pm 1$; hence by conditioning $\bbE_{\bbT_N}$ on the value $\pm 1$ of $h(y)$, finally
\begin{align*}
\bbE_{\bbT_N} [(h(x)-h(y))^2]  = \bbE^{ h(y) = \pm 1 } [(h(x)-h(y))^2] = \bbE_{\bbT_N} [h(x)^2] - 1.
\end{align*}
%Combining the three previously displayed equations yields~\eqref{eq:boring}.
}
\begin{align}
\label{eq:boring}
 \bbE_{\bbT_N} [h(x)^2] =  \bbE_{\bbT_N} [(h(x)-h(y))^2] + 1,
\end{align}
which, by Proposition~\ref{prop:var upper bd on torus}, concludes the proof.
\end{proof}

The rest of this subsection constitutes the proof of Proposition~\ref{prop:var upper bd on torus}. For simplicity, shift $x$ to the origin, $x=0$, and define
\begin{align*}
w_n = \sup_{D: D^c \cap [-n, n]^2 \neq \emptyset} \bbE^{ \pm 1 }_D [h(0)^2],
\end{align*}
where the supremum is over finite subgraphs $D=(V, E) \ni x$ of $\bbT_N$ such that $y \in D^c \cup \partial D$  and that are (i) bounded inside a contractible loop on the toric dual $\bbT_N^*$ that surrounds $x$ but not $y$; or (ii) bounded outside a contractible loop on $\bbT_N^*$ that surrounds $y$ but not $x$; or (iii) bounded between two disjoint non-contractible loops on $\bbT_N^*$ that together separate $x$ from $y$. In all cases, the boundary condition $\{ \pm 1 \}$ is imposed on the primal-vertices adjacent to an edge crossing these bounding dual loops. Note that $w_n$ also depend on $y$ and $N$; we however keep this implicit in the notation; also, $w_n$ is manifestly increasing in $n$, and it suffices to study $w_n$ for $n \leq d_\infty (0, y) = \sup \{ r: \bar{y} \not \in [-r, r]^2\}$, where $\bar{y}$ is the lift of $y \in \bbT_N$ to $[-N, N]^2$. Also, denote
\begin{align*}
u_n = w_{d_\infty (0, y) - n}.
\end{align*}
The proof of Proposition~\ref{prop:var upper bd on torus} now boils down to the following lemma in two parts.

\begin{lemma}
\label{lem:var bound explo}
Suppose that~\eqref{eq:loop-deloc} occurs in Theorem~\ref{thm:loop dichotomy}. Then, there exists $C > 0$ \emph{independent of $N$, $y$ and $x$} such that
\begin{align}
\label{eq:1st upper bound var}
& \text{i)} \qquad w_{2n} \leq w_n + C, \qquad \text{for all $n \leq d_\infty (0, y)/3$; and}
\\
&
\label{eq:2nd upper bound var}
\text{ii)} \qquad u_{n} \leq u_{2n} + C, \qquad \text{for all $n \leq d_\infty (0, y)/3$.}
\end{align}
\end{lemma}

\begin{remark}
Readers only interested in simply-connected discrete domains may observe that any simply-connected domain can be embedded on a large enough torus. Then,~\eqref{eq:1st upper bound var} alone will be sufficient to conclude the upper bound of part~(ii) of Theorem~\ref{thm:3 equiv dichotomies} for simply-connected domains. Consequently, such readers may restrict to ``cases (i)'' in several two-case formulations in the rest of this subsection.
\end{remark}

Let us now see how Lemma~\ref{lem:var bound explo} proves Proposition~\ref{prop:var upper bd on torus}.

\begin{proof}[Proposition~\ref{prop:var upper bd on torus}]
Note first that $w_1$ is bounded by a constant that only depends on $\bbG$, due to the one-Lipschitzness of $h$. Equation~\eqref{eq:1st upper bound var} (and the fact that $w_{n}$ is increasing in $n$) then shows that
\begin{align*}
w_n \leq C' \log n, \qquad \text{for all } n \leq 2 d_\infty (0, y)/3
\end{align*}
and similarly,
\begin{align*}
u_1 \leq u_n + C' \log n, \qquad \text{for all } n \leq 2 d_\infty (0, y)/3;
\end{align*}
here $C'$ is independent of $N$, $x$ and $y$. Then note that by definition
$$
w_{d_\infty (0, y)}-n
=
u_n ,$$
which we use for any fixed $d_\infty (0, y)/3 \leq n \leq 2 d_\infty (0, y)/3$. Also by definition and Lipschitzness, there exists $C''$ depending only on the bi-periodic lattice such that
\begin{align*}
\bbE_{ \bbT_N } [(h(x)-h(y))^2] = \bbE^{ \pm 1 }_{\partial D=\{y \}} [h(0)^2] \leq u_1 + C''.
\end{align*}
The claim follows by combining the displayed inequalities.
\end{proof}

It thus remains to prove Lemma~\ref{lem:var bound explo}. We start with the following.

\begin{lemma}
\label{lem:wn un log bounds}
Suppose that~\eqref{eq:loop-deloc} occurs in Theorem~\ref{thm:loop dichotomy}. Then, there exists $c > 0$ such that the following holds for all $N, n, x, y$. 
Let $D_m \subset \bbT_N$ (resp. $D_m'$) denote subgraphs as in the supremum defining $w_m$ (resp. $u_m$). Then,
\begin{align}
\label{eq:1st upper bound loop prob}
& \text{i)} \quad \bbP^{ \pm 1 }_{D_{2n}} [\partial D_{2n} \stackrel{h \leq_B 5 }{\longleftrightarrow} [-n,n]^2] \geq c , \quad \text{for all $n \leq d(x, y)/3$; and}
\\
&
\label{eq:2nd upper bound loop prob}
\text{ii)} \quad \bbP^{ \pm 1 }_{D'_{n}} [\partial D'_{n} \stackrel{ h \leq_B 5 }{\longleftrightarrow} [-d_\infty (0, y) - 2n,d_\infty (0, y) -2n]^2] \geq c, \quad \text{for all $n \leq d(x, y)/3$.}
\end{align}
\end{lemma}

\begin{figure}
\begin{center}
\includegraphics[width=0.43\textwidth]{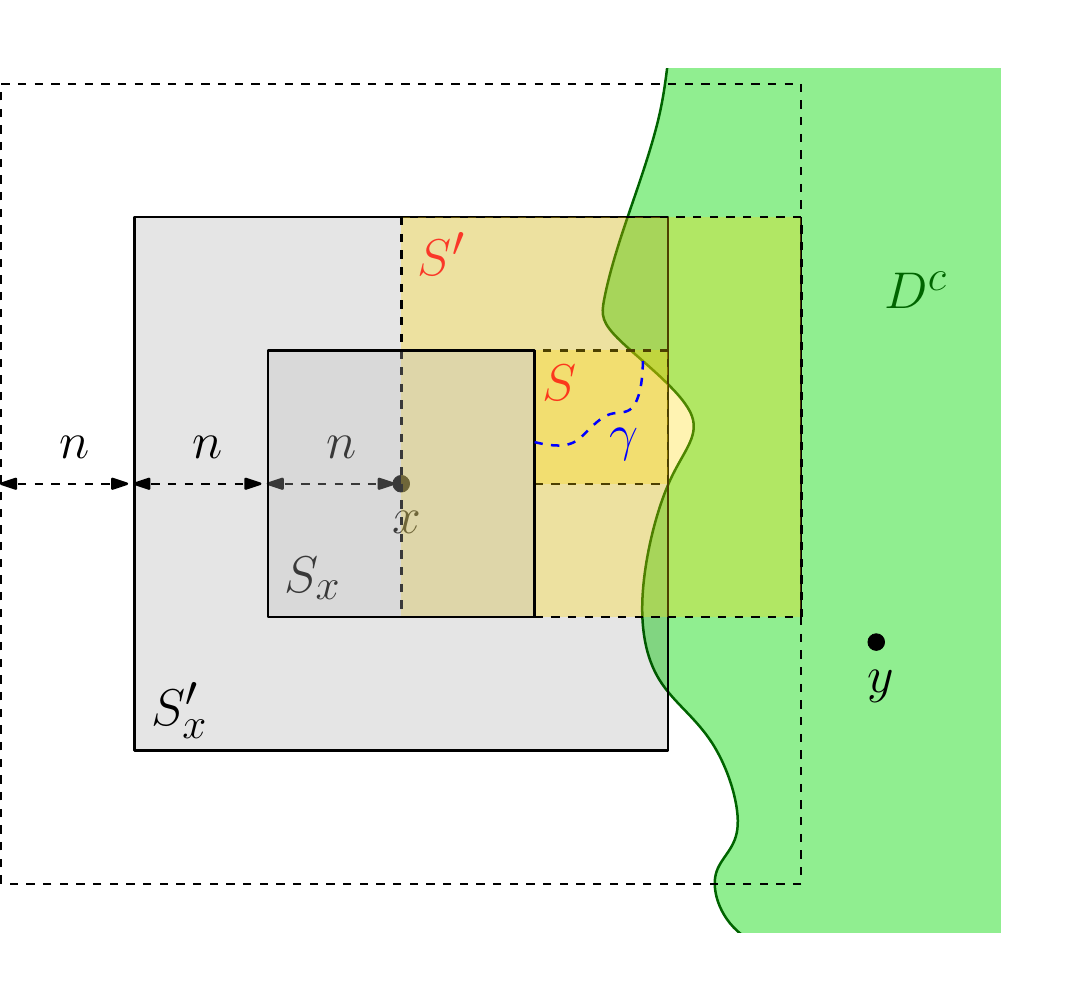} \qquad
\includegraphics[width=0.43\textwidth]{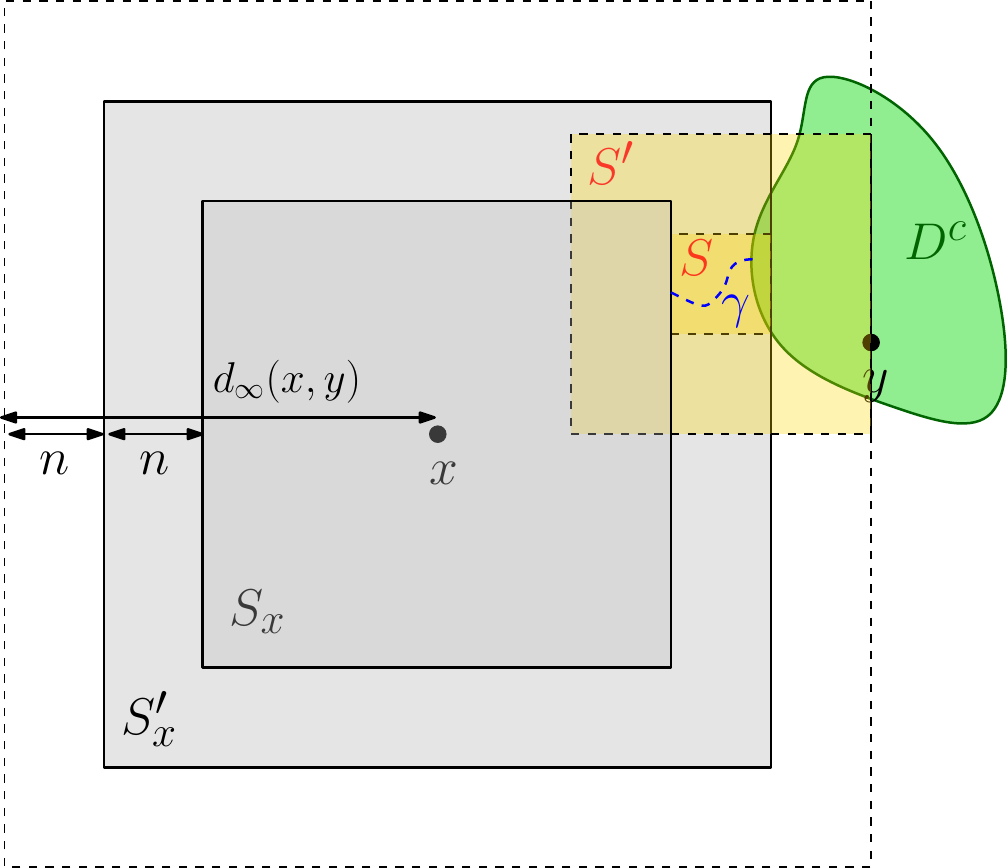}
\end{center}
\caption{
\label{fig:un and wn upper bounds}
The geometric setups in cases (i) (left) and (ii) (right) in the proof of Lemma~\ref{lem:wn un log bounds}.
}
\end{figure}

\begin{proof}
The proof starts with a geometric setup that is different in the two cases (see Figure~\ref{fig:un and wn upper bounds}).

\textbf{Case (i):} Denote $D := D_{2n}$, $S_x := [-n, n]$ and $S_x' := [-2n,  -2n]^2$. The claim is immediate if $\partial D$ intersects $S_x$, so assume not. Divide $S_x' \setminus S_x$ into $12$ translates of the square $[0,n]^2$; let $S$ be one such translate which additionally intersects $\partial D$. Let $S'$ be a translate of $[0,3n]^2$, cocentric with $S$. Note that (i) $x$ lies on the (planar-domain) boundary of $S'$; and (ii) $y$ lies outside of $S'$ (since $n \leq d_\infty (x, y)/3$); and (iii) there exists a primal-path $\gamma$ in $S \cap D$ from $\partial S_x$ to $\partial D$.

\textbf{Case (ii):} Denote  $D := D'_{n}$, $S_x := [-d_\infty (0, y) - 2n,d_\infty (0, y) -2n]^2$ and $S_x' := [-d_\infty (0, y) - n,d_\infty (0, y) -n]^2$. The claim is immediate if $\partial D$ intersects $S_x$, so assume not.  Let $S$ be a translate of $[0,n]^2$ that lies in $S_x' \setminus S_x$ and intersects $\partial D$; let $S'$ be a translate of $[0,3n]^2$, cocentric with $S$. Now, both $x$ and $y$ lie outside of $S'$ and there exists a primal-path $\gamma$ in $S \cap D$ from $\partial S_x$ to $\partial D$

The rest of the proof is identical in the two cases. Denote by $E$ the event in Equation~\eqref{eq:1st upper bound loop prob} or~\eqref{eq:2nd upper bound loop prob}, and suppose that $E$ does \textit{not} occur; hence, there exists a dual-curve of $\calE$ that disconnects $y$ and $\partial D$ from $x$ and $S_x$ (or possibly two dual-curves, if $\partial D$ does not surround $x$). This dual-curve must on the one hand cross the primal-path $\gamma$ to separate $\partial S_x$ and $\partial D$, and hence visit $S$, and on the other hand, it must exit $S'$ (since both $x$ and $y$ are outside of it). I.e.,
\begin{align*}
E^c \subset \calC_\calE^* (S, S'),
\end{align*}
where $\calC_\calE^* (S, S')$ denotes the existence of a dual-crossing on $\calE^*$ from $S$ to $S'$ in $S'$. By a shift of boundary conditions and then FKG for $(|h|, \omega)$, we hence have
\begin{align*}
1 - \bbP^{\pm 1}_D [E]
\leq
\bbP^{\pm 1}_D [\calC_\calE^* (S, S')] 
= 
\bbP^{ \{ 5, 7 \} }_D [\calC_{\omega = 0}^* (S, S')] 
\leq 
\bbP^{ \chi }_{ S' \cap D} [\calC_{\omega = 0}^* (S, S')],
\end{align*}
where $\chi$ is the boundary condition for $\heightfcns_{S' \cap D}$ given on $\partial (S' \cap D)$ which is $\{ \pm 1\}$ on $\partial S'$ and its largest extension to $ \{ 5, 7 \}$ on the rest.

Next, for $\heightfcns_{S' \cap D}$, interpret $h \geq 1$, $B=1$ and $\omega = 1$ on $S' \cap D^c$; hence by duality and the present boundary condition,
$\calC_{\omega = 0}^* (S, S')^c = \calO_{h \omega \geq 1} (S, S')$. Using this complement event, a height shift, and then the boundary condition and bound on dual-degree
\begin{align*}
\bbP^{ \chi }_{  S' \cap D} [\calC_{\omega = 0}^* (S, S')]
= 
%1-\bbP^{ \chi }_{ S' \cap D} [\calO_{ h \omega \geq 1} (S, S')]
%= 
1-\bbP^{ 6 - \chi }_{ S' \cap D } [\calO_{ h \leq_B 5} (S, S')]
=
1-\bbP^{ 6 - \chi }_{ S' \cap D } [\calO_{ |h| \leq_B 5} (S, S')].
\end{align*}
Next, by FKG for $(|h|, -B)$ and then height shift
\begin{align*}
\bbP^{ 6 - \chi }_{ S' \cap D } [\calO_{ |h| \leq_B 5} (S, S')]
 \geq \bbP^{ \{ 5, 7 \} }_{ S'} [\calO_{ |h| \leq_B 5} (S, S')]
= \bbP^{ \pm 1 }_{ S'} [\calO_{ h \omega \geq 1 } (S, S')] \geq c,
\end{align*}
where the boundary conditions $\{ 5, 7 \}$ and $\{ \pm 1 \}$ are imposed on $\partial S'$, and the last step used~\eqref{eq:loop-deloc}. This concludes the proof.
\end{proof}

\begin{corollary}
Suppose that~\eqref{eq:loop-deloc} occurs in Theorem~\ref{thm:loop dichotomy}. Then, there exists $\ell \in \bbN$ and $c' > 0$ such that the following holds for all $N, n, x, y$. 
Let $D_m \subset \bbT_N$ (resp. $D_m'$) denote subgraphs as in the supremum defining $w_m$ (resp. $u_m$). Then,
\begin{align}
\label{eq:1st upper bound loop prob abs val}
& \text{i)} \quad \bbP^{ \pm 1 }_{D_{2n}} [\partial D_{2n} \stackrel{|h| \leq_B 6 \ell -1 }{\longleftrightarrow} [-n,n]^2] \geq c' , \quad \text{for all $n \leq d(x, y)/3$; and}
\\
&
\label{eq:2nd upper bound loop prob abs val}
\text{ii)} \quad \bbP^{ \pm 1 }_{D'_{n}} [\partial D'_{n} \stackrel{ |h| \leq_B 6 \ell - 1 }{\longleftrightarrow} [-d_\infty (0, y) - 2n,d_\infty (0, y) -2n]^2] \geq c', \quad \text{for all $n \leq d(x, y)/3$.}
\end{align}
\end{corollary}

\begin{proof}
By SMP and the previous lemma, one readily deduces that
\begin{align*}
\bbP^{ \pm 1 }_{D_{2n}} [\{ \partial D_{2n} \stackrel{h \leq_B 6\ell -1 }{\longleftrightarrow} [-n,n]^2 \}^c] \leq (1-c)^\ell,
\end{align*}
where $c$ is the constant in the previous lemma. By choosing $\ell$  large enough, there exists $c' > 0$ such that $(1-c)^\ell \leq \tfrac{1-c'}{2}$. Then, by sign flip symmetry
\begin{align*}
\bbP^{ \pm 1 }_{D_{2n}} [\{ \partial D_{2n} \stackrel{|h| \leq_B 6\ell -1 }{\longleftrightarrow} [-n,n]^2 \}^c] \leq 2 \bbP^{ \pm 1 }_{D_{2n}} [\{ \partial D_{2n} \stackrel{h \leq_B 6\ell -1 }{\longleftrightarrow} [-n,n]^2 \}^c] \leq 1- c'.
\end{align*}
Equation~\eqref{eq:2nd upper bound loop prob abs val} is proven identically.
\end{proof}

\begin{proof}[Lemma~\ref{lem:var bound explo}]
We will prove~\eqref{eq:1st upper bound var};~\eqref{eq:2nd upper bound var} is identical. Let $\ell$ and $c'$ be the constants of the above corollary. Let $(h, B, \omega) \sim \bbP^{\pm 1}_{D_{2n}} := \bbP$ and define a percolation process $\nu$ on the edges by
\begin{align*}
\nu_{\langle u, v \rangle} 
 =
\begin{cases}
0, \qquad \{ |h(u)|, |h(v)| \} = \{ 6 \ell \pm 1 \} \\
\tilde{B}_{\langle u, v \rangle}, \qquad |h(u)|= |h(v)| \in \{ 6 \ell \pm 1 \} \\
1, \qquad \text{otherwise};
\end{cases}
\end{align*}
in other words, the event $\{ \partial D_{2n} \stackrel{|h| \leq_B 6 \ell -1 }{\longleftrightarrow} [-n,n]^2 \}^c$ means that there exist $\nu=0$ dual-paths that separate $\partial D_{2n}$ from $[-n,n]^2$; denote this event by $E_n$.
%in other words, $\hat{\omega}_e= 0 \Rightarrow \nu_e = 0$, but $nu$ also contains more zeroes that remove any information about the sign of $h$. 
Trivially,
\begin{align*}
\bbE [h(0)^2] = \bbP [ E_n ] \bbE [ h(0)^2 \; | \; E_n ] + \bbP [E_n^c] \bbE [ h(0)^2 \; | \; E_n^c ].
\end{align*}
where, by the assumed~\eqref{eq:loop-deloc} and the previous corollary
\begin{align*}
\bbP [ E_n ] \leq 1-c'.
\end{align*}
Quite similarly to the proof of the lower bound in Theorem~\ref{thm:main thm 2nd}, one obtains
\begin{align*}
\bbE [ h(0)^2 \; | \; E_n ] \leq w_{2n} + 36 \ell^2 
\qquad \text{and} \qquad
\bbE [ h(0)^2 \; | \; E_n^c ] \leq w_n  + 36 \ell^2 .
\end{align*}
Recalling that $w_{2n} \geq w_n$, the three previously displayed equalities give an upper bound for $\bbE [h(0)^2] = \bbE^{\pm 1}_{D_{2n}} [h(0)^2]$ uniformly in $D_{2n}$, and hence
\begin{align*}
w_{2n} \leq (1-c')(w_{2n} + 36 \ell^2) + c'(w_{n} + 36 \ell^2 )
\Leftrightarrow 
w_{2n} \leq   w_{n} + 36 \ell^2 / c'.
\end{align*}
Recalling the $\ell$ and $c'$ were independent of $x, y, n, N$, the proof is complete.
\end{proof}

\section{On the infinite-volume limit}

We now formalize and prove Corollary~\ref{thm:conseq of main thm intro statement} from the Introduction. This section mimics an unpublished corollary of the logarithmic delocalization of the six-vertex model~\cite{DCKMO}, produced by the author and Piet Lammers in 2020.

\subsection{Gradient Gibbs measures}
\label{subsec:grad Gibbs meas}

Let $D = (V, E)$ be any locally finite connected graph. A \textit{vector field} $g$ on $D$ assigns to every directed version $(u, v)$ of an edge $\langle u, v \rangle \in E$ a value $g_{(u, v)} = -g_{(v, u)}$. If additionally the values of $g$ sum to zero over any (oriented) loop on $D$, we say that $g$ is a \textit{gradient (vector) field}; indeed, such a vector field $g$ is obtained as the discrete gradient  $g_{(u, v)} = h(v)-h(u)$ of a function $h: V \to \bbR$ which is unique up to additive constant.

For the rest of this section, we study the space of $\calS$ vector fields on a planar cubic lattice $\bbG$, with the detailed assumptions given in Section~\ref{subsec:main thm setup}, that only take values in $\{ 0, \pm 2 \}$ on the directed edges. We equip the space $\calS$ with the metric of local convergence and note that it is (sequentially) compact. Gradient fields of height functions on finite $D \subset \bbG$ are naturally embedded in $\calS$ by continuing with the zero field outside of $D$.

A \textit{(random Lipschitz) gradient Gibbs measure} $\sfP$ is supported on the (infinite-volume) gradient fields of $\calS$ and has the property that given the values of $g \sim \sfP$ on the edges of a simple closed primal-loop $\ell$ on $\bbG$, which defines up to additive constant an integer-valued function $\xi$ on the vertices of $\ell$, the conditional law of $g$ inside $\ell$ is the gradient of a random Lipschitz function $h \sim \bbP^\xi$ (where $\xi$ is $\sfP$-a.s. admissible by the gradient-field support of $\sfP$).

By the results of~\cite{She05} (also~\cite{Piet-deloc, Piet-personal}), under the setup of Theorem~\ref{thm:main thm 2nd}, there exists a unique translationally invariant, slope-zero gradient Gibbs measure $\sfP$ on $\bbG$. \textit{Translational invariance} refers here to the translations that generate the bi-periodicity of $\bbG$, and \textit{slope-zero} means that $\sfE[h(x)-h(y)]=0$ for any $x, y \in \bbV$, where $h$ is a field whose gradient is $g \sim \sfP$. A direct consequence (which should not be regarded as a new result;~\cite{Piet-deloc}) is:

\begin{corollary}
\label{cor:inf vol lim on torii}
Let $\bbG$ and $\mathbf{c}_e$ be as in Theorem~\ref{thm:main thm 2nd}.
Then, the gradients of the toric random Lipschitz functions $h \sim \bbP_{\bbT_N}$ converge weakly as $N \to \infty$ to the unique translationally invariant slope-zero gradient Gibbs measure $\sfP$ on $\bbG$.
\end{corollary}

\begin{proof}
Since $\calS$ is compact, extract a subsequence $N_j$ so that the gradients converge weakly to some limit. Any such limit inherits directly from $\bbP_{\bbT_{N_j}}$ translational invariance, zero slope and the Gibbs property. The uniqueness result above thus identifies any subsequential limit, and hence also the weak limit along the entire sequence $N \to \infty$.
\end{proof}

In combination with the main results of this paper, one obtains the logarithmic variance structure of the limiting gradient field.

\begin{corollary}
\label{cor:log var for inf vol lim}
Let $\bbG$, $\mathbf{c}_e$ and $\sfP$ be as in the previous corollary. There exist $c, C > 0$ such that for any $x, y \in \bbV$,
\begin{align*}
c \log d_{\bbG} (x, y) \leq \sfE [(h(x)-h(y))^2] \leq C \log d_{\bbG} (x, y).
\end{align*}
\end{corollary}

\begin{proof}
By Theorem~\ref{thm:main thm 2nd}, Equation~\eqref{eq:torus var lower bound} and Proposition~\ref{prop:var upper bd on torus}, there exist  $c, C > 0$ such that for any $N$ and
$x, y \in V( \bbT_N)$, we have
\begin{align*}
c \log d_{\bbT_N} (x, y) \leq \bbE_{\bbT_N} [(h(x) - h(y))^2] \leq C \log d_{\bbT_N} (x, y).
\end{align*}
The weak limit inherits these bounds (note that $(h(x)-h(y))^2$ is for fixed $x, y$ a bounded random variable by Lipschitzness).
\end{proof}

The rest of this section establishes $\sfP$ as the infinite volume limit of random Lipschitz functions on planar domains.

\begin{proposition}
\label{prop:infinite vol limit}
Let $\bbG$, $\mathbf{c}_e$ and $\sfP$ be as in the previous corollaries, let $M \in \bbN$, and let $D_n \subset \bbG$ be finite connected subgraphs generated by their vertices and $\xi_n$ (not even necessarily interval-valued) admissible boundary conditions on $\partial D_n$ with $\xi_n(x) \subset \{ -M, \ldots, M\} $ for all $x \in \partial D_n$. Then, the gradients of $h \sim \bbP^{\xi_n}_{D_n}$ converge weakly to $\sfP$ as $D_n \to \bbG$.
\end{proposition}

\subsection{Proof of Proposition~\ref{prop:infinite vol limit}}

The proof consist in three steps which may all be of independent interest. First, we show that~\eqref{eq:loop-deloc} actually implies that loops at \textit{any height} can be guaranteed with \textit{any probability}; the price to pay is increasing the aspect ratio of the annulus for the loops. Secondly, we show a mixing of boundary conditions lemma, and thirdly a translational invariance estimate.

Recall that $\calE_m = \calE_m(h,B)$ denotes the edges $e= \langle u, v \rangle$ for which $h(u), h(v) \geq m$ and if $h(u) = h(v) = m$ then additionally $B_{e}=0$, and similarly, $\calE'_m = \calE'_m(h,B)$ denotes the edges $e= \langle u, v \rangle$ for which $|h(u) |, |h(v)| \geq m$ and if $|h(u)| = |h(v)| = m$ then additionally $B_{e}=0$

\begin{lemma}[Loops at any height with any probability]
For all $m \in \bbN$ and $\epsilon > 0$ there exists $\rho > 0$ such that
\begin{align*}
\bbP^{\pm 1}_D [\calO^*_{\calE_m} (L_n)] \geq 1 - \epsilon \qquad \text{whenever } L_{\rho n} \subset D.
\end{align*}
\end{lemma}

\begin{proof}
By FKG for $(|h|, -B)$ and~\eqref{eq:loop-deloc}, whenever $L_{3 n} \subset D$, one has
\begin{align*}
\bbP^{\pm 1}_D [\calO^*_{\calE'_5} (L_n)^c] \leq \bbP^{\pm 1}_{L_{3n}} [\calO^*_{\calE'_5} (L_n)^c]  \leq 1-c.
\end{align*}
Now, suppose that $L_{3^r n} \subset D$. Explore the primal-cluster of $(\calE'_5)^c$ from $\partial D$ inside $L_{3^{r-1} n}$; this exploration either reveals $\calO^*_{\calE'_5} (L_{3^{r-1} n})$ or if not, reveal $|h|$ on entire $L_{3^{r-1} n}^c$ and generate an absolute-value boundary condition which for events of $(|h|, -B)_{|L_{3^{r-1} n}}$ dominates $\bbP^{\pm 1}_{L_{3^{r-1} n}}$. Reasoning as above, one obtains
\begin{align*}
\bbP^{\pm 1}_D [\calO^*_{\calE'_5} (L_{3^{r-2} n})^c] \leq (1-c)^2,
\end{align*}
and similarly for any $k \leq r$
\begin{align*}
\bbP^{\pm 1}_D [\calO^*_{\calE'_5} (L_{3^{r-k} n})^c] \leq (1-c)^k.
\end{align*}

We now define iteratively the following (possibly empty) finite collection of dual-loops $\ell_1, \ell_2, \ldots \ell_L$,  surrounding $L_n$. The first one, $\ell_1$ is the exterior-most loop generating $\calO^*_{\calE'_5} (L_{n})$; by SMP, exploring outside $\ell_1$ generates a boundary condition $\{ 5, 7\}$ or $\{ -5, -7 \}$ on the simply-connected domain inside, both conditions appearing with equal probability. Now, if $\ell_j$ generated a boundary condition $\{ 6k \pm 1 \}$ for some $k \in \bbZ$, set $\tilde{h} = h-6k$ and explore inside $\ell_j$ until the exterior-most loop generating $\calO^*_{\calE'_5 (\tilde{h}, B)} (L_{n})$, if such exists.

Using the observations of the first paragraph, it is readily shown that, if $L_{3^r n} \subset D$, the number $L$ of loops is (or can be coupled to be) larger than the number of successes in $r$ independent coin tosses with success probability $c$. Also, conditional on the number $L$, the heights on the loops perform an $L$-step simple random walk, stepping up or down by $6$ units with equal probability. By properties of independent coin tosses and simple random walks, by taking $r$ large enough, one can ensure that the heights on the loops $\ell_1, \ldots \ell_L$ reach any $m$ with any probability $1-\epsilon$.
\end{proof}

For the mixing lemma, we need a slightly more general version of Proposition~\ref{prop:pos assoc of percolated height fcns} which also allows boundary conditions in terms of $\omega$. Hence, for a finite connected $D=(V, E)$, and given boundary condition $\xi$, $E' \subset E$ and a percolation $\beta: E' \to \{ 0, 1\}$, denote $\bbP_D^{\xi}[\; \cdot \; | \omega_{E'} = \beta] =: \bbP_D^{\xi, \beta} [\; \cdot \; ]$ (we always assume $\bbP_D^{\xi}[ \omega_{E'} = \beta] > 0$ below).

\begin{lemma}
\label{lem:general CBC}
	Let $D$ be a finite connected graph, $\xi \preceq_{abs} \xi'$ two absolute-value boundary conditions on $\Delta$, $\beta \leq \beta'$ two percolations on $E'$, and $F,G: \heightfcns_D \times \{ 0, 1 \}^E \to \bbR$ increasing functions; then we have
	\begin{align} 
		\bbE^{\xi, \beta}_D \big[F( \vert h \vert, \omega ) \big]
		\leq
		\bbE^{\xi', \beta'}_D \big[F( \vert h \vert, \omega) \big].
	\end{align}
\end{lemma}

\paragraph{Proof sketch}
Perhaps the most direct proof is to verify the Holley criterion for the pairs $\vert h \vert, \omega$. The irreducibility and initial states are almost identical to the case of $\vert h \vert$ alone (Appendix ~\ref{app:basics}). For the stochastic domination criterion, there are two parts: updating $|h|$ at a vertex and updating $\omega$ at an edge. The former is almost identical the case of $\vert h \vert$ alone (whose proof is based on comparing the weights $W(h)$, and changing them to $W(h,B)$, where $(h, B)$ gives the correct $(h, \omega)$, has essentially no effect since $h$ and $B$ are independent and $B$ is fixed in the comparisons). The updating of $\omega$ is directly handled by the positive association of the FK-Ising model (by Corollary~\ref{cor:FK cond law given abs height}). \hfill $\square$

\begin{corollary}[Mixing]
\label{cor:mixing}
For any even $m \in \bbN$ and any $\epsilon > 0$ there exists $\rho$ such that the following holds. For any $D \supset L_{\rho n}$ and any (not even necessarily interval-valued) boundary condition $\xi$ on $\partial D$ with $\xi(x) \subset [1, m-1]$ for all $x \in \partial D$, there is a coupling $\bbP$ of $h \sim \bbP_D^\xi$ and $h' \sim \bbP_D^{m \pm 1}$ such that
\begin{align*}
\bbP [ h = h' \text{ on } L_{n}] \geq 1- \epsilon.
\end{align*}
\end{corollary}

\begin{proof}
Pick $\rho$ by the previous lemma so that
\begin{align}
\label{eq:we have a loop}
\bbP_D^{m \pm 1} [\calO^*_{\omega = 0} (L_n)] \geq 1- \epsilon.
\end{align}
Suppose first that $\xi$ is single-valued. Next we would, informally speaking, like to use ``SMP for two coupled height functions''; we have however formulated no such result, but the following is a direct consequence of the definition of $\bbP^{\xi, \beta}_D$ as conditional laws:

\begin{lemma}
Let $\xi, \xi'$ be an absolute-value boundary conditions, $E_i \subset E$ and $\beta_i, \beta_i' : E_i \to \{ 0, 1\}$ such that $\bbP^{\xi, \beta_i}_D$ and $\bbP^{\xi', \beta_i'}_D$ exist. Let $E_{i+1} = E_i \cup \{ e_{i+1}\} $ where $e_{i+1}$ is given by a deterministic function of the setup $(E_i, \xi, \beta_i, \xi', \beta_i')$. Then, for $\beta_{i+1}^+$ (resp. $\beta_{i+1}^-$) the extension of $\beta_i$ to $e_{i+1}$ by $1$ (resp. by $0$) we have
\begin{align*}
\bbP^{\xi, \beta_i}_D [\omega (e_{i+1}) = 1] \bbP^{\xi, \beta_{i+1}^+ }_D [\; \cdot \;] 
+ \bbP^{\xi, \beta_i}_D [\omega (e_{i+1}) = 0] \bbP^{\xi, \beta_{i+1}^- }_D [\; \cdot \;] = \bbP^{\xi, \beta_i}_D [\; \cdot \;],
\end{align*}
and similarly for the ``primed case''.
\end{lemma}
We then construct the coupling $\bbP$ via the following exploration algorithm: start from the absolute-value boundary condition $\xi$ and $\xi' = \{ m \pm 1 \}$ on $\partial D$ and $E_0 = \emptyset$; at this point there are no $\beta_0, \beta_0'$. At any step, given $E_i$ and $ \beta_i \leq \beta_i'$ let $e_{i+1}$ be $\beta_i'$-connected to $\partial D$ (choose one by an arbitrary deterministic function on $\beta_i'$:s), sample $\omega (e_{i+1})$ from its law under $\bbP^{\xi, \beta_i}_D$ and $\omega' (e_{i+1})$ from its law under $\bbP^{\xi', \beta_i'}_D$, the two coupled so that $\omega (e_{i+1}) \leq \omega' (e_{i+1})$ (this is possible by Lemma~\ref{lem:general CBC}) and independently of the previous randomness in the algorithm. Then define $\beta_{i+1}$ and $\beta_{i+1}'$ by extending by these values to $e_{i+1}$. By construction, $ \beta_i \leq \beta_i'$ at all steps, and by the above lemma, the law of $(h, B, \omega)$ (resp. $(h', B', \omega')$) obtained from the \textit{random} measure $\bbP^{\xi, \beta_i}_D [ \; \cdot \;]$ (resp. $\bbP^{\xi', \beta_i'}_D [ \; \cdot \;]$) is $\bbP^\xi_D$ (resp. $\bbP^{\xi'}_D$).

The algorithm stops exploring at the time $T$ when it has exhausted the (random) $\beta'_{T}$-cluster of $\partial D$. Note that also $\beta_{T}=0$ on the dual-curves bounding this cluster. Then, by Lemma~\ref{lem:SMP with percolation} ,$(h, B, \omega)$ and $(h', B', \omega')$ under $\bbP^{\xi, \beta_T}_D [ \; \cdot \;]$ and $\bbP^{\xi', \beta_T'}_D [ \; \cdot \;]$ have the same law inside these dual-curves. We sample them to be the same inside. By~\eqref{eq:we have a loop}, this algorithmic procedure yields the desired coupling.

For set-valued $\xi$, the measure $\bbP_D^\xi$ is a convex combination of the compatible single-valued boundary conditions (with the appropriate convex coefficients). The claim follows then by coupling each single-valued condition to $\bbP_D^{m \pm 1}$ as above.
\end{proof}

\begin{corollary}[Translational invariance]
\label{cor:transl inv}
For any $\epsilon > 0$ there exists $\rho$ such that for any $D \supset L_{\rho n}$ and its translate $D'$ by one bi-periodicity unit, there is a coupling $\bbP$ of $h \sim \bbP_D^{\pm 1}$ and $h' \sim \bbP_{D'}^{\pm 1}$ such that
\begin{align*}
\bbP [ h = h' \text{ on } L_{n}] \geq 1- \epsilon.
\end{align*}
\end{corollary}

\begin{proof}
%Let $\tilde{h} \sim \bbP_D^{\pm 1}$ and $\tilde{h}' \sim \bbP_{D'}^{\pm 1}$ be independent under a coupling measure $\bbP$. Reveal the values of both on $(D \cap D')^c$ and $\partial (D \cap D')$; there exists an even $k \in \bbN$, depending only on $\bbG$ and its biperiodic structure, such that by Lipschitzness, the revealed values satisfy $|\tilde{h}|, |\tilde{h}'| \leq k-1$.
%
Observe first that there exists an even $k \in \bbN$, depending only on $\bbG$ and its biperiodic structure, such that the absolute-value of a $\bbP_D^{\pm 1}$ or $ \bbP_{D'}^{\pm 1}$ height function is at most $( k-1)$ on $\partial (D \cap D')$.
% Pick $\rho$ by the previous lemma so that
%\begin{align*}
%\bbP_{D \cap D'}^{2k \pm 1} [\calO^*_{\omega = 0} (L_n)] \geq 1- \epsilon/2.
%\end{align*}
%Let $h_* \sim \bbP_{D \cap D'}^{2k \pm 1}$ be defined under the coupling measure $\bbP$.
By the previous Corollary, we then choose $\rho$ so that for any single-valued boundary condition $\xi$ on $\partial (D \cap D')$ with $|\xi| \leq k-1$ there is a coupling $\bbP^{(\xi)}$ of $h_* \sim \bbP_{D \cap D'}^{2k \pm 1}$ and $h_{**} \sim \bbP_{D \cap D'}^{\xi + k}$ such that 
\begin{align*}
\bbP^{(\xi)} [ h_* = h_{**} \text{ on } L_{n}] \geq 1- \epsilon/2;
\end{align*}
this coupling induces a conditional law $\bbP^{(\xi)}[ \; \cdot \; | h_*]$ of $h_{**}$ given $h_*$.

A coupling $\bbP$ required in the present statement can now be constructed as follows. First, let $h_* \sim \bbP_{D \cap D'}^{2k \pm 1}$ be a variable in the coupling space. Then, for any $\xi$ as above, define $h_{**}^\xi$ on the coupling space by having the $\bbP$-conditional law $\bbP^{(\xi)}[ \; \cdot \; | h_*]$ given $h_*$, and so that $h_{**}^\xi$ for different $\xi$ are $\bbP$-conditionally independent given $h_*$. 

Let the coupling space $\bbP$ furthermore contain $\tilde{h} \sim \bbP_D^{\pm 1}$ and $\tilde{h}' \sim \bbP_{D'}^{\pm 1}$, independent of each other and the above random variables $h_*, (h_{**}^\xi)_\xi$. Finally, define $h$ (resp. $h'$) to coincide with $\tilde{h}$ (resp. $\tilde{h}'$) on $(D \cap D')^c$ and $\partial (D \cap D')$, and to coincide inside $D \cap D'$ with that $(h_{**}^\xi-k)$ where $\xi=\tilde{h}_{| \partial (D \cap D')}$ (resp. $\xi=\tilde{h}'_{| \partial (D \cap D')}$). By the SMP, the marginal laws are still $h \sim \bbP_D^{\pm 1}$ and $h' \sim \bbP_{D'}^{\pm 1}$, and by the Union bound,
\begin{align*}
\bbP [ h = h_{*}=h' \text{ on } L_{n}] \geq 1- \epsilon.
\end{align*}
This concludes the proof.
%
%Given the result of the above exploration, couple to it $h \sim \bbP_{D \cap D'}^{\xi}$ and $h' \sim _{D \cap D'}^{\xi'}$
%Let then $(\tilde{h}', \omega') \sim \bbP_D^{m \pm 1}$ and couple under $\bbP$ to it $(\tilde{h}, \omega) \sim \bbP_D^\xi$ so that $(|\tilde{h}'|, \omega') \succeq (|\tilde{h}|, \omega) $. 
%Explore $\omega'$ from $\partial D$ to the outer-most dual-loop generating $\calO^*_{\omega' = 0} (L_n)$ (which exists with probability at least $(1-\epsilon)$); note that also $\omega = 0$ on this loop. If such loop does not exist, we set $h'=\tilde{h}'$ and $h = \tilde{h}$. If it exists, reveal $\tilde{h}'$ and $\tilde{h}$ outside of it. By SMP, given the result of this exploration, the conditional law of $\tilde{h}'$ (resp. $\tilde{h}$) inside the dual-loop is given by $\pm 1$ boundary conditions. Thus, sample ${h}'$ and ${h}$ to be equal inside the dual-loop, and coinciding outside of it with $\tilde{h}'$ and $\tilde{h}$, respectively.
\end{proof}

\begin{proof}[Proposition~\ref{prop:infinite vol limit}]
By Corollary~\ref{cor:mixing}, it suffices to prove the claim when $\xi_n = \{ \pm 1 \}$. The proof then follows that of Corollary~\ref{cor:inf vol lim on torii}: extract a subsequence $n_j$ so that the gradients of $h \sim \bbP^{\pm 1}_{D_{n_j}}$ converge weakly to some limit. Any such limit inherits from $\bbP^{\pm 1}_{D_{n_j}}$ a zero slope and the Gibbs property and, by Corollary~\ref{cor:transl inv}, such a limit is also translationally invariant. This identifies the subsequential limit as $\sfP$, which is hence also the weak limit along the entire sequence $\bbP^{\pm 1}_{D_{n}}$.
\end{proof}

\appendix

\section{Phases for coupled models}
\label{app:Other models}

The purpose of this appendix is to briefly review two couplings of random Lipschitz functions to the random homomorphism model and phase characterizations that follow from those for the Lipschitz functions. There exist also other couplings from the Lipschitz functions, e.g. the Yadin bijection~\cite{Peled-highD-Lip} or most famously to the $n=2$ loop $O(n)$ model on the honeycomb lattice (or other $3$-regular lattices), but we omit these discussions, the former being non-planar and the latter well-studied.

\subsection{The couplings}

\paragraph{The random homomorphism model}

Let $G=(V, E)$ be a finite, connected, bipartite graph; we declare the vertices of $G$ as \textit{even} or \textit{odd} according to the bipartition. A \textit{(graph) homomorphism} $g: V \to \bbZ$ is a function with $|g(u)-g(v)| = 1$ for all adjacent vertices $u$ and $v$, and such that $g$ is odd on the odd vertices and even on the even vertices.

Given a non-empty $\Delta \subset V$, and $\xi: \Delta \to \bbZ$ such that $g_{| \Delta} = \xi$ for some homomorphism $g$, the \textit{random homomorphism model} $\bbQ^\xi_{G}$ on $G$ with boundary condition $\xi$ is the uniform measure on homomorphisms $g: V \to \bbZ$ with $g_{| \Delta} = \xi$. Set-valued boundary conditions can be defined analogously.

\paragraph{$\mathbf{c}=2$ and homomorphisms on dotted graphs}

The first coupling is for bipartite graphs obtained by ``dotting'': let $G=(V, E)$ be a finite connected graph and let $\dot{G}=(\dot{V}, \dot{E})$ be the bipartite \textit{dotted graph} of $G$ obtained, pictorially, by adding a vertex in the middle of each edge of $G$; we declare the vertices of $\dot{G}$ that were already present in $G$ as odd, and the vertices of $\dot{G}$ in the middle of the edges of $G$ as even. Restriction to odd vertices now clearly provides a surjection from homomorphisms on $\dot{G}$ to height functions on $G$, and a boundary condition $\xi$ on $\Delta \subset V$ is admissible for height functions on $G$ if and only if it is admissible for homomorphisms on $\dot{G}$. Furthermore, this surjection yields the following coupling, whose proof we leave for the reader.

\begin{proposition}
\label{prop:dotting coupling}
Let $G=(V, E)$ be a finite connected graph, $\dot{G}=(\dot{V}, \dot{E})$ its dotted bipartite graph, $\Delta \subset V$  and $\xi: \Delta \to \Zodd$ a (possibly set-valued) boundary condition admissible for random Lipschitz functions on $G$ (equivalently, for homomorphisms on $\dot{G}$). If $g \sim \bbQ^\xi_{\dot{G}}$, then its restriction $g_{|G}$ has the law $ \bbP^\xi_{G}$ of random Lipschitz functions on $G$ with edge weights $\mathbf{c}_e = 2$ for all $e \in E$. 
\end{proposition}

\paragraph{$\mathbf{c}=\sqrt{2}$ and homomorphisms on triangle--star transformed graphs}

We say that a finite connected graph $G=(V, E)$ is a \textit{triangulation} if there is a collection of mutually edge-disjoint triangles $t_1, \ldots, t_m$ that covers $E$. By a \textit{triangle}, we mean here a cyclic path three \textit{distinct} oriented edges $e_1, e_2, e_3$ (the endpoints of these edges need not be three distinct vertices). The triangle--star (or $\nabla$--$Y$) transform $G_Y = (V_Y, E_Y)$ of $G$ is the bipartite connected graph obtained by adding $m$ new vertices $\varpi_1, \ldots, \varpi_m$ and replacing the edges of each triangle $t_i = ( e^i_1 = ( u, v ),  e^i_2 = ( v, w ) ,  e^i_3 = ( w, u )  )$ by the three edges $\langle  u, \varpi_i \rangle$, $\langle  v, \varpi_i \rangle$, $\langle  w, \varpi_i \rangle$. The vertices of $G_Y$ that were already present in $G$ are again declared as odd, and $\varpi_1, \ldots, \varpi_m$ as even. Restriction again provides a surjection from homomorphisms on $G_Y$ to height functions on $G$, and a boundary condition $\xi$ on $\Delta \subset V$ is admissible for height functions if and only if it is admissible for homomorphisms.

\begin{proposition}
\label{prop:star-triangle coupling}
Let $G=(V, E)$ be a finite connected triangulation graph, $G_Y = (V_Y, E_Y)$ its $\nabla$--$Y$ transform, let $\Delta \subset V$  and $\xi: \Delta \to \Zodd$ a (possibly set-valued) boundary condition admissible for random Lipschitz functions on $G$ (equivalently, for homomorphisms on $\dot{G}$). If $g \sim \bbQ^\xi_{G_Y}$, then its restriction $g_{|G}$ has the law $ \bbP^\xi_{G}$ of random Lipschitz functions on $G$ with edge weights $\mathbf{c}_e = \sqrt{2}$ for all $e \in E$. 
\end{proposition}

\begin{proof}
We say that a given height function $h \in \heightfcns_G$ is \textit{flat} on a triangle $t_i = \{ ( u, v ), ( v, w ) , ( w, u )  \}$ if $h(u)=h(v)=h(w)$ (in the complementary case, one of the vertices has a different height that the two others). Now, we have
\begin{align*}
\# \{ \text{ homomorphisms $g$ on $G_Y$ with } g_{|G} = h \} =2^{ \# \{t_i: \; h \text{ flat on } t_i \} }.
\end{align*}
Sampling a random height function on $G$ as $ g_{|G}$, where $g \sim \bbQ^\xi_{G_Y}$, the probability of each height configuration $h$ with $h(x) \in \xi(x)$ for all $x \in \Delta$ is thus proportional to $2^{ \# \{t_i: \; h \text{ flat on } t_i \} }$.
On the other hand, sampling a random height function on $G$ from the law $ \bbP^\xi_{G}$ (with $\mathbf{c}_e \equiv \sqrt{2}$) the probability of a configuration $h$ with $h(x) \in \xi(x)$ for all $x \in \Delta$ is proportional to
\begin{align*}
W(h)
& = \prod_{\substack{ t_i = \{ ( u, v ), ( v, w ) , ( w, u ) \} \\ 1\leq i \leq m} }\sqrt{2}^{\bbI\{ h(u)=h(v) \} + \bbI\{ h(v)=h(w) \} + \bbI\{ h(w)=h(u) \}} \\
& = \sqrt{2}^m \cdot 2^{ \# \{t_i: \; h \text{ flat on } t_i \} }.
\end{align*}
It follows that the two random height functions are equal in distribution.
\end{proof}

\subsection{Phase characterizations}

We now state some phase characterizations following from the main results of this paper by the above couplings. For the sake of brevity, we will not explicate how a ``discrete domain'' should be defined for each different lattice (to be coherent with the above couplings), thus leaving the statements just slightly informal.
We also note that the Gibbs measure and variance dichotomies and their equivalence, similar to Section~\ref{sec:equivalent phase characterizations}, hold for the random homomorphism model~\cite{CPT18}; we thus trust that the reader understands the terms \textit{localized} or \textit{logarithmically delocalized}.

\begin{figure}
\begin{center}
\includegraphics[width=0.45\textwidth]{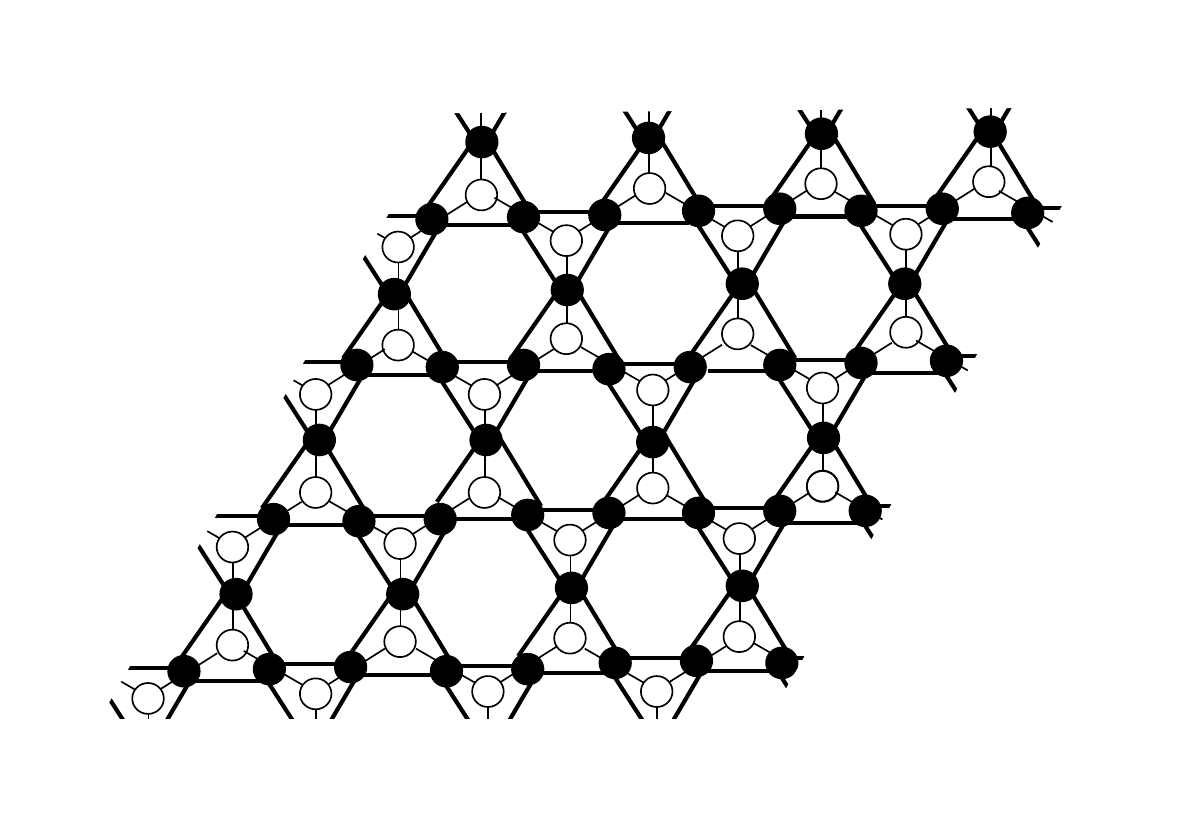}
\includegraphics[width=0.3\textwidth]{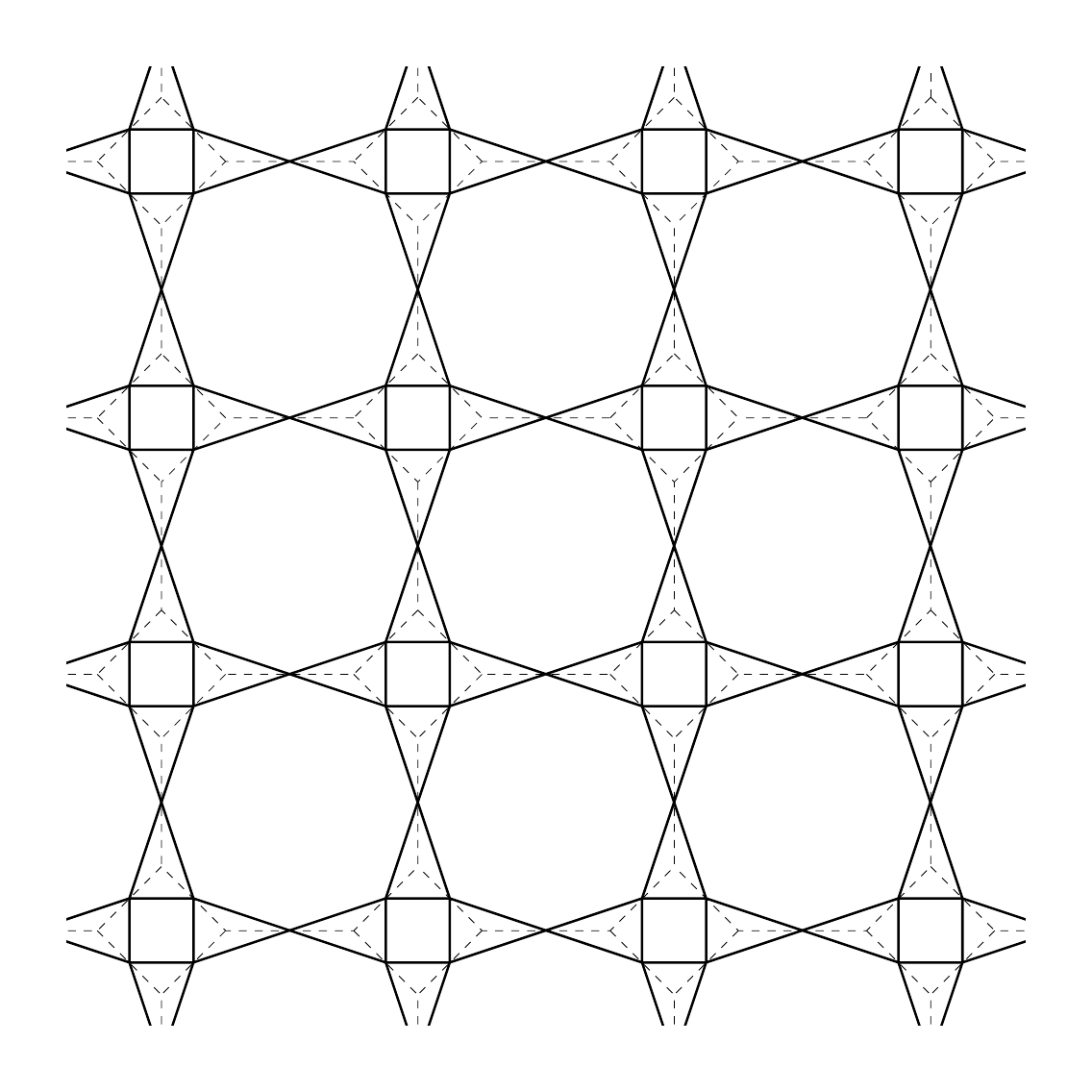} \\
\includegraphics[width=0.5\textwidth]{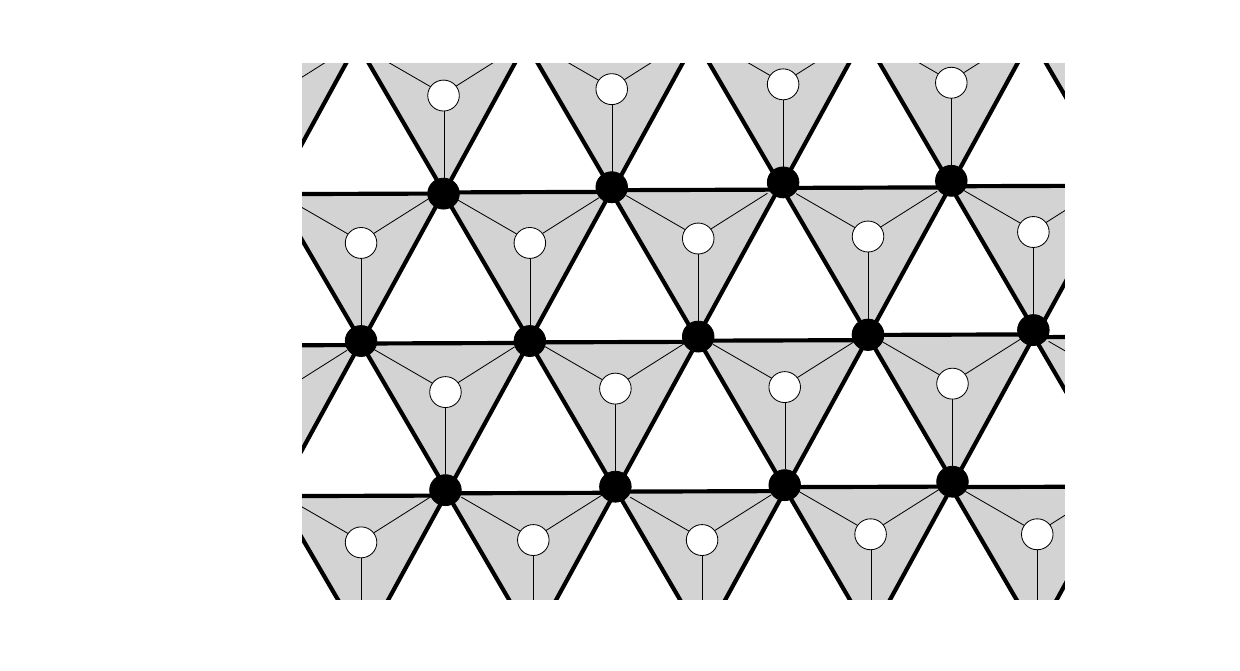}%
\includegraphics[width=0.43\textwidth]{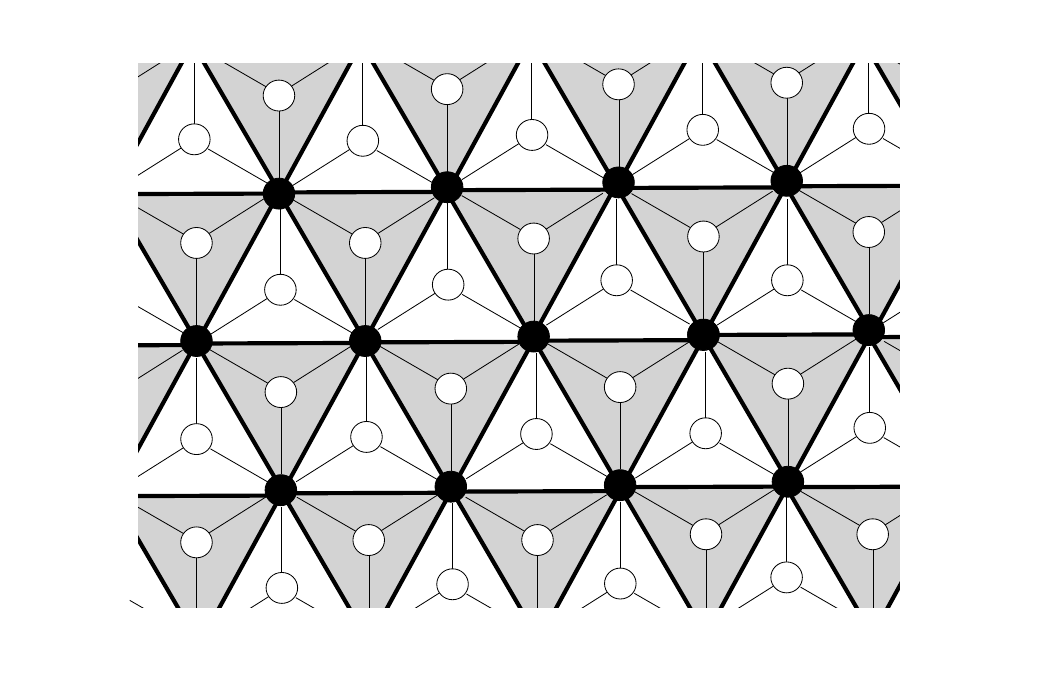}
\end{center}
\caption{
\label{fig:many diff lattices}
Top-left: The kagome lattice (black vertices, thick edges) is a triangulation and its $\nabla$--$Y$ transform is the dotted hexagonal lattice (both black and white vertices, thin edges). Top-right: an illustration of the same general idea of $Y$--$\nabla$ transforming of the dotted version of a $3$-regular lattice, in this case the square-octagon lattice. Bottom-left: The triangular lattice (black vertices, thick edges) is a triangulation (triangles filled with gray) and its $\nabla$--$Y$ transform is the hexagonal lattice (both black and white vertices, thin edges).  Bottom-left: The triangular lattice with doubled edges (black vertices, thick edges, each edge doubled) is a triangulation (both white and gray triangles) and its $\nabla$--$Y$ transform is the rhombille lattice (both black and white vertices, thin edges).
}
\end{figure}

1) \textit{The random homomorphism model on the dotted triangular lattice is localized.} This holds due to its coupling to the $\mathbf{c} \equiv 2$ random Lipschitz model on the triangular lattice (Proposition~\ref{prop:dotting coupling}), which is localized by Corollary~\ref{cor:localized regimes}.

2) \textit{The random homomorphism model on the dotted honeycomb (resp. square-octagon) lattice (see Figure~\ref{fig:many diff lattices} (top-left)) is logarithmically delocalized}, by its coupling to the $\mathbf{c} \equiv 2$ random Lipschitz model on the hexagonal lattice (Proposition~\ref{prop:dotting coupling}) and Theorem~\ref{thm:main thm 2nd}.

3) \textit{The random Lipschitz model with $\mathbf{c} \equiv \sqrt{2}$ on the kagome lattice (see Figure~\ref{fig:many diff lattices} (top-left)) is logarithmically delocalized}, by its coupling to the random homomorphism model on the dotted hexagonal lattice (Proposition~\ref{prop:star-triangle coupling}).

3') A non-quantitative delocalization can be similarly deduced for more general lattices. Let $\Gamma$ be any bi-periodic $3$-regular planar lattice. By~\cite[Theorem~2.7]{Piet-deloc} and~\cite{CPT18}, the random homomorphism model is delocalized (but not necessarily logarithmically) on its bipartite dotted lattice $\dot{\Gamma}$. The odd vertices of this dotted lattice are all of degree $3$, and hence it is the star-triangle transform is a certain triangulation lattice, on which the $\mathbf{c} \equiv \sqrt{2}$ random Lipschitz model is delocalized. An example of this procedure with $\Gamma$ being the square-octagon lattice is illustrated in Figure~\ref{fig:many diff lattices}(top-right). Theorem~\ref{thm:main thm 2nd} shows that this delocalization is logarithmic for suitable lattices $\Gamma$, such as the illustrated case of the square-octagon lattice.

4) \textit{The random homomorphism model on the honeycomb lattice is logarithmically delocalized}, by its coupling to the $\mathbf{c} \equiv \sqrt{2}$ random Lipschitz model on the triangular lattice (Proposition~\ref{prop:star-triangle coupling}; Figure~\ref{fig:many diff lattices} (bottom-left)) and the logarithmic delocalization of the latter~\cite{DGPS}. It should be emphasized that (although a the heavy lifting for this result is~\cite{DGPS}), the author is not aware of a prior explicit statement in the literature. The non-quantitative delocalization of this model follows by~\cite[Theorem~2.7]{Piet-deloc} and~\cite{CPT18}.

5) \textit{The random homomorphism model on the rhombille lattice (Figure~\ref{fig:many diff lattices} (bottom-right)) is localized}, by its coupling to the $\mathbf{c} \equiv \sqrt{2}$ random Lipschitz model on the triangular lattice with doubled edges (Proposition~\ref{prop:star-triangle coupling}). The latter is equivalent to $\mathbf{c} \equiv 2$ and non-doubled edges, which is localized by Corollary~\ref{cor:localized regimes}. This is arguably the simplest bipartite planar lattice on which the random homomorphism model is localized, highlighting the non-triviality of delocalization results such as the main result of the present paper.

\section{Positive association of the absolute-height}
\label{app:basics}

%We prove Lemmas~\ref{prop:|h|-monotonicity} and~\ref{prop:SMP ineq for abs val}.

\subsection{The Holley and FKG criteria}

Let $D=(V, E)$ be a finite connected graph, and let $\mu$ and $\mu'$ be two probability measures on $\heightfcns_D$.
We say that $\mu'$ \emph{stochastically dominates} $\mu$, denoted $\mu \leq_{st} \mu'$, 
if there exists a probability measure $\nu$ on pairs $(h,h')\in \heightfcns_D \times \heightfcns_D$ such that the first and second marginal distributions are $\mu$ and $\mu'$, respectively, and $\nu[h \preceq h']=1$. Note that if $\mu \leq_{st} \mu'$, then, for all increasing $F: \heightfcns_D \to \bbR$,
\begin{align*}
	\mu [F(h)]  \leq \mu' [F(h) ].
\end{align*}
We say that $\mu$ is \emph{irreducible} if for any two $h, h' \in \heightfcns_D$ with $\mu[h], \mu[h'] > 0 $, there exists a finite sequence $h=h_0, h_1, \ldots, h_m = h'$ on $\heightfcns_D$, such that for every $1 \leq i \leq m$, $\mu[h_i] > 0$ and $h_i$ differs from $h_{i-1}$ at one vertex only. 
We now recall the classical Holley and FKG criteria. For details see the extensive discussion of these criteria in~\cite{Gri06}.

\begin{lemma}[Holley's criterion]\label{lem:Holley}
    %Let $\mu$ and $\mu'$ be laws of $\heightfcns_D$-valued random variables $H$ and $H'$, respectively.
Let $D=(V, E)$ be a finite connected graph, and $\mu, \mu'$ probability measures on $\heightfcns_D$. The following criteria are sufficient to guarantee that $\mu \leq_{st} \mu' $: 
    \begin{itemize}[noitemsep]
        \item $\mu$ and $\mu'$ are irreducible; and
        \item there exist $h \preceq h' \in \heightfcns_D$ such that $\mu[h]>0$ and $\mu'[h'] > 0$; and
        \item for every vertex $x \in V$, every $k \in \Zodd$, 
        $\mu$-almost every $\chi \in \heightfcns_{D \setminus \{x\} }$, 
        and $\mu'$-almost every $\chi' \in \heightfcns_{D \setminus \{x\} }$ with $\chi \preceq \chi'$, we have
        \begin{align}\label{eq:Holley}
            \mu[ h(x) \geq k \; \vert \; h_{\vert G \setminus \{x\}} = \chi ]
            \leq \mu'[ h(x) \geq k \; \vert \; h_{\vert G \setminus \{x\}} = \chi' ].
        \end{align}
    \end{itemize}
\end{lemma}

\begin{lemma}[FKG criterion]\label{lem: FKG}
Let $D=(V, E)$ be a finite connected graph, and $\mu$ an irreducible probability measure on $\heightfcns_D$. If for every vertex $x \in V$, every $k \in \Zodd$, and $\mu$-almost every $\chi \in \heightfcns_{G \setminus \{x\} }$ and $\chi' \in \heightfcns_{G \setminus \{x\}}$ with $\chi \preceq \chi'$, we have \begin{align}
\label{eq:FKG criterion}
 \mu[ h(x) \geq k \; \vert \; h_{\vert G \setminus \{x\}} = \chi ]
 \leq 
  \mu[ h(x) \geq k \; \vert \; h_{\vert G \setminus \{x\}} = \chi' ],
\end{align}
then for all increasing functions $F_1, F_2: \heightfcns_G \to \bbR$,
\begin{align*}
\mu [F_1 (h) F_2 (h)] \geq \mu [F_1 (h) ] \mu [ F_2 (h)].
\end{align*}
\end{lemma}

\subsection{Proof of Propositions~\ref{prop:|h|-monotonicity} and~\ref{prop:SMP ineq for abs val}}

The inequality~\eqref{eq:CBC-|h|} follows by checking Holley's criterion for $\mu$ the law of $|h|$ when $ h \sim \bbP^{\xi}$ and $\mu'$ the law of $|h'|$ under $h' \sim \bbP^{\xi'}$. Then, by taking $\mu'=\mu$ in the Holley criterion, we observe that $\mu$ satisfies the FKG criterion, which yields~\eqref{eq:FKG-|h|}, so it is ``a special case of''~\eqref{eq:CBC-|h|}. Finally,~\eqref{eq:bdary cut} is proven via the Holley's criterion for $\mu'$ the law of $|h'|$ when $h' \sim \bbP^{\xi'}_{D'}$ and $\mu$ the law of $|h|$ when $h \sim \bbP^{\xi}_D$.\footnote{Formally, we extend $h$ by all-ones to $D' \setminus D$ to precisely match the setup of the criterion.} The two verifications of Holley's criterion are largely identical. Their proof consists of several routine computations and a nontrivial case of~\eqref{eq:Holley}, where, roughly speaking, the sign of $h$ is not fixed around $x$ by $\chi$.

\subsubsection{Routine computations}

\textbf{Irreducibility:} We check the irreducibility of the law of $|h|$ with any absolute-value boundary condition $h \sim \bbP^{\xi} $. Let $H, H'$ be absolute values of height functions $h, h' \in \heightfcns_D$ allowed by $ \xi $. It suffices to find a finite sequence of functions from $H$ to $H \vee H'$; a similar path can be reversed from $H \vee H'$ to $H'$. For a lighter notation we assume $H \preceq H'$. We'll find a path $H = H_0 \preceq H_1 \preceq \ldots$ so that $H_i \preceq H'$ and so that at each step, unless $H_i = H'$, $H_i$ will be increased by $2$ at one vertex $x_i$ to obtain $H_{i+1}$. In finitely many steps, such a path will achieve $H'$.

Now, if $H_i = H'$ we're done. Otherwise, we want to choose $x_i$. Note that $H_i (x_i)$ can be lifted to yield another height function if and only if $x_i$ has no neighbours with a strictly lower height in $H_i$.
Let $x_i \in V$ be the vertex where $H'-H_i$ attains its maximum value; if there are several such vertices, pick one amongst them with minimal height $H_i(x_i)$. It follows that $x_i$ cannot have neighbours with strictly lower height $H_i$, so its value can be increased, and $H_{i+1}$ thus obtained satisfies $H_{i+1} \preceq H'$.

\textbf{Initial states:} We only need to consider the case of~\eqref{eq:CBC-|h|} here. Note that a function $\zeta: \Delta \to \Zodd$ is such that $\bbP^\xi [|h_{|\Delta}| = \zeta] > 0$ if and only if it is an admissible single-valued boundary condition on $\Delta$ with $|a(x)| \leq \zeta(x) \leq b(x)$ for all $x \in \Delta$; we then say that $\zeta$ is \textit{$\mu$ possible}. Let $\zeta$ and $\zeta'$ be any $\mu$ and $\mu'$ possible conditions, respectively. Since $\xi \preceq_{abs} \xi'$, it follows that also $\zeta \wedge \zeta'$ (resp. $\zeta \vee \zeta'$) is $\mu$ possible (resp. $\mu'$ possible). Let then $\overline{h}$ (resp. $\overline{h}'$) be the maximal height function~\eqref{eq:max and min height fcns} with the boundary condition $\zeta \wedge \zeta'$ (resp. $\zeta \vee \zeta'$); hence $0 < \overline{h} \preceq \overline{h}'$. We also directly have $\mu[\overline{h}] = \bbP^{\xi}[ |h| = \overline{h}] \geq \bbP^{\xi}[ h = \overline{h}] > 0$, and similarly for $\mu' [\overline{h}']$.

\textbf{Property~\eqref{eq:Holley}:} We start with trivialities. First, for~\eqref{eq:CBC-|h|} (resp.~\eqref{eq:bdary cut}), there is nothing to check if $x \in S' \setminus S$ (resp. $x \in D' \setminus D$). Second, we are actually proving a stochastic domination of real random variables. If $x \in \Delta$, the boundary conditions may restrict $H(x)$ (resp. $H'(x)$) below (resp. above) some value (we denote $|h|=H$ and $|h'|=H'$ throughout the proof); stochastic domination for such conditional laws follows from that of the unconditional ones; thus little attention will be paid to the boundary conditions.

Let us denote by $M$ the maximum of $\chi'$ on the neighbours $N_x$ of $x$, and by $m$ the minimum of $\chi$; thus $H(x) \leq m+2$ and $H'(x) \geq M-2$ . If $M-m \geq 4$, we thus automatically have $H'(x) \geq H(x)$. We hence assume $M-m \leq 2$ below; note that then the values of both $\chi$ and $\chi'$ vary by at most $2$ on $N_x$.

\textbf{Case 1:} Suppose first that $m \geq 3$. This means that the sign of $h$ (and $h'$ similarly) is constant on $N_x \cup \{ x \}$ and, in particular, already the absolute value $H$ reveals, for any $y \in N_x$, whether $h(x)=h(y)$ or not (and for $H'$ and $h'$ similarly). % The proof splits into three subcases.

 \textbf{Case 1a:} Suppose first that $\chi'-\chi$ is a constant $0$ or $2$ on $N_x$. Then $h(x)$ behaves identically to $h'(x)$ or $h'(x)- 2$, respectively, and~\eqref{eq:Holley} follows.
 
 \textbf{Case 1b:} Suppose then that $\chi-\chi'$ is not constant on $N_x$, and that both $\chi$ and $\chi'$ take both values $m$ and $M=m+2$ on $N_x$. Then it necessarily occurs that $H(x), H'(x) \in \{m, M\}$, and it suffices to show that
\begin{align*}
\frac{\mu[ H(x) = M \; \vert \; H_{\vert D \setminus \{x\}} = \chi ]}{ \mu[ H(x) = m \; \vert \; H_{\vert D \setminus \{x\}} = \chi ]}
            \leq \frac{\mu'[ H'(x) = M \; \vert \; H'_{\vert D \setminus \{x\}} = \chi' ]}{ \mu'[ H'(x) = m  \; \vert \; H'_{\vert D \setminus \{x\}} = \chi' ]}.
\end{align*}
We compute
\begin{align*}
\frac{\mu[ H(x) = M \; \vert \; H_{\vert D \setminus \{x\}} = \chi ]}{ \mu[ H(x) = m \; \vert \; H_{\vert D \setminus \{x\}} = \chi ]} = 
\frac{ \prod_{y \in N_x: \chi(y)=M} \mathbf{c}_{\langle x, y \rangle }}{ \prod_{y \in N_x: \chi(y)=m} \mathbf{c}_{\langle x, y \rangle }},
\end{align*}
and using a similar formula for the ``primed case'',
\begin{align*}
\frac{\mu'[ H'(x) = M \; \vert \; H'_{\vert D \setminus \{x\}} = \chi' ]}{ \mu'[ H'(x) = m  \; \vert \; H'_{\vert D \setminus \{x\}} = \chi' ] }
\Bigg/
\frac{\mu[ H(x) = M \; \vert \; H_{\vert D \setminus \{x\}} = \chi ]}{ \mu[ H(x) = m \; \vert \; H_{\vert D \setminus \{x\}} = \chi ]}
= \prod_{y \in N_x: \chi'(y)=M \; \& \; \chi(y)=m} \mathbf{c}_{\langle x, y \rangle }^2 \geq 1.
\end{align*}
Thus, \eqref{eq:Holley} follows. 

\textbf{Case 1c:} The remaining subcase is when $\chi-\chi'$ is not constant on $N_x$, and exactly one out of $\chi$ and $\chi'$ is a constant function on $N_x$. For definiteness, let us study the case $\chi'_{| N_x} = M$, $\chi_{| N_x} \in \{ M, M-2 \}$ here (the other one is similar). $H'(x)$ may thus take any of the values $\{ M-2, M, M+2 \}$, while $H(x)$ is either $ M-2$ or $ M $. It suffices to show that
\begin{align*}
\mu[ H(x) = M \; \vert \; H_{\vert D \setminus \{x\}} = \chi ]
\leq \mu'[ H'(x) \geq M \; \vert \; h_{\vert D \setminus \{x\}} = \chi' ].
\end{align*}
The LHS becomes
\begin{align*}
\mu[ H(x) = M \; \vert \; H_{\vert D \setminus \{x\}} = \chi ] 
&= \frac{
\prod_{y \in N_x: \chi(y)=M} \mathbf{c}_{\langle x, y \rangle }
}{ 
\prod_{y \in N_x: \chi(y)=M} \mathbf{c}_{\langle x, y \rangle } + \prod_{y \in N_x: \chi(y)=m} \mathbf{c}_{\langle x, y \rangle }
}
\\
&=
 \frac{
\prod_{y \in N_x: \chi(y)=M} \mathbf{c}_{\langle x, y \rangle }/\prod_{y \in N_x: \chi(y)=m} \mathbf{c}_{\langle x, y \rangle }
}{ 
 1+ \prod_{y \in N_x: \chi(y)=M} \mathbf{c}_{\langle x, y \rangle } /\prod_{y \in N_x: \chi(y)=m} \mathbf{c}_{\langle x, y \rangle }
}
\end{align*}
while the RHS becomes 
\begin{align*}
\mu'[ H'(x) \geq M \; \vert \; h_{\vert D \setminus \{x\}} = \chi' ] = 
\frac{
1 + \prod_{y \in N_x} \mathbf{c}_{\langle x, y \rangle }
}{ 
1+1 + \prod_{y \in N_x} \mathbf{c}_{\langle x, y \rangle }
}.
\end{align*}
Both are of the form $u/(1+u)$ which is increasing in $u \in \bbR_+ $. We then deduce that the RHS is larger as $\mathbf{c}_e \geq 1$.

\textbf{Case 2:} Suppose now that $m = 1$; thus, $H(x)$ is either $1$ or $ 3$.

\textbf{Case 2a:} Suppose also that $\chi'_{| N_x} $ is a constant $3$. Thus, $H'(x)$ may thus take any of the values $\{ 1, 3, 5 \}$.
It suffices to show that
\begin{align*}
\mu[ H(x) = 3 \; \vert \; H_{\vert D \setminus \{x\}} = \chi ]
\leq \mu'[ H'(x) \geq 3 \; \vert \; h_{\vert D \setminus \{x\}} = \chi' ].
\end{align*}
To compute the left-hand side we observe that, $H(x)=3$ can only occur if $\sign (h)$ is constant on $N_x$.\footnote{If $x \in S$, then this sign must be a constant $+1$; the induced changes are immaterial.} Thus,
\begin{align*}
\mu[ H(x) = 3 \; \vert \; H_{\vert D \setminus \{x\}} = \chi ]
&=
\bbP^{\xi} \big[ \sign (h) \text{ cst. on } N_x \; \vert \; | h_{\vert D \setminus \{x\}} | = \chi \big]
\\ & \qquad \times
\bbP^{\xi} \big[ |h(x)| = 3 \; \vert \; \sign (h) \text{ cst. on } N_x  \; \& \; | h_{\vert D \setminus \{x\}} | = \chi \big] \\
& \leq \bbP^{\xi} \big[ |h(x)| = 3 \; \vert \; \sign (h) \text{ cst. on } N_x  \; \& \; | h_{\vert D \setminus \{x\}} | = \chi \big] \\
& = \begin{cases}
\frac{
1 
}{ 
1+1 + \prod_{y \in N_x} \mathbf{c}_{\langle x, y \rangle }
}, \qquad  \text{ if } h_{|N_x} \text{ constant } \pm 1  \\
\frac{
\prod_{y \in N_x: \chi(y)=3} \mathbf{c}_{\langle x, y \rangle } 
}{ 
\prod_{y \in N_x: \chi(y)=3} \mathbf{c}_{\langle x, y \rangle } + \prod_{y \in N_x: \chi(y)=1} \mathbf{c}_{\langle x, y \rangle }
}, \qquad \text{ if } h_{|N_x} \text{ takes two values } \pm \{ 1, 3\}
\end{cases} \\
& \leq
\frac{ \prod_{y \in N_x: \chi(y)=3} \mathbf{c}_{\langle x, y \rangle } / \prod_{y \in N_x: \chi(y)=1} \mathbf{c}_{\langle x, y \rangle }  }{1 + \prod_{y \in N_x: \chi(y)=3} \mathbf{c}_{\langle x, y \rangle } / \prod_{y \in N_x: \chi(y)=1} \mathbf{c}_{\langle x, y \rangle } }.
\end{align*}
To compute the right-hand side, observe that the sign of $h'$ on $N_x \cup \{ x \}$ is constant; thus, the RHS can be computed identically the right-hand side in case 1c,
\begin{align*}
\mu'[ H'(x) \geq 3 \; \vert \; h_{\vert D \setminus \{x\}} = \chi' ] = 
\frac{
1 + \prod_{y \in N_x} \mathbf{c}_{\langle x, y \rangle }
}{ 
1+1 + \prod_{y \in N_x} \mathbf{c}_{\langle x, y \rangle }
}.
\end{align*}
Our estimates for the left-hand side and right-hand side are both of the form $u/(1+u)$, the right-hand side corresponding to a larger $u$. 

\subsubsection{The nontrivial case}
\label{subsubsec:case 2b}

\textbf{Case 2b:} Suppose now that $m=1$ and $\chi' $ takes the value $1$ somewhere on $N_x$. Thus, $H(x)$ and $H'(x)$ must thus both take values in $\{ 1, 3 \}$,
 and it suffices to show that
\begin{align*}
\frac{\mu[ H(x) = 3 \; \vert \; H_{\vert D \setminus \{x\}} = \chi ]}{\mu[ H(x) = 1 \; \vert \; H_{\vert D \setminus \{x\}} = \chi ]}
\leq
\frac{\mu'[ H'(x) =  3 \; \vert \; H'_{\vert D \setminus \{x\}} = \chi' ]}{
\mu'[ H'(x) =  1 \; \vert \; H'_{\vert D \setminus \{x\}} = \chi' ]
}
.
\end{align*}

Let $H_1 \in \heightfcns_D$ (resp. $H_3 \in \heightfcns_D$) be the height function equal to $1$ (resp. $3$) at $x$ and coinciding with $\chi$ on $D \setminus\{x\}$. Define
\begin{align}\label{eq:def of Z}
    Z_1 = \sum_{\substack{ h \in \heightfcns_D \\ | h | =H_1 \\ h > 0 \text{ on } \calS_1 } } W (h)
    \text{\qquad and \qquad}
    Z_3 = \sum_{\substack{ h \in \heightfcns_D \\ | h | = H_3 \\ h > 0 \text{ on } \calS_3 } } W (h),
\end{align}
where $\calS_1 \subset S$ (resp. $\calS_3 \subset S$ ) is the set of vertices where the boundary condition $\xi$ reveals the sign of $h$ with $|h| = H_1$ (resp. $|h|=H_3$) to be $+1$.
A similar formula is obtained for the ``primed'' configurations. 
We thus wish to show that 
\begin{align}\label{eq:to show for Holley 2}
	Z_3 \big/ Z_1 \leq  Z'_3 \big/ Z'_1.
\end{align}

Suppose first that $x \in D \setminus S'$, and hence $\calS_1 = \calS_3 =: \calS$ and $\calS_1' = \calS_3' =: \calS'$. There is an injection $\mathbf T$ from the height functions $h$ contributing to $Z_3$ to the height functions $h$ contributing to $Z_1$: simply change the value $\pm 3$ of $h(x)$ to $ \pm 1$. The image of this injection is exactly those $h$ contributing to $Z_1$ for which in addition $h$ has a constant sign on $N_x \cup \{ x\}  $. The weights transform under this injection as
\begin{align*}
W (\mathbf T(h)) = \frac{
\prod_{y \in N_x \; : \; \chi(y)=1 } \mathbf{c}_{\langle x, y \rangle }
}{
\prod_{y \in N_x  \; : \; \chi(y)=3 } \mathbf{c}_{\langle x, y \rangle }
} W (h).
\end{align*}
We can thus express $Z_3$ using this up-to-constant weight-preserving injection as
\begin{align*}
Z_3 = \frac{
\prod_{y \in N_x \; : \; \chi(y)=3 } \mathbf{c}_{\langle x, y \rangle }
}{
\prod_{y \in N_x  \; : \; \chi(y)=1 } \mathbf{c}_{\langle x, y \rangle }
}
 \sum_{\substack{ h \in \heightfcns_D \\ | h | = H_1 \\ h \geq 0 \text{ on } \calS  \\
\mathrm{sign}(h) \text{ cst. on } N_x \cup \{ x\} } } W (h),
\end{align*}
and finally, using~\eqref{eq:def of Z},
\begin{align}
\label{eq:Z2/Z0 using sign probas}
Z_3 \big/ Z_1 = 
\frac{
\prod_{y \in N_x \; : \; \chi(y)=3 } \mathbf{c}_{\langle x, y \rangle }
}{
\prod_{y \in N_x  \; : \; \chi(y)=1 } \mathbf{c}_{\langle x, y \rangle }
}
\bbP^{\xi } [ \mathrm{sign}(h) \text{ cst. on } N_x \cup \{ x\} \; | \; | h | = H_1 ].
\end{align}
A similar formula holds for the ``primed'' case. Observing the order of the prefactors
\begin{align*}
\frac{
\prod_{y \in N_x \; : \; \chi(y)=3 } \mathbf{c}_{\langle x, y \rangle }
}{
\prod_{y \in N_x  \; : \; \chi(y)=1 } \mathbf{c}_{\langle x, y \rangle }
}
\leq
\frac{
\prod_{y \in N_x \; : \; \chi'(y)=3} \mathbf{c}_{\langle x, y \rangle }
}{
\prod_{y \in N_x  \; : \; \chi'(y)=1 } \mathbf{c}_{\langle x, y \rangle }
}
\end{align*}
it then suffices to show that
\begin{align}
\label{eq:last objective 1}
\bbP^{\xi } [ \mathrm{sign}(h) \text{ cst. on }N_x \cup \{ x\} \; | \; | h | = H_1 ] 
\leq
\bbP^{\xi'} [ \mathrm{sign}(h') \text{ cst. on } N_x \cup \{ x\} \; | \; | h' | = H'_1 ].
\end{align}

If $x \in S$,\footnote{Recall that the remaining case $x \in S' \setminus S$ was handled in the very beginning.} an analogous injection changes the value $h(x)=3$ to $1$, and the image are those $h$ contributing to $Z_1$ for which $h$ has a constant sign $+1$ on $N_x \cup \{ x\}  $, and after analogous computation,
\begin{align*}
Z_3 \big/ Z_1 = 
\frac{
\prod_{y \in N_x \; : \; \chi(y)=3 } \mathbf{c}_{\langle x, y \rangle }
}{
\prod_{y \in N_x  \; : \; \chi(y)=1 } \mathbf{c}_{\langle x, y \rangle }
}
\bbP^{\xi } [ \mathrm{sign}(h)=+1 \text{ on } N_x \cup \{ x\} \; | \; | h | = H_1 ],
\end{align*}
and it suffices to show that
\begin{align}
\label{eq:last objective 2}
\bbP^{\xi } [ \mathrm{sign}(h)=+1 \text{ on }N_x \cup \{ x\} \; | \; | h | = H_1 ] 
\leq
\bbP^{\xi'} [ \mathrm{sign}(h')=+1 \text{ on } N_x \cup \{ x\} \; | \; | h' | = H'_1 ].
\end{align}

As for this the objectives~\eqref{eq:last objective 1} and~\eqref{eq:last objective 2}, recall from Corollary~\ref{cor:FK cond law given abs height} that the signs of $h$ (resp. $h'$) under the above conditional law are constant on the clusters of the percolation $\omega$ (resp. $\omega'$), being $+1$ when so determined by the boundary condition and otherwise given by i.i.d. fair coin flips. The objective then follows since by Lemma~\ref{lem:FK coupling 2}, the percolations can be coupled so that $\omega \subset \omega'$. This concludes case 2b and the entire verification of the Holley criterion.

\section{Some postponed details}
\label{app:roskis}

We explicate here for the sake of completeness some easy proofs that were omitted in the main text.

\begin{proof}[Lemma~\ref{lem:FK cluster conditioning}]
We will show that property (1) gives the correct marginal law for $\varpi$; property (2) then follows immediately.
Under the conditional FK-Ising model $ \sfP^{G}_{\mathrm{FK}} [\cdot \; | \; E_{\mathsf{fix}} \subset \omega]$, each $\omega \supset E_{\mathsf{fix}}$ appears with a probability proportional to 
\begin{align*}
2^{c_G(\omega ) } \prod_{e \in E \setminus E_{\mathsf{fix}}}  \left( (p_e)^{\mathbbm{1}\{ e \in \omega \}} (1-p_e)^{\mathbbm{1}\{ e \not \in \omega \}} \right).
\end{align*}
Since $E_{\mathsf{fix}} \subset \omega$, it follows that $c_G(\omega ) = c_\calG(\omega )$. Re-sampling on $E_{\mathsf{fix}} $ to obtain $\varpi$ will only change the state of loop edges on $\calG$ so $c_\calG(\omega ) = c_\calG(\varpi )$, and the probability of each $\varpi$ is thus proportional to
\begin{align*}
2^{c_\calG (\varpi ) } \prod_{e \in E }  \left( (p_e)^{\mathbbm{1}\{ e \in \varpi \}} (1-p_e)^{\mathbbm{1}\{ e \not \in \varpi \}} \right),
\end{align*}
which is the definition of $\sfP^{\calG}_{\mathrm{FK}} $. The case of wired measures is identical.
\end{proof}

\begin{proof}[of Lemma~\ref{lem:SMP with percolation}]
Observe first that the above-mentioned conditional law  of $(h_{|V}, \omega_{|E}, B_{|E})$ and $\bbP^{\alpha }_D$ have the same support.
The rest is a simple computation with weights:
\begin{align*}
\bbP^{\alpha }_D [(h_{|V}, \omega_{|E}, B_{|E})] 
= \bbP^{\alpha }_D [(h_{|V},  B_{|E})]
= \frac{1}{Z^\alpha_D}
\prod_{e = \langle v, w \rangle \in E} \mathbf{c}_e^{ \mathbb{I} \{ h(v) = h(w) \} }  (1-1/\mathbf{c}_e)^{\mathbbm{1}\{B_e=1\}} (1/\mathbf{c}_e)^{\mathbbm{1}\{B_e=0\}},
\end{align*}
where $Z^\alpha_D$ is the random Lipschitz partition function,
and
\begin{align*}
\bbP^{\xi }_{D'} & 
[(h_{|V^c}, \omega_{| E (V^c)}, B_{| E (V^c)} ) \text{ and } \omega_{|E(V, V^c)} = 0  \text{ and } (h_{|V}, \omega_{|E}, B_{|E})] 
\\
& =
\bbP^{\xi } [(h_{|V^c}, B_{| E (V^c)} )  \text{ and } \omega_{|E(V, V^c)} = 0  \text{ and } (h_{|V}, B_{|E})]
 \\
& =  \frac{1}{Z^\xi_{D'}}
\underbrace{\Big( \prod_{e = \langle v, w \rangle \in E(V^c)} \mathbf{c}_e^{ \mathbb{I} \{ h(v) = h(w) \} }  (1-1/\mathbf{c}_e)^{\mathbbm{1}\{B_e=1\}} (1/\mathbf{c}_e)^{\mathbbm{1}\{B_e=0\}} \Big) }
_{=:f( h_{|V^c}, B_{| E (V^c)} )>0} 
\\
& \qquad \times
\Big( \prod_{e = \langle v, w \rangle \in E(V, V^c)} \mathbf{c}_e^{ \mathbb{I} \{ h(v) = h(w) = \pm 1 \} } (1/\mathbf{c}_e)^{ \mathbb{I} \{ h(v) = h(w) = \pm 1 \} } \Big)
\\
& \qquad \times 
\Big( \prod_{e = \langle v, w \rangle \in E} \mathbf{c}_e^{ \mathbb{I} \{ h(v) = h(w) \} }  (1-1/\mathbf{c}_e)^{\mathbbm{1}\{B_e=1\}} (1/\mathbf{c}_e)^{\mathbbm{1}\{B_e=0\}} \Big)
\\
& = \frac{Z^\alpha_D }{Z^\xi_{D'}}
f(h_{|V^c},  B_{| E (V^c)})
\bbP^{\alpha }_D [(h_{|A}, \omega_{|E(A)}, B_{|E(A)})].
\end{align*}
\end{proof}

\begin{proof}[of Proposition~\ref{prop:pos assoc of percolated height fcns}]
% (assuming Propositions~\ref{prop: CBC and FKG},~\ref{prop:|h|-monotonicity} and~\ref{prop:SMP ineq for abs val})]
We only explicate the case (2) when $F$ and $G$ are increasing functions in $(|h|, \omega)$. We start with the analogue of~\eqref{eq:CBC-|h|} in Proposition~\ref{prop:|h|-monotonicity}. Recall from Corollary~\ref{cor:FK cond law given abs height} that the conditional law of $\omega$ given $\{ |h|=H \}$ is the FK model $\sfP^{D}_{\mathrm{FK}} [\cdot \; | \; E_{\mathsf{fix}}(H) \subset \omega]$ if $S=\emptyset$, or $\sfP^{D,+}_{\mathrm{FK}} [\cdot \; | \; E_{\mathsf{fix}}(H) \subset \omega]$ if $S \neq \emptyset$. Let us denote either one by $\sfP^{D,*}_{\mathrm{FK}} [\cdot \; | \; E_{\mathsf{fix}}(H) \subset \omega]$; thus
\begin{align*}
\bbE^\xi [F( |h|, \omega )]
=
\sum_H \bbP^\xi [|h|=H] \sfE^{D,*}_{\mathrm{FK}} [F(H, \omega) \; | \; E_{\mathsf{fix}}(H) \subset \omega].
\end{align*}
Then setting
\begin{align*}
\sfE^{D,*}_{\mathrm{FK}} [F(H, \omega) \; | \; E_{\mathsf{fix}}(H) \subset \omega] =: \tilde{F}(H)
\end{align*}
one observes, by Lemma~\ref{lem:FK coupling 2}, that $\tilde{F}$ is increasing in $H$ (as well as in $D$) and $S$. Thus, using~\eqref{eq:CBC-|h|},
\begin{align*}
\bbE^\xi [F( |h|, \omega )]
= \bbE^\xi [\tilde{F}( |h|)]
 \leq \bbE^{\xi'} [\tilde{F}( |h|)] = \bbE^{\xi'} [F( |h|, \omega)].
\end{align*}
This proves the analogue of~\eqref{eq:CBC-|h|}. The analogue of~\eqref{eq:bdary cut} is derived identically.

It thus remains to prove the analogue of~\eqref{eq:FKG-|h|}. Denoting $\bbP^\xi = \bbP$ and compressing conditional law arguments into conditional expectation notations, we compute
\begin{align*}
\bbE [F( |h|, \omega )G(|h|, \omega)] &=  
 \bbE [\bbE [ F(|h|, \omega)G(|h|, \omega) | \sigma(|h|)] ]  
\end{align*}
where, by the FKG inequality for the conditional FK law of $\omega$ given $|h|$,
\begin{align*}
\bbE [ F(|h|, \omega)G(|h|, \omega) | \sigma(|h|)]
\geq \bbE [ F(|h|, \omega) | \sigma(|h|)] \bbE [ G(|h|, \omega) | \sigma( |h|)].
\end{align*}
Then, as motivated above, $\tilde{F}(|h|) := \bbE [ F(|h|, \omega) | \sigma(|h|)]$ and $\tilde{G}(|h|) := \bbE [ G(|h|, \omega) | \sigma(|h|)]$, are increasing in $|h|$. Hence, by~\eqref{eq:FKG-|h|}
\begin{align*}
\bbE [F( |h|, \omega )G(|h|, \omega)]
\geq \bbE [\tilde{F}(|h|) \tilde{G}(|h|) ]  \geq \bbE [\tilde{F}(|h|) ] \bbE [ \tilde{G}(|h|) ] = \bbE^{\xi} [F( |h|, \omega )] \bbE [G(|h|, \omega)].
\end{align*}
This proves the analogue of~\eqref{eq:FKG-|h|}.

The cases when the percolation process is the Bernoulli one $B$ can be proven with a similar computation which however is even simpler since the percolation process is in that case independent of the height function.
\end{proof}

\paragraph{Proof of Lemma~\ref{lem:narrow quad cross 2}} 

Before proving Lemma~\ref{lem:narrow quad cross 2}, we need the following generalization of Lemma~\ref{lem:sym quad cross}. (Letting $Q=D$ and $\calA = \{V_\mathsf{top}(Q) \stackrel{h \omega \geq 1}{\longleftrightarrow} V_\mathsf{bottom}(Q) \}$ below indeed returns Lemma~\ref{lem:sym quad cross}.)

\begin{lemma}
\label{lem:sym quad cross 2}
Let $Q \subset D$ be two quads, both symmetric with respect to $\tau$, and let $\xi$ be an interval-valued boundary condition on $\partial D$, such that $2-\xi \succeq \tau. \xi$. Let $\calA \subset \{V_\mathsf{top}(Q) \stackrel{h \omega \geq 1 \; \mathrm{in} \; Q}{\longleftrightarrow} V_\mathsf{bottom}(Q) \} $ be increasing in $(h, B)$. Then,
\begin{align*}
\bbP^\xi_D [\calA] \leq 1/2.
\end{align*}
\end{lemma}

\begin{proof}
For streamlined notation, denote throughout the proof $V_\mathsf{[side]}(Q) =: V_\mathsf{[side]}$ and $\stackrel{\calE \; \mathrm{in} \; Q}{\longleftrightarrow} =: \stackrel{\calE }{\longleftrightarrow}$ for $\calE \subset E$, and similarly for dual vertices and crossings.
Note that $\calA^c \supset \{ V_\mathsf{left} \stackrel{h \omega \leq 0 }{\longleftrightarrow}_* V_\mathsf{right}^* \} \supset \{ V_\mathsf{left} \stackrel{h \omega \leq 0 }{\longleftrightarrow} V_\mathsf{right} \} $ by~\eqref{eq:quad cross duality}--\eqref{eq:quad cross dual vs primal}. Combining with the symmetry,
\begin{align*}
1- \bbP^{\xi}_D [ \calA ]
 \geq
 \bbP^{\xi}_D [  V_\mathsf{left} \stackrel{h \omega \leq 0 }{\longleftrightarrow} V_\mathsf{right}  ] 
=
 \bbP^{\tau.\xi}_D [   V_\mathsf{top} \stackrel{h \omega \leq 0 }{\longleftrightarrow} V_\mathsf{bottom}  ] .
\end{align*}
The event in the last step is decreasing in $(h, B)$, so by the FKG,
\begin{align*}
 \bbP^{\tau.\xi}_D [   V_\mathsf{top} \stackrel{h \omega \leq 0 }{\longleftrightarrow} V_\mathsf{bottom}  ] \geq  \bbP^{ 2- \xi}_D [   V_\mathsf{top} \stackrel{h \omega \leq 0}{\longleftrightarrow} V_\mathsf{bottom}  ].
\end{align*}

Next, define a coupling $\bbP$ of $({h}, {B}, {\omega}) \sim \bbP^{2-\xi}_D$ to $(\tilde{h}, \tilde{B}, \tilde{\omega}) \sim \bbP^{\xi}_D$ as in the proof of Lemma~\ref{lem:sym quad cross}; in this coupling
$
\{   V_\mathsf{top} \stackrel{h \omega \leq 0 }{\longleftrightarrow} V_\mathsf{bottom} \} \supset \{  V_\mathsf{top} \stackrel{\tilde{h} \tilde{\omega} \geq 1 }{\longleftrightarrow} V_\mathsf{bottom} \},
$
and the marginal distribution of the coupling thus yield
\begin{align*}
\bbP^{ 2- \xi}_D [   V_\mathsf{top} \stackrel{h \omega \leq 0 }{\longleftrightarrow} V_\mathsf{bottom}  ]
= 
\bbP [  V_\mathsf{top} \stackrel{{h} {\omega} \leq 0 }{\longleftrightarrow} V_\mathsf{bottom}  ]
\geq 
\bbP [   V_\mathsf{top} \stackrel{\tilde{h} \tilde{\omega} \geq 1}{\longleftrightarrow} V_\mathsf{bottom}  ]
=
\bbP_D^\xi [ \calA ].
\end{align*}
Combining the displayed inequalities now proves the claim.
\end{proof}

\begin{proof}[Lemma~\ref{lem:narrow quad cross 2}]
The event $\{ V_\mathsf{left}^* (D) \stackrel{\calE(h, B)^* \cap C}{\longleftrightarrow}_* V_\mathsf{right}^* (D) \}$ is decreasing in $(h, B)$ so it suffices to consider the maximal $\xi$ allowed by the statement, i.e., the maximal extension of $\{ \pm 1 \}$ on $V_{\mathsf{left}}(D) \cup V_{\mathsf{right}}(D)$ to entire $\partial D$ with $\xi \preceq \{ 5, 7 \}$. Also, denote by $\calE'(h, B) = \calE(h, B)^c$ the set of edges between two heights $h \geq 3$ with additionally $B_e=1$ on edges $e$ between two heights $h=3$.

Recall that the left and right walls $\gamma^*_{\mathsf{left}}$ and $\gamma^*_{\mathsf{right}}$  of $C$ are contained in $V_\mathsf{left}^* (D) $ and $ V_\mathsf{right}^* (D)$, respectively; hence
$$
\{ \gamma_\mathsf{left}^* \stackrel{\calE(h, B)^* \cap C}{\longleftrightarrow}_* \gamma_\mathsf{right}^* \}
 \subset
\{ V_\mathsf{left}^* (D) \stackrel{\calE(h, B)^* \cap C}{\longleftrightarrow}_* V_\mathsf{right}^* (D) \}.
$$
Recall also that any vertex on $\partial C$ belongs to $V_{\mathsf{left}}(D)$, $V_{\mathsf{right}}(D)$, $V_{\mathsf{top}}(Q)$ or $V_{\mathsf{bottom}}(Q)$, and $\xi = \{ \pm 1\}$ on the two former ones. In particular, using~\eqref{eq:quad cross duality} (with quad structure for $C$ induced by the left and right walls) and the boundary condition, we obtain
$$
\{ \gamma_\mathsf{left}^* \stackrel{\calE(h, B)^* \cap C}{\longleftrightarrow}_* \gamma_\mathsf{right}^* \}^c 
\subset \{ V_\mathsf{top} (Q) \stackrel{\calE'(h, B) \cap C}{\longleftrightarrow} V_\mathsf{bottom} (Q) \},
$$  
and we thus have
\begin{align*}
\bbP^\xi_D [V_\mathsf{left}^* (D) \stackrel{\calE(h, B)^* \cap C}{\longleftrightarrow}_* V_\mathsf{right}^* (D)]
& 
\geq
1-
\bbP^\xi_D [V_\mathsf{top} (Q) \stackrel{\calE'(h, B) \cap C}{\longleftrightarrow} V_\mathsf{bottom} (Q) ].
\end{align*}
Then, denoting by $\calE'(|h|, B)$ the edges between absolute-heights $|h| \geq 3$ with additionally $B_e=1$ for edges between two absolute-heights $3$ (this is an increasing set in $(|h|, B)$), inclusion gives
\begin{align*}
\bbP^\xi_D [V_\mathsf{top} (Q) \stackrel{\calE'(h, B) \cap C}{\longleftrightarrow} V_\mathsf{bottom} (Q) ]
\leq
\bbP^\xi_D [V_\mathsf{top} (Q) \stackrel{\calE'( |h|, B) \cap C}{\longleftrightarrow} V_\mathsf{bottom} (Q) ].
\end{align*}

We now aim to use the FKG for $(|h|, B)$, with a domain-gluing argument analogous to the proof of Lemma~\ref{lem:narrow quad cross}. Since $C$ is on the bottom of $D$, we have $D_{\mathsf{top}}= D \setminus C$. Set
\begin{align*}
D' = Q \bigcupplus D_{\mathsf{top}} \bigcupplus \tau ( D_{\mathsf{top}} )
\end{align*} 
so that $D \subset D'$. Again, $D'$ is a ``discrete Riemann surface'', and a $\tau$-symmetric quad, with same corner faces as $Q$. Note that $V_{\mathsf{top}}(D) \subset D_{\mathsf{top}}$, that $\tau ( D_{\mathsf{top}} )$ is glued on either left or right side of $Q$. In particular, $V_{\mathsf{top}}(D)$ and $V_{\mathsf{bottom}}(Q)$ (as well as $V_{\mathsf{bottom}}(D) \subset V_{\mathsf{bottom}}(Q)$) remain boundary vertices of $D'$. Let $\xi'$ be the boundary condition on $\partial D'$ which is $\{ \pm 1 \}$ outside of $V_{\mathsf{top}}(D)$ and  $V_{\mathsf{bottom}}(D)$ and on them the largest extension $\preceq \{ 5, 7 \}$. By the FKG,
\begin{align*}
\bbP^\xi_D [V_\mathsf{top} (Q) \stackrel{\calE'( |h|, B) \cap C}{\longleftrightarrow} V_\mathsf{bottom} (Q) ] 
&
\leq 
\bbP^{\xi'}_{D'} [V_\mathsf{top} (Q) \stackrel{\calE'( |h|, B) \cap C}{\longleftrightarrow} V_\mathsf{bottom} (Q) ]
\\
& = 
\bbP^{\xi'}_{D'} [V_\mathsf{top} (Q) \stackrel{\calE'( h, B) \cap C}{\longleftrightarrow} V_\mathsf{bottom} (Q) ],
\end{align*}
where the last step used the fact that $C$ lies on the bottom of $D$, and hence the boundary condition $\xi$ determines the sign of $h$ on the crossings. Setting $\tilde{h}= h-2$ and defining $\tilde{\omega}$ via $\tilde{h}$ and $B$ in the usual manner, one obtains $\calE'( h, B) = \{ \tilde{h} \tilde{\omega}\geq 1 \}$, and an application of Lemma~\ref{lem:sym quad cross 2} yields
\begin{align*}
\bbP^{\xi'}_{D'} [V_\mathsf{top} (Q) \stackrel{\calE'( h, B) \cap C}{\longleftrightarrow} V_\mathsf{bottom} (Q) ] \leq 1/2.
\end{align*}
Combining all the displayed non-set inequalities concludes the proof.
\end{proof}

\begin{proof}[Lemma~\ref{lem:find a crossing}]
Let $\gamma_1$, $\gamma_2$, and $\gamma_3$ be any vertical primal-crossings of $R$ inducing respective events $ \calE^+_1 \cap \calE^-_2 \cap \calE^+_3$ (recall that these paths are vertex-disjoint). We first claim that there is a vertical dual-crossing $\gamma_2^*$ of $R$ with $\nu = 0$ between $\gamma_1$ and $\gamma_3$, and landing on $I^2 \cup I^3$ and $L^2 \cup L^3$ on the bottom and top, respectively. To see this, note first that on $\gamma_1$ and $\gamma_3$, $\tilde{h} \leq 5$ with additionally $\nu_e= 1$ on edges between two heights $5$. Let $\gamma_3'$ be the vertical crossing of $R$ on which $ | \tilde{h} | \leq 5$ and additionally $\nu_e= 1$ on edges between two absolute-heights $| \tilde{h} |  = 5$, and which lies between $\gamma_2$ and $\gamma_3$, obtained by modifying $\gamma_3$ to wind around any primal-edges between two heights $\tilde{h} \leq -5$ in $R$ along the closest such path on the side of $\gamma_2$ (such exists as the dual-degree is at most $11$). Note that $\nu = 1$ on $\gamma_3'$. Let $\gamma_2^*$ be the vertical dual-crossing of $R$ on which $\nu_e= 0$ and which lies between $\gamma_2$ and $\gamma_3'$, obtained as the left boundary of the primal-cluster of $\nu =1$ containing $\gamma_3'$. (Indeed, that cluster cannot contain any edge of $\gamma_2$, on which $\tilde{h} \geq 5$ with additionally $\nu_e= 0$ on edges between two heights $5$.) This shows $\gamma_2^*$ satisfies the claimed properties.

We note at this point that modifying analogously $\gamma_1$ to a $\nu = 1$ path $\gamma_1'$ shows that no $\nu = 0$ dual-path between $ I^2 \cup I^3$ and $L^2 \cup L^3$ can pass left of $\gamma_1'$, i.e., the concept of \textit{left-most path} in the statement makes sense, and such left-most path is between $\gamma_1$ and $\gamma_3$.

Next, we claim that \textit{any} vertical dual-crossing $\gamma^*$ of $R$ with $\nu = 0$ lying between $\gamma_1$ and $\gamma_3$ induces $\calV^{i_0, j_0, j_0, i_0 \; *}_{\alpha_0, \beta_0, \gamma_0} (\nu = 0)$. Indeed, such $\gamma^*$ clearly traverses between $ I_{i_0}$ and $\subset L_{i_0}$. Also, the unique sub-crossing of $R_{\mathsf{bot}}$ (resp. $R_{\mathsf{top}}$) is topologically forced to be of the type $\alpha_0$ (resp. $\gamma_0$) by the two curves $\gamma_1$ and $\gamma_3$. It remains to motivate the appearance of $j_0$, $j_0$ and $\beta_0$ in $\calV^{i_0, j_0, j_0, i_0 }_{\alpha_0, \beta_0, \gamma_0} (h \omega \leq 0)$. Let $\gamma^{\mathsf{tr}}_1$ and $\gamma^{\mathsf{tr}}_3$ be the truncations of $\gamma_1$ and $\gamma_3$, respectively, from $I_{i_0}$ on the bottom of $R$ up to the first hitting of $R_{\mathsf{top}}$, hence on $K_{j_0}$. Note that $\gamma^*$ is topologically forced to exit the domain constrained by $\gamma^{\mathsf{tr}}_1$ and $\gamma^{\mathsf{tr}}_3$, $I_{i_0}$ and $K_{j_0}$, and this can only happen via $K_{j_0}$. This is also when $\gamma^*$ reaches $R_{\mathsf{top}}$ for the first time, motivating one $j_0$. This reaching of $R_{\mathsf{top}}$ occurs in the domain constrained by $\eta_1$, $\eta_3$ (related to $\gamma_1$, $\gamma_3$), $J_{j_0}$ and $K_{j_0}$; in particular, the last hitting of $R_{\mathsf{bot}}$ by $\gamma^*$ before that is topologically forced to occur on $J_{j_0}$; the type $\beta_0$ is then topologically forced by $\eta_1$ and $\eta_3$. This proves that $\gamma^*$ induces $\calV^{i_0, j_0, j_0, i_0 \; *}_{\alpha_0, \beta_0, \gamma_0} (\nu = 0)$.

We list the conclusions of the above three paragraphs: (i) there is a vertical dual-crossing of $R$ with $\nu = 0$, landing on $I^2 \cup I^3$ and $L^2 \cup L^3$ on the bottom and top, respectively; (ii) there exists a left-most such path, lying between $\gamma_1$ and $\gamma_3$; (iii) any vertical dual-crossing of $R$ with $\nu = 0$ between $\gamma_1$ and $\gamma_3$  induces $\calV^{i_0, j_0, j_0, i_0 \; *}_{\alpha_0, \beta_0, \gamma_0} (\nu = 0)$. The claim follows.
\end{proof}

\bibliographystyle{abbrv}
\def\cprime{$'$}

\end{document}